\documentclass[11pt,a4paper,twoside]{amsart}
\usepackage[top=1in, bottom=1.25in, marginparwidth=1in, inner=1in, outer=1in]{geometry}

\usepackage{lmodern}
\usepackage[T2A,T1]{fontenc} % T2A for Cyrillic letters
\usepackage[utf8]{inputenc}
\usepackage[USenglish]{babel}
\usepackage{yfonts}
\usepackage{booktabs}% for tables
\usepackage{afterpage}
\usepackage{pinlabel}

\usepackage{graphicx}
\usepackage{tikz}
\usepackage{tikz-cd}
\usetikzlibrary{cd,decorations.pathreplacing}
\usetikzlibrary{shapes,decorations}%,arrows,arrows.meta}
\tikzcdset{arrow style=tikz, diagrams={>=latex}}
\tikzset{>=latex}
\tikzstyle{tight}=[font=\scriptsize, inner sep=1pt, outer sep=1pt]

\usepackage{float}
\usepackage[labelfont=bf]{caption}%Claudius: "font=footnotesize" gives warnings "Cannot determine current font size"
\DeclareCaptionFormat{myformat}{\fontsize{9}{10}\selectfont#1#2#3}
\usepackage[labelfont=up,margin=5pt]{subcaption}
\usepackage{setspace} 
\usepackage[hidelinks]{hyperref}
\hypersetup{%
	pdfencoding=auto, %
	bookmarksdepth=3%
}
\usepackage{bookmark}

\usepackage{amsmath}
\usepackage{nicefrac}
\usepackage[makeroom]{cancel}
\usepackage{amsfonts}
\usepackage{stmaryrd}
\usepackage{amssymb}
\usepackage{amsthm}
\usepackage{mathtools}
\usepackage{mathrsfs}
\usepackage{bm}
\mathtoolsset{mathic=true}
\def\co{\colon\thinspace\relax}% colon spaces

\newtheorem{theorem}{Theorem}[section]
\newtheorem*{theorem*}{Theorem}
\newtheorem{lemma}[theorem]{Lemma}

\newtheorem{corollary}[theorem]{Corollary}
\newtheorem{proposition}[theorem]{Proposition}
\theoremstyle{definition}
\newtheorem{definition}[theorem]{Definition}
\newtheorem{assumptions}[theorem]{Assumptions}
\newtheorem{example}[theorem]{Example}

\newtheorem{remark}[theorem]{Remark}

\makeatletter
\newcommand*{\math@version@bold}{bold}
\DeclareMathOperator\DD{% or \DeclareMathOperator*\DD
	\textrm{%
		\usefont{T2A}{cmr}{\ifx\math@version\math@version@bold bx\else m\fi}{n}%
		\CYRD
	}%
} 
\makeatother

\usepackage{enumitem}
\newcounter{dummy}
\makeatletter
\newcommand\myitem[1][]{%
	\item[\textnormal{(}#1\textnormal{)}]\refstepcounter{dummy}%
	\def\@currentlabel{\textnormal{(}#1\textnormal{)}}%
}
\makeatother

\renewcommand{\geq}{\geqslant}
\renewcommand{\leq}{\leqslant}
\newcommand{\tauRight}{\raisebox{-7pt}{\includegraphics[scale=1.3]{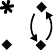}}}
\newcommand{\tauBottom}{\raisebox{-8pt}{\includegraphics[scale=1.3]{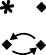}}}
\newcommand{\rightPunctures}{\mathrm{(}\raisebox{-3pt}{\includegraphics[scale=0.7]{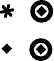}}\mathrm{)}}
\newcommand{\bottomPunctures}{\mathrm{(}\raisebox{-3pt}{\includegraphics[scale=0.7]{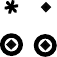}}\mathrm{)}}
\newcommand{\diagonalPunctures}{\mathrm{(}\raisebox{-3pt}{\includegraphics[scale=0.7]{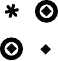}}\mathrm{)}}

\newcommand{\fieldTwoElements}{\mathbb{F}}
\newcommand{\F}{{\mathbb{F}}}

\newcommand{\Z}{\mathbb{Z}}
\newcommand{\ZZ}{\mathbb{Z}/2}

\newcommand{\BNAlgH}{\mathcal{B}}
\newcommand{\A}{\mathcal{A}}
\newcommand{\B}{\mathcal{B}}
\newcommand{\tM}{\widetilde{M}}
\newcommand{\tR}{\widetilde{R}}
\newcommand{\Binf}{\mathcal{B}^\infty}
\newcommand{\BinfU}{\mathcal{B}^*[U]}
\newcommand{\BinfUU}{\mathcal{B}^*[U,U^{-1}]}
\newcommand{\arcHor}{\mathbf{a}^\bullet}
\newcommand{\arcVer}{\mathbf{a}^\circ}
\newcommand{\arcDiag}{\mathbf{a}^{\scriptscriptstyle\blacksquare}}
\newcommand{\AnUU}{\mathcal{A}^*[U,U^{-1}]}

\newcommand{\Ad}{\operatorname{\mathcal{A}}^\partial}
\newcommand{\Aminus}{\operatorname{\mathcal{A}}^-}

\newcommand{\hol}{{\circ}}
\newcommand{\sol}{{\bullet}}

\newcommand{\Tw}{\text{Tw}}
\newcommand{\W}{\mathcal{W}}
\DeclareMathOperator{\Coh}{Coh}

\newcommand{\Spec}{\text{Spec}\:}
\DeclareMathOperator{\Perf}{Perf}
\DeclareMathOperator{\MF}{MF}
\newcommand{\sg}{\mathit{sg}}

\newcommand{\tg}{\tilde{g}}
\DeclareMathOperator{\Cobb}{\Cob_{\bullet}}
\DeclareMathOperator{\Cobl}{\Cob_{/{\mathit{l}}}}

\DeclareMathOperator{\id}{id}

\DeclareMathOperator{\Mod}{Mod}

\DeclareMathOperator{\GL}{GL}

\DeclareMathOperator{\SL}{SL}

\DeclareMathOperator{\Mor}{Mor}
\DeclareMathOperator{\End}{End}

\DeclareMathOperator{\Cob}{Cob}

\newcommand{\Diag}{\mathcal{D}} %for diagrams of links
\DeclareMathOperator{\HFK}{\widehat{HFK}}

\newcommand{\HF}{\operatorname{HF}}

\DeclareMathOperator{\Hast}{H_\ast}

\DeclareMathOperator{\HFT}{HFT}
\newcommand{\CFTminus}{\operatorname{CFT}^-}
\DeclareMathOperator{\CFTd}{CFT^\partial}

\newcommand{\KhTl}[1]{{\llbracket #1 \rrbracket}_{/l}} % [[T]]_{/l}
\newcommand{\KhTb}[1]{{\llbracket #1 \rrbracket}_{\bullet}} % [[T]]_{/l}
\DeclareMathOperator{\Kh}{Kh}% Khovanov homology
\DeclareMathOperator{\BN}{BN}% Bar-Natan homology %
\newcommand{\Khr}{\widetilde{\Kh}}% reduced Khovanov homology
\newcommand{\BNr}{\widetilde{\BN}}% reduced Bar-Natan homology

\newcommand{\rKh}{\mathbf{r}}
\newcommand{\sKh}{\mathbf{s}}

\newcommand{\InfConLift}[1]{\bar{#1}}
\newcommand{\Lgamma}{\InfConLift{\gamma}}

\newif\ifA% true: A-link, false: thin link
\newif\ifi% true: interior
\newif\ifm% true: mirror
\newif\ifd% true: boundary
\newcommand{\thth}[1]{% prototype for thin/A-link fillings
	\operatorname{%
		\ifd%
			\partial%
		\fi%
		\ifA%
			\ifi%
				\mathring{\mathrm{A}}%
			\else%
				\mathrm{A}%
			\fi%
		\else%
			\ifi%
				\mathring{\Theta}%
			\else%
				\Theta%
			\fi%
		\fi%
		\ifm%
			\ifA%
				^{\!\mirror}%
			\else%
				^{\mirror}%
			\fi%
		\fi%
		_{#1}%
	}%
}
\newcommand{\HFsup}{\mathrm{HF}}
\newcommand{\Khsup}{\mathrm{Kh}}
\newcommand{\Gsup}{\mathit{G}}
\newcommand{\makemirror}[2]  {\newcommand{#1}{\mtrue#2\mfalse}}
\newcommand{\makeALink}[2]   {\newcommand{#1}{\Atrue#2\Afalse}}
\newcommand{\makeInterior}[2]{\newcommand{#1}{\itrue#2\ifalse}}
\newcommand{\makeBoundary}[2]{\newcommand{#1}{\dtrue#2\dfalse}}

\newcommand{\Thin}{\thth{}}
\newcommand{\ThinZ}{\thth{\Z}}
\newcommand{\ThinZZ}{\thth{\ZZ}}
\newcommand{\ThinG}{\thth{\Gsup}}
\newcommand{\ThinHF}{\thth{\HFsup}}
\newcommand{\ThinKh}{\thth{\Khsup}}
\makemirror{\mirrorThin}  {\Thin}
\makemirror{\mirrorThinZ} {\ThinZ} 
\makemirror{\mirrorThinZZ}{\ThinZZ} 
\makemirror{\mirrorThinG} {\ThinG} 
\makemirror{\mirrorThinHF}{\ThinHF} 
\makemirror{\mirrorThinKh}{\ThinKh} 

\makeBoundary{\BdryThin}  {\Thin}
\makeBoundary{\BdryThinZ} {\ThinZ} 
\makeBoundary{\BdryThinZZ}{\ThinZZ} 
\makeBoundary{\BdryThinG} {\ThinG} 
\makeBoundary{\BdryThinHF}{\ThinHF} 
\makeBoundary{\BdryThinKh}{\ThinKh} 
\makemirror{\mirrorBdryThin}  {\BdryThin}
\makemirror{\mirrorBdryThinZ} {\BdryThinZ} 
\makemirror{\mirrorBdryThinZZ}{\BdryThinZZ} 
\makemirror{\mirrorBdryThinG} {\BdryThinG} 
\makemirror{\mirrorBdryThinHF}{\BdryThinHF} 
\makemirror{\mirrorBdryThinKh}{\BdryThinKh} 

\makeInterior{\IntThin}  {\Thin}
\makeInterior{\IntThinZ} {\ThinZ} 
\makeInterior{\IntThinZZ}{\ThinZZ} 
\makeInterior{\IntThinG} {\ThinG} 
\makeInterior{\IntThinHF}{\ThinHF} 
\makeInterior{\IntThinKh}{\ThinKh} 
\makemirror{\mirrorIntThin}  {\IntThin} 
\makemirror{\mirrorIntThinZ} {\IntThinZ} 
\makemirror{\mirrorIntThinZZ}{\IntThinZZ} 
\makemirror{\mirrorIntThinG} {\IntThinG} 
\makemirror{\mirrorIntThinHF}{\IntThinHF} 
\makemirror{\mirrorIntThinKh}{\IntThinKh} 
\makeALink{\ALink}  {\Thin}
\makeALink{\ALinkHF}{\ThinHF} 
\makeALink{\ALinkKh}{\ThinKh} 
\makemirror{\mirrorALink}  {\ALink}
\makemirror{\mirrorALinkHF}{\ALinkHF} 
\makemirror{\mirrorALinkKh}{\ALinkKh} 

\makeBoundary{\BdryALink}  {\ALink}
\makeBoundary{\BdryALinkHF}{\ALinkHF} 
\makeBoundary{\BdryALinkKh}{\ALinkKh} 
\makemirror{\mirrorBdryALink}  {\BdryALink}
\makemirror{\mirrorBdryALinkHF}{\BdryALinkHF} 
\makemirror{\mirrorBdryALinkKh}{\BdryALinkKh} 

\makeInterior{\IntALink}  {\ALink}
\makeInterior{\IntALinkHF}{\ALinkHF} 
\makeInterior{\IntALinkKh}{\ALinkKh} 
\makemirror{\mirrorIntALink}  {\IntALink} 
\makemirror{\mirrorIntALinkHF}{\IntALinkHF} 
\makemirror{\mirrorIntALinkKh}{\IntALinkKh} 
\newcommand{\mirror}{\operatorname{m}} % usual mirroring in the plane (swap over and under) 
\newcommand{\FourPuncturedSphere}{S^2_4}
\newcommand{\FourPuncturedSphereKh}{S^2_{4,\ast}}
\newcommand{\ThreePuncturedSphere}{S^2_{3}}

\newcommand{\PuncturedPlane}{\mathbb{R}^2\smallsetminus \Z^2}

\newcommand{\QPI}{\operatorname{\mathbb{Q}P}^1}

\newcommand{\vc}[1]{\vcenter{\hbox{#1}}}%
\newcommand{\mypic}[2]{%
  \newcommand{#2}{%
    \vc{%
      \includegraphics[page=#1]%
      {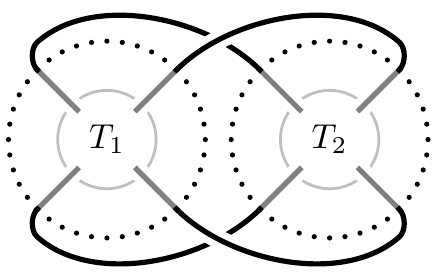}%
    }%
  }%
}%

\mypic{1}{\tanglepairingI}
\mypic{2}{\tanglepairingII}
\mypic{3}{\rigidcurveIntro}
\mypic{4}{\pretzeltangleDownstairs}
\mypic{5}{\pretzeltangleUpstairs}
\mypic{6}{\pretzeltangle}
\mypic{7}{\InlineTrivialTangle}
\mypic{8}{\xI}
\mypic{9}{\xII}
\mypic{10}{\xIII}
\mypic{11}{\xyI}
\mypic{12}{\xyII}
\mypic{13}{\xyIII}
\mypic{14}{\xyIV}
\mypic{15}{\HFTintervalI}
\mypic{16}{\HFTintervalII}
\mypic{17}{\HFTintervalIII}
\mypic{18}{\HFTintervalIV}
\mypic{19}{\HFTtransitivity}
\mypic{20}{\HFTbigon}
\mypic{21}{\Khbigon}
\mypic{22}{\pretzeltangleKh}
\mypic{23}{\pretzeltangleDownstairsKhr}
\mypic{24}{\pretzeltangleUpstairsKhr}
\mypic{25}{\pretzeltangleDownstairsBNr}
\mypic{26}{\pretzeltangleUpstairsBNr}
\mypic{27}{\DotC}
\mypic{28}{\DotB}
\mypic{29}{\DotCblue}
\mypic{30}{\DotBblue}
\mypic{31}{\DotCred}
\mypic{32}{\DotBred}
\mypic{33}{\ConnectivityX}
\mypic{34}{\Li}
\mypic{35}{\Lo}
\mypic{36}{\Ni}
\mypic{37}{\No}
\mypic{40}{\Nil}
\mypic{41}{\Nol}
\mypic{42}{\CrossingL}
\mypic{43}{\CrossingR}
\mypic{50}{\LoDotB}
\mypic{51}{\NiDotL}
\mypic{52}{\NiDotR}
\mypic{53}{\NoDotT}
\mypic{54}{\NoDotB}
\mypic{55}{\closureDiag}
\mypic{56}{\Circle}
\mypic{57}{\CrossingLBasepoint}
\mypic{58}{\CrossingRBasepoint}
\mypic{59}{\PairingTrefoilArcINTRO}
\mypic{60}{\PairingTrefoilLoopINTRO}
\mypic{61}{\KhOneCurve}
\mypic{62}{\KhTwoCurves}
\mypic{63}{\WindingNonSpecialA}
\mypic{64}{\WindingNonSpecialB}
\mypic{65}{\WindingNonSpecialC}
\mypic{66}{\WindingNonSpecialD}
\mypic{67}{\WindingNonSpecialE}
\mypic{68}{\ThinnessObstructionBNr}
\mypic{69}{\FukayaFigureEightArcAlg}
\mypic{70}{\LittleSpecialWrapping}
\mypic{71}{\MuchSpecialWrapping}
\mypic{72}{\GeographyForSpecials}
\mypic{73}{\ExceptionalExample}
\mypic{74}{\tube}
\mypic{75}{\plane}
\mypic{76}{\planedot}
\mypic{77}{\planedotdot}
\mypic{78}{\planedotstar}
\mypic{79}{\SpherePic}
\mypic{80}{\Spheredot}
\mypic{81}{\Spheredotdot}
\mypic{82}{\Spheredotdotdot}
\mypic{83}{\DiscT}
\mypic{84}{\DiscB}
\mypic{85}{\DiscR}
\mypic{86}{\DiscL}
\mypic{87}{\DiscTdot}
\mypic{88}{\DiscBdot}
\mypic{89}{\DiscRdot}
\mypic{90}{\DiscLdot}
\mypic{91}{\GeographyCovering}

\begin{document}
\title{Khovanov multicurves are linear}

\author{Artem Kotelskiy}
\address{Mathematics Department \\ Stony Brook University}
\email{artofkot@gmail.com}

\author{Liam Watson}
\address{Department of Mathematics \\ University of British Columbia}
\email{liam@math.ubc.ca}
\thanks{AK is supported by an AMS-Simons travel grant. LW is supported by an NSERC discovery/accelerator grant and was partially supported by funding from the Simons Foundation and the Centre de Recherches Math\'ematiques, through the Simons-CRM scholar-in-residence program. CZ is supported by the Emmy Noether Programme of the DFG, Project number 412851057.}

\author{Claudius Zibrowius}
\address{Faculty of Mathematics \\ University of Regensburg}
\email{claudius.zibrowius@posteo.net}

\begin{abstract}
%In previous works, we introduced certain multicurve invariants for Conway tangles, namely the Heegaard Floer invariant \(\HFT\) and the Khovanov invariant \(\Khr\). Applying ideas from homological mirror symmetry, we show that the algebraically defined invariant \(\Khr\) is subject to the same strong geography restrictions that were already known for the analytically defined invariant \(\HFT\). 
In previous work we introduced a Khovanov multicurve invariant  \(\Khr\) associated with Conway tangles. Applying ideas from homological mirror symmetry we show that \(\Khr\) is subject to strong geography restrictions: Every component of the invariant is linear, in the sense that it admits a lift to a curve homotopic to a straight line in an appropriate planar cover of the tangle boundary.
\end{abstract}
\maketitle

%!TEX root = ../main.tex
% Intro.tex
\section{Introduction}\label{sec:intro}

%A Conway tangle $T$ is an embedding of two intervals and a finite (possibly empty) set of circles into the three-dimensional ball $B^3$ such that the preimage of the boundary sphere $\partial B^3$ is equal to the four endpoints $\partial T$ of the intervals. 
A Conway tangle is a proper embedding of two intervals and a finite (possibly empty) set of circles into the three-dimensional ball $B^3$.
We consider such embeddings up to ambient isotopy fixing the boundary pointwise. 
%With a Conway tangle, we can associate two objects in the Fukaya category of the four-punctured sphere $\partial B^3\smallsetminus\partial T$, namely the bigraded multicurves
Given a Conway tangle $T$ there are two associated objects in the Fukaya category of the four-punctured sphere $\partial B^3\smallsetminus\partial T$, namely, the bigraded multicurves
\[
\HFT(T)
\quad
\text{and}
\quad
\Khr(T)
\] 
The first invariant was defined by the third author using Heegaard Floer theory \cite{pqMod}; it should be understood as a relative version of knot Floer homology \(\HFK\) % due to Ozváth and Szabó, and independently Rasmussen 
\cite{OSHFK,Jake}.
Similarly, the second invariant should be understood as a relative version of reduced Khovanov homology; it was defined by the authors as a geometric interpretation of Bar-Natan's Khovanov homology for tangles \cite{KWZ}. 
Remarkably, these two multicurve invariants share many formal properties. 
For instance, both satisfy gluing theorems that recover the corresponding link homology theory in terms of the Lagrangian Floer theory \cite{KWZ,pqMod}. 
Also, both invariants detect rational tangles, or more generally, they detect if a tangle is split \cite{LMZ,KLMWZ}.  

The purpose of this paper is to show that the similarity between \(\HFT(T)\) and \(\Khr(T)\) extends to the fundamental structure of the invariants themselves. 
The results that we show in this paper have already been applied to resolve several problems:
they were the key to the proof of an equivariant version of the Cosmetic Surgery Conjecture \cite{KLMWZ}; 
they form the basis of the thin gluing criteria established in~\cite{KWZ-thin}; and
they allow one to give an elementary reproof of the Cosmetic Crossing Conjecture for split links \cite{KLMWZ}, originally due to Joshua Wang \cite{Wang}.
In addition, our results also explain phenomena observed by the second author in the study of the Khovanov homology of strongly invertible knots \cite{Watson2017,KWZ-strong}. 

Setting aside technicalities such as local systems and gradings, we now briefly explain the main structural result established in this paper.
By definition, both \(\HFT(T)\) and \(\Khr(T)\) take the form of collections of immersed curves on the four-punctured sphere \(B^3\smallsetminus\partial T\), considered up to regular isotopy. 
In each case, the immersed curves in question are rather highly structured. To describe this structure, we first pass to a planar cover of $S^2 \smallsetminus 4\text{pt} =\partial B^3 \smallsetminus \partial T$ that factors through the toroidal two-fold cover: % it is not branched if the spaces are punctured
\[
\left(\mathbb R^2 \smallsetminus \mathbb Z^2\right)  
\to 
\left(\mathbb T^2 \smallsetminus 4\text{pt}\right) 
\to 
\left(S^2 \smallsetminus 4\text{pt}\right) 
\]
Our interest is in those immersed curves whose lift to this cover are homotopic to straight lines; in general terms, the lines of interest fall into two classes, which we call \emph{rational} and \emph{special}. This dichotomy will be discussed in detail (see Definition~\ref{def:RationalVsSpecialKht}), but the important point is that the difference between the two classes amounts to how the lines interact with lifts of the tangle ends in the planar cover. In \cite{pqSym}, the third author solved the geography question for components of \(\HFT(T)\):

\begin{theorem}\label{thm:geography:HFT:intro}
Every component of \(\HFT(T)\) is either rational or special.
\end{theorem}

\(\HFT(T)\) is a geometric interpretation of an algebraic invariant, namely, a curved type~D structure \(\CFTd(T)\) over an algebra~\(\Ad\). The proof of Theorem~\ref{thm:geography:HFT:intro} is based on the existence of an extension of \(\CFTd(T)\) to a curved type D structure \(\CFTminus(T)\) over an algebra~\(\Aminus\) that comes with an epimorphism \(\Aminus\rightarrow\Ad\). This additional structure and, ultimately, the geography result that it leads to play a key role in establishing $\delta$-graded mutation invariance in link Floer homology. To some degree, this extension remains internal to tangle Floer homology; its existence appeals to the Heegaard diagram present in the definition of the invariant. 

Similarly, \(\Khr(T)\) is a geometric interpretation of a type~D structure \(\DD_1(T)^{\BNAlgH}\) over an algebra~\(\BNAlgH\). This type~D structure recasts Bar-Natan's tangle variant of Khovanov homology \cite{BarNatanKhT}. 
Given this starting point, it is less clear where the requisite additional structure comes from. Nevertheless, in this paper, we show:

\begin{theorem}\label{thm:geography:Khr:intro}
	Every component of \(\Khr(T)\) is either rational or special.
\end{theorem}

 The proof of Theorem~\ref{thm:geography:Khr:intro} is also based on an extension of the algebraic invariant: We extend \(\DD_1(T)^{\BNAlgH}\) to a type~D structure \(\DD_1(T)^{\BinfU}\) over an algebra~\(\BinfU\). However, the algebra \(\BinfU\) is an \(A_\infty\) algebra and the extension \(\DD_1(T)^{\BinfU}\) reaches outside of Bar-Natan's framework. Namely, we leverage the matrix factorizations framework~\cite{KR_mf_I,KR_mf_II} to define an $A_\infty$ enhancement of Bar-Natan's cobordism category $\Cob_{/l}$, and in the case of Conway tangles, we describe this $A_\infty$ enhancement explicitly. The latter description depends on a particular quasi-isomorphism of algebras, provided by the homological mirror symmetry of the three-punctured sphere~\cite{AAEKO,LekPol,Orlov}. 
 
%!TEX root = ../main.tex
\section{\texorpdfstring{The tangle invariant \(\Khr\)}{The tangle invariant Khr}}\label{sec:review:Kh}

We review some properties of the immersed curve invariant \(\Khr\) of Conway tangles from~\cite{KWZ}. 
We work exclusively over the two-element field \(\fieldTwoElements\), with remarks about other coefficient systems when appropriate.

\subsection{\texorpdfstring{The definition of \(\Khr\)}{The definition of Khr}}\label{sec:review:Kh:definition}
 
Let \(T\) be an oriented \textbf{pointed} Conway tangle, that is a Conway tangle \(T\) in the three-ball \(B^3\) with a choice of distinguished tangle end, which we usually mark by~\(\ast\) and call the \textbf{special} tangle end. 
%With such a tangle, we associate an invariant \(\Khr(T)\) that takes the form of a multicurve on a four-punctured sphere \(\FourPuncturedSphereKh\), which we can naturally identify with the boundary of \(B^3\) minus the four tangle ends \(\partial T\).
Given $T$ the invariant \(\Khr(T)\) that takes the form of a multicurve on a four-punctured sphere \(\FourPuncturedSphereKh\), which we can naturally identify with the boundary of \(B^3\) minus the four tangle ends \(\partial T\). 
Here, a multicurve is a collection of immersed curves with local systems. 
%Let us explain we mean by this: 
%To make this precise, there are two types of immersed curves: compact curves and non-compact curves. 
To make this precise, consider the cases of compact curves and non-compact curves separately.
A compact curve in \(\FourPuncturedSphereKh\) is an immersion of \(S^1\), considered up to homotopy, that defines a primitive element of \(\pi_1(\FourPuncturedSphereKh)\), together with a local system, ie an invertible matrix over \(\fieldTwoElements\) 
%which is 
considered up to matrix similarity. 
Local systems can be viewed as vector bundles up to isomorphism, where either \(\fieldTwoElements\) is equipped with the discrete topology or the bundle is equipped with a flat connection. 
We will always drop local systems from our notation when they are trivial, ie if they are equal to the unique one-dimensional local system.
A non-compact curve is a non-null-homotopic immersion of an interval into the four-punctured sphere with ends on the two of the three non-special punctures of \(\FourPuncturedSphereKh\). Local systems are always trivial in this case.
Both types of curves carry a bigrading (explained below) and multiple parallel immersed curves in the same bigrading are set to be equivalent to a single curve with a local system that is the direct sum of the individual local systems. 
%We will always assume that parallel immersed curves are bundled up this way. 
% This viewpoint, however, is only relevant for the classification result. 

With this terminology in place, we can sketch the construction of \(\Khr(T)\). It is defined in two steps. 

First, one fixes a diagram \(\Diag_T\) of the pointed tangle $T$. Bar-Natan associates with such a diagram a bigraded chain complex \(\KhTl{\Diag_T}\) over a certain cobordism category \(\Cobl\), whose objects are crossingless tangle diagrams \cite{BarNatanKhT}. 
This complex is a tangle invariant up to bigraded chain homotopy; we denote it by \(\KhTl{T}\). 
% The orientation of the tangle is only needed to fix the absolute grading. 
Thanks to a process Bar-Natan calls \textit{delooping}~\cite[Observation~4.18]{KWZ}, any chain complex over \(\Cobl\) can be written as a chain complex over the full subcategory \(\End_{\Cobl}(\Lo\oplus\Li)\) of \(\Cobl\) generated by the crossingless tangles without closed components. 
This subcategory is isomorphic to the following quiver algebra \cite[Theorem~1.1]{KWZ}:
% Cob and this subcategory are not equivalent, but their respective additive enlargements are!
\begin{equation}\label{eq:B_quiver}
\BNAlgH =
\fieldTwoElements\Big[
\begin{tikzcd}[row sep=2cm, column sep=1.5cm]
\DotB
\arrow[leftarrow,in=145, out=-145,looseness=5]{rl}[description]{{}_{\bullet}D_{\bullet}}
\arrow[leftarrow,bend left]{r}[description]{{}_{\bullet}S_{\circ}}
&
\DotC
\arrow[leftarrow,bend left]{l}[description]{{}_{\circ}S_{\bullet}}
\arrow[leftarrow,in=35, out=-35,looseness=5]{rl}[description]{{}_{\circ}D_{\circ}}
\end{tikzcd}
\Big]\Big/\Big(
\parbox[c]{120pt}{\footnotesize\centering
${}_{\bullet}D_{\bullet} \cdot {}_{\bullet}S_{\circ}=0={}_{\bullet}S_{\circ}\cdot  {}_{\circ}D_{\circ}$\\
${}_{\circ}D_{\circ}\cdot  {}_{\circ}S_{\bullet}=0={}_{\circ}S_{\bullet}\cdot  {}_{\bullet}D_{\bullet}$
}\Big)
\end{equation}
The objects \(\Lo\) and \(\Li\) correspond to \(\DotB\) and \(\DotC\), respectively. We denote the idempotent constant paths on \(\DotB\) and \(\DotC\) by \(\iota_{\bullet}\) and \(\iota_{\circ}\), respectively. We will sometimes abuse notation by using $S$  for either ${}_{\circ}S_{\bullet}$  or ${}_{\bullet}S_{\circ}$ and using $D$ for either ${}_{\circ}D_{\circ}$  or ${}_{\bullet}D_{\bullet}$.
In addition, the subscript $\star\in\{\DotC,\DotB\}$ on the left or right of an algebra element $a$ will always indicate that multiplying by $\iota_\star$ from the left or right respectively preserves the element $a$. This allows shorthand notation such as $S_{\circ}={}_{\bullet}S_{\circ}$ and $S^3_{\bullet}=  {}_{\circ}S_{\bullet}   \cdot {}_{\bullet}S_{\circ} \cdot  {}_{\circ}S_{\bullet}$.
The algebra \(\BNAlgH\) carries a bigrading: The quantum grading \(q\) and the delta grading  \(\delta \) are determined by 
\[
\text{gr}(D_{\bullet}) = \text{gr}(D_{\circ}) = q^{-2}\delta^{-1}
\qquad 
\text{and}
\qquad
\text{gr}(S_{\bullet}) = \text{gr}(S_{\circ}) = q^{-1} \delta^{-\frac{1}{2}} 
\] 
Differentials of bigraded chain complexes over \(\BNAlgH\) are defined to preserve quantum grading and decrease \(\delta\)-grading by 1. 
The isomorphism \(\End_{\Cobl}(\Lo\oplus\Li) \cong \BNAlgH\) allows us to translate the delooped chain complex \(\KhTl{\Diag_T}\) into a bigraded chain complex \(\DD(\Diag_T)^\BNAlgH\)~\cite[Definition~1.2]{KWZ}. ({Chain complexes over ordinary algebras} have also appeared in the literature under the name of {type~D structures}~\cite[Definition~2.2.23]{LOTBimodules}, and we will therefore use the two terms interchangeably; see~\cite[Proposition 2.13]{KWZ} for more on the equivalence of these objects.)  
By construction, the bigraded chain homotopy type of \(\DD(\Diag_T)^\BNAlgH\) is an invariant of the tangle~\(T\), and thus we will sometimes write  \(\DD(T)^\BNAlgH\) for  \(\DD(\Diag_T)^\BNAlgH\). 
Moreover, using the central element 
\[
H\coloneqq D+S^2 = D_{\bullet} + D_{\circ} + S_{\circ}S_{\bullet} + S_{\bullet}S_{\circ} ~ \in ~ \BNAlgH 
\]   
we define a bigraded chain complex \(\DD_1(\Diag_T)\) as the mapping cone 
\[\DD_1(\Diag_T)\coloneqq [q^{-1}\delta^{\frac 1 2}\DD(\Diag_T)\xrightarrow{H\cdot \id} q^{1}\delta^{\frac 1 2}\DD(\Diag_T)]\]
where  $H\cdot \id (x) = x\otimes H$ for every generator $x$ in $\DD(\Diag_T)$.
The bigraded chain homotopy type of \(\DD_1(\Diag_T)\)  is also a tangle invariant, and we will write  \(\DD_1(T)\) for \(\DD_1(\Diag_T)\). 

\begin{figure}[t]
	\centering
	\begin{subfigure}[t]{0.48\textwidth}
		\centering
		$\PairingTrefoilArcINTRO$
		\caption{The curve associated with $\textcolor{blue}{
				[
				\protect\begin{tikzcd}[nodes={inner sep=2pt}, column sep=13pt,ampersand replacement = \&]
				\protect\DotCblue
				\protect\arrow{r}{S}
				\protect\&
				\protect\DotBblue
				\protect\arrow{r}{D}
				\protect\&
				\protect\DotBblue
				\protect\arrow{r}{S^2}
				\protect\&
				\protect\DotBblue
				\protect\end{tikzcd}
				]}$}\label{fig:exa:classification:curves:arc}
	\end{subfigure}
	\
	\begin{subfigure}[t]{0.48\textwidth}
		\centering
		$\PairingTrefoilLoopINTRO$
		\caption{The curve associated with $\textcolor{red}{
				[
				\protect\begin{tikzcd}[nodes={inner sep=2pt}, column sep=23pt,ampersand replacement = \&]
				\protect\DotCred
				\protect\arrow{r}{D+S^2}
				\protect\&
				\protect\DotCred
				\protect\end{tikzcd}
				]}$}\label{fig:exa:classification:curves:loop}
	\end{subfigure}
	\caption{The geometric interpretation of some chain complexes over the algebra $\BNAlgH$ illustrating the classification theorem in the second part of the construction of \(\BNr(T)\) and \(\Khr(T)\)
	}\label{fig:exa:classification:curves}
\end{figure}
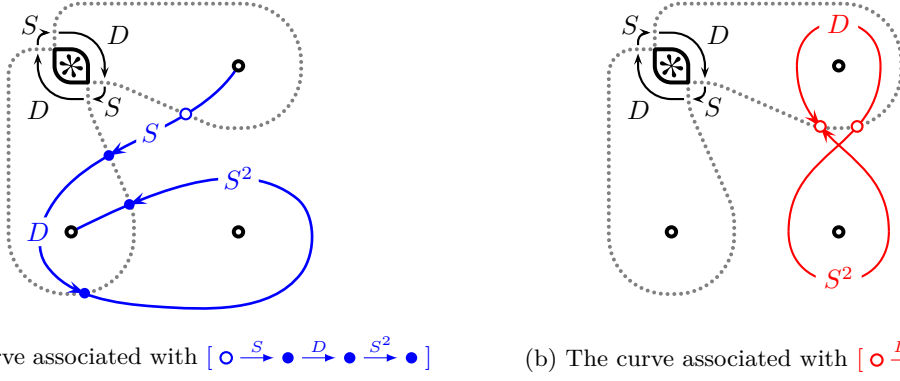

The second step in the definition of \(\Khr(T)\) relies on a classification result, which says that the chain homotopy classes of bigraded chain complexes over \(\BNAlgH\) are in one-to-one correspondence with free homotopy classes of bigraded immersed multicurves with local systems on \(\FourPuncturedSphereKh\) \cite[Theorem~1.5]{KWZ}; see also \cite{HKK}.
The correspondence between chain complexes and immersed multicurves uses a parametrization of \(\FourPuncturedSphereKh\), given by the two dotted arcs shown in Figure~\ref{fig:Kh:example:Curve:Downstairs}.
We will generally assume that the multicurves intersect these arcs minimally. Then, roughly speaking, the intersection points correspond to generators of the according chain complexes and paths between those intersection points correspond to the differentials. 
This is illustrated in Figure~\ref{fig:exa:classification:curves}, cf~\cite[Example~1.6]{KWZ}.
A bigrading on a multicurve then is a bigrading of these intersection points that is consistent with the differential. 
Finally, the multicurve invariant \(\Khr(T)\) is defined as the collection of bigraded immersed curves on \(\FourPuncturedSphereKh\) that corresponds to the bigraded type D structure \(\DD_1(T)\). 
Within the equivalence class of type~D structures that are chain homotopy equivalent to \(\DD_1(T)\), there exist certain distinguished representatives from which the multicurve \(\Khr(T)\) can be read off directly, as in Figure~\ref{fig:exa:classification:curves}. 
We denote such representatives by $\DD_1^c(T)$. 
Similarly, the type~D structure \(\DD(T)\) corresponds to a multicurve \(\BNr(T)\) and we write $\DD^c(T)$ for a type~D structure from which \(\BNr(T)\) can be read off directly.

While \(\Khr(T)\) only consists of compact curves, ie immersed circles, \(\BNr(T)\) also contains \(2^{|T|}\) non-compact components, where \(|T|\) is the number of closed components of \(T\). 
\begin{remark}[Coefficients]
The construction of the tangle invariants \(\DD(T)^\BNAlgH\) and \(\DD_1(T)^\BNAlgH\) can be done over $\Z$, however the classification result only works over fields, and hence so does the construction of the immersed curve invariants. 
% However, in this paper, \(\BNr(T)\) and \(\Khr(T)\) will always denote the curves over $\fieldTwoElements$, unless stated otherwise.
\end{remark}
 
One can identify \(\FourPuncturedSphereKh\) with \(\partial B^3\smallsetminus \partial T\) using the parametrization of the latter shown in Figure~\ref{fig:Kh:example:tangle}. 
This identification is natural, in the following sense: 
If a tangle \(T'\) is obtained from \(T\) by adding twists to the tangle ends, the complexes \(\DD(T')\) and \(\DD_1(T')\) determine two new sets of immersed curves \(\BNr(T')\) and \(\Khr(T')\), which agree with the ones obtained by twisting the immersed curves \(\BNr(T)\) and \(\Khr(T)\) accordingly. This assumes that we work over \(\fieldTwoElements\) \cite[Theorem~1.13]{KWZ}, but we expect the same to hold over arbitrary fields.	

\begin{figure}[t]
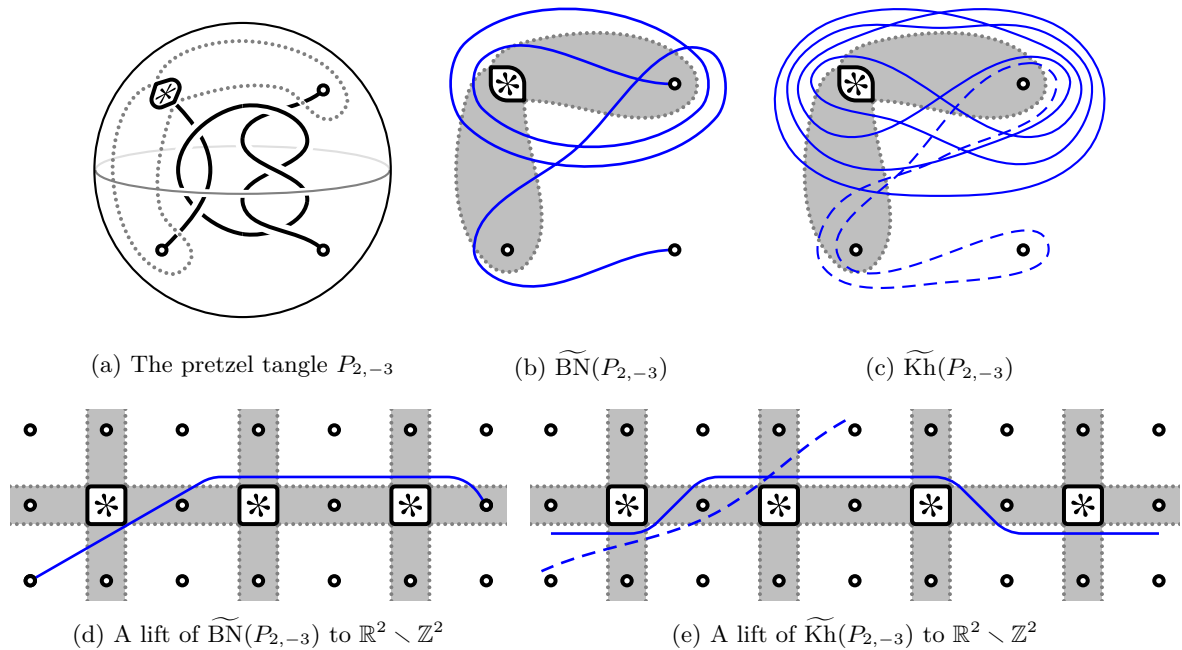

	\centering
	\begin{subfigure}{0.28\textwidth}
		\centering
		\(\pretzeltangleKh\)
		\caption{The pretzel tangle \(P_{2,-3}\)}\label{fig:Kh:example:tangle}
	\end{subfigure}
	\begin{subfigure}{0.28\textwidth}
		\centering
		\(\pretzeltangleDownstairsBNr\)
		\caption{\(\BNr(P_{2,-3})\)}\label{fig:BN:example:Curve:Downstairs}
	\end{subfigure}
	\begin{subfigure}{0.28\textwidth}
	\centering
	\(\pretzeltangleDownstairsKhr\)
	\caption{\(\Khr(P_{2,-3})\)}\label{fig:Kh:example:Curve:Downstairs}
	\end{subfigure}
	\bigskip
	\\
	\begin{subfigure}{0.42\textwidth}
		\centering
		\(\pretzeltangleUpstairsBNr\)
		\caption{A lift of \(\BNr(P_{2,-3})\) to \(\PuncturedPlane\) }\label{fig:BN:example:Curve:Upstairs}
	\end{subfigure}
	\begin{subfigure}{0.55\textwidth}
		\centering
		\(\pretzeltangleUpstairsKhr\)
		\caption{A lift of \(\Khr(P_{2,-3})\) to \(\PuncturedPlane\) }\label{fig:Kh:example:Curve:Upstairs}
	\end{subfigure}
	\caption{The Khovanov and Bar-Natan invariant of the \((2,-3)\)-pretzel tangle }\label{fig:Kh:example}
\end{figure}

\begin{theorem}\label{thm:Kh:Twisting}
	For all 
	\( 
	\tau\in \Mod(\FourPuncturedSphere)
	\), 
	\(
	\Khr(\tau(T))
	=
	\tau(\Khr(T))
	\)
	and 
	\(
	\BNr(\tau(T))
	=
	\tau(\BNr(T))
	\).
\end{theorem}

\begin{remark}
	When working over $\fieldTwoElements$, the distinguished tangle end \(\ast\) only plays a role in the second step of the construction of \(\BNr(T)\) and \(\Khr(T)\). 
	If one works away from characteristic 2, it also plays a subtle role in the first step: In this case, there are four different isomorphisms between \(\End_{\Cobl}(\Lo\oplus\Li)\) and \(\BNAlgH\), which only differ by the signs on the basic morphisms \(D_{\bullet}\) and \( D_{\circ}\). Each of these isomorphisms corresponds to a choice of distinguished tangle end; see~\cite[Theorem~4.21, Observation~4.24]{KWZ} for details. 
\end{remark}

\begin{example}\label{exa:Khr:2m3pt}
	We usually draw the four-punctured sphere \(\FourPuncturedSphereKh\) as the plane plus a point at infinity minus the four punctures and indicate its standard parametrization that identifies \(\FourPuncturedSphereKh\) with \(\partial B^3\smallsetminus \partial T\) by two dotted arcs as in  
	Figures~\ref{fig:BN:example:Curve:Downstairs} and~\ref{fig:Kh:example:Curve:Downstairs}. 
	The blue curves in these figures show \(\BNr(P_{2,-3})\) and \(\Khr(P_{2,-3})\), respectively, where \(P_{2,-3}\) is the \((2,-3)\)-pretzel tangle from Figure~\ref{fig:Kh:example:tangle}, cf~\cite[Example~6.7]{KWZ}.
	All components of these curves carry the (unique) one-dimensional local system over $\fieldTwoElements$.
\end{example}

\begin{example}
	For any slope \(s\in\QPI\), \(\BNr(Q_s)\) consists of a single arc which is obtained by pushing the tangle strand that does not end on the distinguished tangle end \(\ast\) onto \(\FourPuncturedSphereKh\). The invariant \(\Khr(Q_s)\) is equal to a figure-eight curve that lies in a small neighbourhood of \(\BNr(Q_s)\) and encloses the two tangle ends on either side, see \cite[Example~6.6]{KWZ}. 
	The local system on this curve is one-dimensional.
	We expect that over arbitrary fields, the underlying curve for \(\Khr(Q_s)\) is the same as over \(\fieldTwoElements\), and that the local systems on these curves are equal to \((-1)\).
\end{example}

Like \(\HFT\), the tangle invariants in Khovanov theory detect rational tangles. 

\begin{theorem}
	\label{thm:Kh:rational_tangle_detection}
	A tangle \(T\) is rational if and only if \(\Khr(T)\) consists of a single figure-eight curve carrying the unique one-dimensional local system. 
\end{theorem}

\begin{proof}
	This follows from essentially the same arguments as \cite[Theorem~6.2]{pqMod}. 
	Indeed, something more general holds:
	%More generally, any tangle invariant detects rational tangles, as long as the tangle invariant satisfies a gluing theorem to a link invariant that detects the two-component unlink. 
	Any tangle invariant detects rational tangles if it satisfies a gluing theorem that recovers a link invariant known to detect the two-component trivial link.	 
	The gluing theorem for \(\Khr\) is stated below (Theorem~\ref{thm:GlueingTheorem:Kh});  and the requisite detection result was proved by Hedden and Ni~\cite{KhDetectsTwoComponentUnlink}, based on Kronheimer and Mrowka's unknot detection for Khovanov homology~\cite{KhDetectsUnknot}.
\end{proof}

Note that \(\Khr(Q_s)\) is not embedded. This is true more generally \cite[Proposition~6.18]{KWZ}: 

\begin{proposition}\label{prop:Khr_not_embedded}
	For any pointed Conway tangle \(T\), no component of \(\Khr(T)\), over any field, is embedded.
\end{proposition}

\subsection{\texorpdfstring{A gluing theorem for \(\Khr\)}{A gluing theorem for Khr}}\label{sec:review:Kh:gluing}

\begin{figure}[bt]
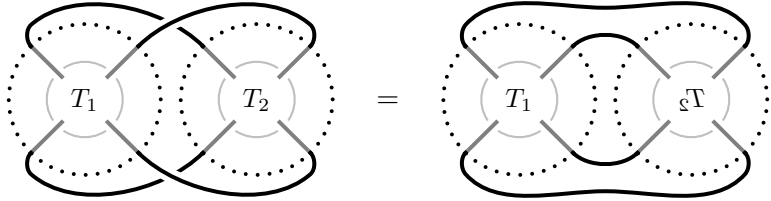

	\centering
	\(
	\tanglepairingI
	\quad = \quad
	\tanglepairingII
	\)
	\caption{Two Conway tangle decompositions defining the link \(T_1\cup T_2\). The tangle \protect\reflectbox{\(T_2\)} is the result of rotating \(T_2\) around the vertical axis. By rotating the entire link on the right-hand side around the vertical axis, we can see that \(T_1\cup T_2=T_2\cup T_1\).}
	\label{fig:tanglepairing} 
\end{figure}

Denote the two-dimensional vector space supported in \(\delta\)-grading \(+\tfrac{1}{2}\) and quantum gradings \(\pm1\) by
\[V\coloneqq \delta^{\frac 1 2} q^1\F \oplus \delta^{\frac 1 2} q^{-1} \F\]

\begin{theorem}[\text{\cite[Theorem~1.9]{KWZ}}]\label{thm:GlueingTheorem:Kh}
	Let \(L=T_1\cup T_2\) be the result of gluing two oriented pointed Conway tangles as in Figure~\ref{fig:tanglepairing} such that the orientations match. Let \(T^*_1\) be the mirror image of \(T_1\) with the orientation of all components reversed. Then  
	\[
	\begin{aligned}
	\Khr(L)\otimes V
	&\cong
	\HF\left(\Khr(T^*_1),\Khr(T_2)\right) \\
	\Khr(L)&\cong \HF\left(\Khr(T_1^*),\BNr(T_2)\right)
	\end{aligned}
	\]
	as relatively bigraded $\F$-vector spaces.
\end{theorem}

\(\Khr(T^*_1)\) can be easily computed from  \(\Khr(T_1)\) \cite[Proposition~7.1]{KWZ}. For this, let \(\mirror\) denote the mirror operation, ie the involution of the four-punctured sphere that fixes the four punctures pointwise, fixes the parametrizing arcs setwise, and interchanges the front and back. 

\begin{lemma}\label{lem:mirroring:Khr}
	For any pointed Conway tangle \(T\), \(\Khr(T^*)=\mirror(\Khr(T))\) up to an appropriate bigrading shift.
\end{lemma}

\begin{remark}
	Shumakovitch showed that over \(\fieldTwoElements\) the unreduced Khovanov homology of a link splits into two copies of reduced Khovanov homology, whose \(\delta\)-gradings differ by one \cite{Shum_torsion}. 
	Theorem~\ref{thm:GlueingTheorem:Kh} does \emph{not} compute the unreduced Khovanov homology; instead we have that the \(\delta\)-gradings of the two copies of reduced Khovanov homology are identical. 
	To obtain unreduced Khovanov homology, one can use yet another immersed curve invariant, namely \(\Kh(T)\), which we introduced in~\cite{KWZ} and which satisfies analogous gluing theorems. 
\end{remark}

\subsection{The plan for the rest of the paper}\label{subsec:plan}

We now outline the geography results for the invariant \(\Khr\) and the strategy for our proof.
As pointed out in the introduction, it is useful to consider the immersed curves in the covering space \(\PuncturedPlane\) of \(\FourPuncturedSphereKh\).
This covering space is illustrated in Figures~\ref{fig:BN:example:Curve:Upstairs} and~\ref{fig:Kh:example:Curve:Upstairs}, where the parametrization of \(\FourPuncturedSphereKh\) has been lifted to \(\PuncturedPlane\) and the two non-adjacent faces and their preimages under the covering map are shaded grey. 
The two figures also include the lifts of the components of \(\BNr(P_{2,-3})\) and \(\Khr(P_{2,-3})\), respectively. 
Note that for \(\Khr(P_{2,-3})\), the lift of each component can be isotoped into an arbitrarily small neighbourhood of a straight line of some rational slope \(\nicefrac{p}{q}\in\QPI\) passing through some punctures. 
Our geography result implies that this is true in general for the invariant \(\Khr(T)\). 
%But first, let us introduce some terminology.
Stating this carefully requires some additional terminology.

We will consider closed curves \(\gamma\) in \(\FourPuncturedSphereKh\). 
In view of Proposition~\ref{prop:Khr_not_embedded}, we will restrict to those curves whose free homotopy class does not contain a representative curve that is embedded. 
It is useful to introduce a ``normal form'' for curves in \(\PuncturedPlane\), cf~\cite[Section~7.1]{HRW} and \cite[Section~3]{pqSym}:

\begin{definition}\label{def:peg_board_rep}
	Consider the standard Riemannian metric on \(\PuncturedPlane\), which induces a Riemannian metric on \(\FourPuncturedSphereKh\). 
	Fix some \(\varepsilon\) with \(0<\varepsilon<1/2\). 
	Define an \textbf{\(\bm{\varepsilon}\)-peg-board representative} of a closed curve \(\gamma\) in \(\FourPuncturedSphereKh\) as a representative of the homotopy class of \(\gamma\) which has minimal length among all representatives of distance \(\varepsilon\) to all four punctures in \(\FourPuncturedSphereKh\). 
\end{definition}

The intuition behind this definition is to think of the four punctures of \(\FourPuncturedSphereKh\) as pegs of radii \(\varepsilon\) and then to imagine pulling the curve \(\gamma\) ``tight'', like a rubber band. 
In the setting of Khovanov homology, we will be able to assume that the \(\varepsilon\)-peg-board representative of \(\gamma\) is unique.
Consider an infinite connected lift $\InfConLift{\gamma}\co \mathbb R \to\PuncturedPlane $ of a peg-board representative $\gamma\co S^1\to\FourPuncturedSphereKh$---in the sense that the compositions  $ \mathbb R \xrightarrow{\InfConLift{\gamma}} (\PuncturedPlane) \to\FourPuncturedSphereKh$ and $\mathbb R  \to S^1 \xrightarrow{\gamma} \FourPuncturedSphereKh$ coincide---along with the pegs of radii \(\varepsilon\) that sit at the lattice points of $\PuncturedPlane$. If we then take the limit \(\varepsilon\rightarrow0\), the lift $\InfConLift{\gamma}$ becomes a piecewise linear curve, which we call a \textbf{singular peg-board representative} of \(\gamma\). %Up to deck transformations, it is unique. 
It is unique up to deck transformations.

We say a curve \(\gamma\) in \(\FourPuncturedSphereKh\) \textbf{wraps} around a puncture if there exists some angle \(\alpha>0\) and  $\delta>0$ such that for all \(\delta > \varepsilon >0\), the \(\varepsilon\)-pegboard representative of \(\Lgamma\) changes its direction at a lift of this puncture by an angle \(\geq\alpha\). 
A curve \(\gamma\) is called \textbf{linear} if it does not wrap around any puncture.
If \(\gamma\) is linear, its singular peg-board representative is contained in a single line. However, the converse is false, due to possible $k\pi$-turns of $\gamma$ around punctures. (This is referred to as peg-wrapping in \cite{HRW}.)
%	Note that the above definitiomn of linearity coincides with the one in Definition~\ref{def:linearity_via_derivative}.
\begin{definition}\label{def:RationalVsSpecialKht}
	We call a linear curve \(\gamma\) \textbf{special} if its singular peg-board representative contains the lift of the special puncture $\ast$, and \textbf{rational} otherwise.
\end{definition}

\begin{example}
	%\(\Khr(P_{2,-3})\) from Figure~\ref{fig:Kh:example:Curve:Upstairs} consists of a rational curve of slope \(\nicefrac{1}{2}\) and a special curve of slope \(0\). \(\BNr(P_{2,-3})\) from Figure~\ref{fig:BN:example:Curve:Upstairs} is not linear, since it wraps around a non-special puncture, namely the one corresponding to the upper right tangle end. 
	The invariant \(\Khr(P_{2,-3})\) shown in Figure~\ref{fig:Kh:example:Curve:Upstairs} consists of a rational curve of slope \(\nicefrac{1}{2}\) and a special curve of slope \(0\). The invariant \(\BNr(P_{2,-3})\) shown in Figure~\ref{fig:BN:example:Curve:Upstairs} is not linear, since it wraps around a non-special puncture, namely the one corresponding to the upper right tangle end.
	The singular peg-board representative of \(\BNr(P_{2,-3})\) consists of two linear segments of slopes \(\nicefrac{1}{2}\) and \(0\). 
\end{example}

\begin{theorem}\label{thm:geography_of_Khr}
	For any pointed Conway tangle \(T\), every component of \(\Khr(T)\) is linear.
\end{theorem}

The proof occupies Sections~\ref{sec:Kh:geography:non-special} and~\ref{sec:no_wrapping_around_special}: First, in Theorem~\ref{thm:almost_no_wrapping}, we rule out wrapping around non-special punctures; then, in Theorem~\ref{thm:no_wrapping_around_special}, we rule out wrapping around the special puncture. 

With the linearity result in hand, it is possible to further restrict the geography of \(\Khr(T)\); this is the focus of Section~\ref{sec:further_restrictions}. 
For \(n\in\mathbb N\), let \(\rKh_{n}(0)\) and \(\sKh_{2n}(0)\) be the immersed curves in \(\FourPuncturedSphereKh\) that  admit lifts to the curves $\tilde{\rKh}_{n}(0)$ and $\tilde{\sKh}_{2n}(0)$, respectively, in Figure~\ref{fig:geography:Upstairs}; curves for $n=1,2,3$ are illustrated in Figures~\ref{fig:geography:Downstairs1}--\ref{fig:geography:Downstairs3}.
We will refer to the subscripts \(n\), respectively \(2n\), as the \textit{lengths} of those curves. 
For every slope \(\nicefrac{p}{q}\in\QPI\), we  define the curves $\rKh_n(\nicefrac{p}{q})$ and $\sKh_{2n}(\nicefrac{p}{q})$ as the images of \(\rKh_{n}(0)\) and \(\sKh_{2n}(0)\), respectively,  under the action of the matrix
\[
\begin{bmatrix*}[c]
q & r \\
p & s
\end{bmatrix*} \qquad\text{where} \qquad qs-pr=1,
\]
considered as an element the mapping class group $\operatorname{Mod}(\FourPuncturedSphereKh) \cong PSL(2,\Z)$ consisting of mapping classes fixing the special puncture $\ast$. (This transformation maps straight lines of slope 0 to straight lines of slope \(\nicefrac{p}{q}\). The isomorphism $\operatorname{Mod}(\FourPuncturedSphereKh) \cong PSL(2,\Z)$ is induced by the two-fold branched cover $\mathbb T^2 \to \FourPuncturedSphereKh$ and the isomorphism $\operatorname{Mod}(\mathbb T^2) \cong SL(2,\Z)$.)
The local systems on each of these curves are defined to be trivial. 
In Section~\ref{sec:further_restrictions}, we prove:

\begin{theorem}\label{thm:further_geography_of_Khr}
	Every component of the curve invariant \(\Khr(T)\) is equal to either \(\rKh_n( \nicefrac p q)\) or  \(\sKh_{2n}(\nicefrac p q)\) for some \(n\geq 1, \nicefrac p q \in \QPI\).
\end{theorem}

So in particular, the local system on any component of \(\Khr(T)\) is trivial for any tangle \(T\). 

\begin{figure}[t]
	\centering
	\begin{subfigure}{0.3\textwidth}
		\centering
		\labellist 
		\footnotesize \color{blue}
		\pinlabel $\sKh_{2}(0)$ at 65 115
		\pinlabel $\rKh_{1}(0)$ at 65 28
		\endlabellist
		\includegraphics[scale=1]{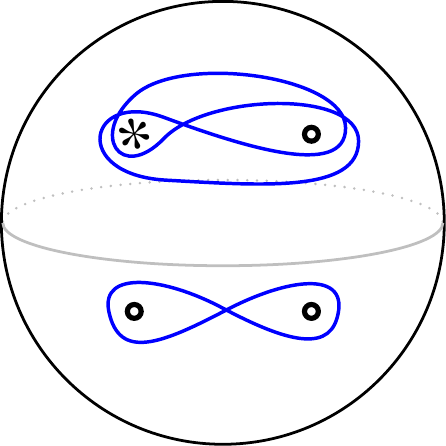}
		\caption{\(n=1\)}\label{fig:geography:Downstairs1}
	\end{subfigure}
	\begin{subfigure}{0.3\textwidth}
		\centering
		\labellist 
		\footnotesize \color{blue}
		\pinlabel $\sKh_{4}(0)$ at 65 118
		\pinlabel $\rKh_{2}(0)$ at 65 13
		\endlabellist
		\includegraphics[scale=1]{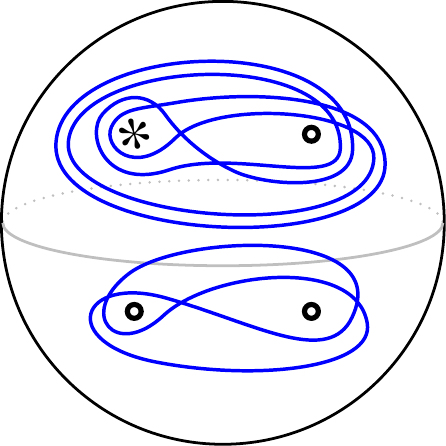}
		\caption{\(n=2\)}\label{fig:geography:Downstairs2}
	\end{subfigure}
	\begin{subfigure}{0.3\textwidth}
		\centering
		\labellist 
		\footnotesize \color{blue}
		\pinlabel $\sKh_{6}(0)$ at 65 121
		\pinlabel $\rKh_{3}(0)$ at 65 10
		\endlabellist
		\includegraphics[scale=1]{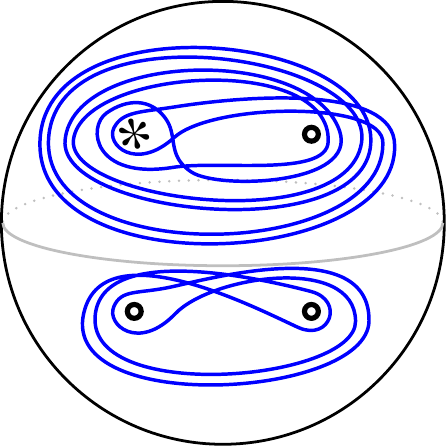}
		\caption{\(n=3\)}\label{fig:geography:Downstairs3}
	\end{subfigure}
	\\
	\begin{subfigure}{0.9\textwidth}
		\centering
		\(\GeographyCovering\)
		\caption{}\label{fig:geography:Upstairs}
	\end{subfigure}
	\caption{The curves \(\rKh_n(0)\) and \(\sKh_{2n}(0)\) (a--c) and their lifts to \(\PuncturedPlane\) (d). }\label{fig:geography}
\end{figure}

%!TEX root = ../main.tex
\section{Wrapping around non-special tangle ends}\label{sec:Kh:geography:non-special}

\begin{theorem}\label{thm:almost_no_wrapping}
	A curve that wraps around a non-special puncture cannot be a component of \(\Khr(T)\) for any pointed Conway tangle \(T\). 
\end{theorem}

The proof of Theorem~\ref{thm:almost_no_wrapping} uses the naturality under twisting tangle ends (Theorem~\ref{thm:Kh:Twisting}) in an essential way. If this naturality holds over arbitrary fields, then so does Theorem~\ref{thm:almost_no_wrapping}.

Recall that $\DD^c(T)$ and $\DD_1^c(T)$ are the type D structures from which we can read off the multicurves $\BNr(T)$ and $\Khr(T)$ directly, as in Figure~\ref{fig:exa:classification:curves}.

\begin{proposition}\label{prop:higher_powers_when_wrapping}
	The differential of the type D structure \(\DD_1^c(T)\) only contains linear combinations of \(D,S,S^2\in\BNAlgH\). 
\end{proposition}

Given two type D structures \(X_1^\BNAlgH\) and \(X_2^\BNAlgH\), the space of morphisms \(\Mor(X_1,X_2)\) carries a natural action by the center \(Z(\BNAlgH)\) of \(\BNAlgH\). This action is defined by 
\[
z \cdot f
=
(\id_{X_2}\otimes \mu_{\BNAlgH})
\circ
(f\otimes z)
\quad
\text{for }
f\in\Mor(X_1,X_2),
\text{ and } 
z\in Z(\BNAlgH),
\]
where \(\mu_{\BNAlgH}\) denotes the multiplication in \(\BNAlgH\). 
In particular, for \(z\in\{H,D_{\circ},D_{\bullet},S^2\}\) and \(X=X_1=X_2\), the type D structure homomorphisms 
\begin{equation}\label{eq:morphisms}
H\cdot\id, ~ S^2\cdot \id , ~ D_{\bullet} \cdot \id, ~D_{\circ} \cdot \id ~ \in ~ \Mor(X,X)
\end{equation}
are defined by the formulas
\begin{align*}
H\cdot\id (x) 
&= x \otimes H, 
&
S^2\cdot\id (x) 
&= x \otimes S^2, 
\\ 
D_{\bullet} \cdot \id (x_{\bullet})
&= x_{\bullet} \otimes D_{\bullet},
~
D_{\bullet} \cdot \id (x_{\circ})
= 0,
&
D_{\circ} \cdot \id (x_{\circ})
&= x_{\circ} \otimes D_{\circ},
~
D_{\circ} \cdot \id (x_{\bullet})
= 0,
\end{align*}
where $x_{\bullet}$ and $x_{\circ}$ denote generators in idempotents \(\DotB\) and \(\DotC\), respectively.
The following lemma will be used in this section multiple times. 

\begin{lemma}\label{lem:morphisms_are_preserved}
  Let \(X_1^\BNAlgH\simeq \tilde{X}_1^\BNAlgH\) and \(X_2^\BNAlgH\simeq \tilde{X}_2^\BNAlgH\) be two pairs of chain homotopic type~D structures. 
  Then the induced chain homotopy equivalence \(\Mor(X_1,X_2) \simeq \Mor(\tilde{X}_1,\tilde{X}_2)\) is compatible with the actions of \(Z(\BNAlgH)\). 
  In particular, if \(X=X_1=X_2\) and \(\tilde{X}=\tilde{X}_1=\tilde{X}_2\), the morphism $z \cdot \id\in \Mor(X,X)$ is identified with the morphism $z\cdot \id\in \Mor(\tilde{X},\tilde{X})$ for \(z\in\{H,D_{\circ},D_{\bullet},S^2\}\).
\end{lemma}
\begin{proof}
	Straightforward.
%In the case of~\cite[Cancellation Lemma~2.16]{KWZ}, the two chain maps in the proof of that lemma satisfy $G\circ F=\id$, which implies $G \circ (H\cdot \id)  \circ F=H\cdot\id$. This proves that the homotopy equivalence \(\Mor(X_1,X_1) \to \Mor(X_2,X_2)\) sends $H\cdot \id$ to $H\cdot \id$. It follows that the inverse map \(\Mor(X_1,X_1) \leftarrow \Mor(X_2,X_2)\) does the same (not necessarily on the nose, though, only up to homotopy, because $F\circ G \neq \id$). The same works for other morphisms from Equation~\eqref{eq:morphisms}.
%
%The case of~\cite[Clean-Up Lemma~2.11]{KWZ} is analogous.
\end{proof}

\begin{proof}[Proof of Proposition~\ref{prop:higher_powers_when_wrapping}]
	For simplicity, let us assume that \(\BNr(T)\) carries no non-trivial local systems; the general argument is very similar to the one below.
	
  We have
  $$\DD^c_1(T)\simeq \DD_1(T)=[\DD(T)\xrightarrow{H\cdot \id}\DD(T)]\simeq [\DD^c(T)\xrightarrow{H\cdot \id}\DD^c(T)]$$
  where the last homotopy equivalence follows from $\DD(T)\simeq \DD^c(T)$ and Lemma~\ref{lem:morphisms_are_preserved}. Let us study the type D structure $[\DD^c(T)\xrightarrow{H\cdot \id}\DD^c(T)]$: If the differential of \(\DD^c(T)\) contains any component \(D^n\) or \(S^m\) starting at \(\DotB\) for \(n>1\) and \(m>2\), then the corresponding portions of the mapping cone look as follows (without the dotted arrows):
	\[
	\begin{tikzcd}
	\DD^c(T)
	\arrow{d}[swap]{H\cdot \operatorname{id}_{\DD^c(T)}}
	&
	\cdots
	\arrow[leftrightarrow,dashed]{r}{S^a}
	&
	\DotB
	\arrow{r}{D^n}
	\arrow{d}[swap]{H}
	&
	\DotB
	\arrow[leftrightarrow,dashed]{r}{S^b}
	\arrow{d}{H}
	&
	\cdots
	&
	\cdots
	\arrow[leftrightarrow,dashed]{r}{D^a}
	&
	\DotB
	\arrow{r}{S^m}
	\arrow{d}[swap]{H}
	&
	\DotC
	\arrow[leftrightarrow,dashed]{r}{D^b}
	\arrow{d}{H}
	&
	\cdots
	\\
	\DD^c(T)
	&
	\cdots
	\arrow[leftrightarrow,dashed]{r}{S^a}
	&
	\DotB
	\arrow{r}{D^n}
	\arrow[dotted]{ru}[description]{D^{n-1}}
	&
	\DotB
	\arrow[leftrightarrow,dashed]{r}{S^b}
	&
	\cdots
	&
	\cdots
	\arrow[leftrightarrow,dashed]{r}{D^a}
	&
	\DotB
	\arrow{r}{S^m}
	\arrow[dotted]{ru}[description]{S^{m-2}}
	&
	\DotC
	\arrow[leftrightarrow,dashed]{r}{D^b}
	&
	\cdots
	\end{tikzcd}
	\]
	Here, \(a\) and \(b\) are some positive integers.
	By applying the Clean-Up Lemma \cite[Lemma~2.17]{KWZ} to the dotted arrows, we can see that this type D structure is isomorphic to the following:
	\[
	\begin{tikzcd}
	\DD^c(T)
	\arrow{d}[swap]{H\cdot \operatorname{id}_{\DD^c(T)}}
	&
	\cdots
	\arrow[leftrightarrow,dashed]{r}{S^a}
	&
	\DotB
	%	\arrow{r}{D^n}
	\arrow{d}{H}
	&
	\DotB
	\arrow[leftrightarrow,dashed]{r}{S^b}
	\arrow{d}[swap]{H}
	&
	\cdots
	&
	\cdots
	\arrow[leftrightarrow,dashed]{r}{D^a}
	&
	\DotB
	%	\arrow{r}{S^n}
	\arrow{d}{H}
	&
	\DotC
	\arrow[leftrightarrow,dashed]{r}{D^b}
	\arrow{d}[swap]{H}
	&
	\cdots
	\\
	\DD^c(T)
	&
	\cdots
	\arrow[leftrightarrow,dashed]{r}{S^a}
	&
	\DotB
	%	\arrow{r}{D^n}
	&
	\DotB
	\arrow[leftrightarrow,dashed]{r}{S^b}
	&
	\cdots
	&
	\cdots
	\arrow[leftrightarrow,dashed]{r}{D^a}
	&
	\DotB
	%	\arrow{r}{S^n}
	&
	\DotC
	\arrow[leftrightarrow,dashed]{r}{D^b}
	&
	\cdots
	\end{tikzcd}
	\]
	Note that the type D structure remains the same outside of the shown region. The same argument can be used to replace components of the differential containing \(D^n\) or \(S^m\) starting at \(\DotC\); all we need to do in the diagrams above is to exchange \(\DotB\) and \(\DotC\). 
	
  Denote by $X$ the type~D structure obtained from $[\DD^c(T)\xrightarrow{H\cdot \id}\DD^c(T)]$ by cleaning up all higher powers.
  The above shows that $\DD_1^c(T)$ is homotopy equivalent to the type D structure $X$ with the property that the components of its differential are linear combinations of \(D,S,S^2\in\BNAlgH\). 
	It remains to see that this property is preserved under the algorithm that turns $X$ into $\DD_1^c(T)$. 
	The arrow-pushing algorithm from~\cite{HRW} that is applied for this purpose in~\cite[Section~5]{KWZ} only modifies the curve in the neighbourhoods of the arcs corresponding to the two objects \(\DotB\) and \(\DotC\). Therefore, it suffices to see that the simply-faced precurve associated with \(X\) (an auxiliary object introduced in~\cite[Section~5.3]{KWZ} to intermediate between type D structures and multicurves) also has this special property. 
	The algorithm in the proof of \cite[Proposition 5.10]{KWZ} unfortunately does not preserve the property in general. 
	For example, 
	\[
	\Bigg[
	\begin{tikzcd}
	\DotB
	\arrow{r}{S^2}
	&
	\DotB
	\arrow[leftarrow]{r}{S}
	&
	\DotC
	\arrow{r}{S^2}
	&
	\DotC
	\end{tikzcd}
	\Bigg]
	\cong
	\Bigg[
	\begin{tikzcd}
	\DotB
	\arrow[bend right=10,swap]{rrr}{S^3}
	&
	\DotB
	\arrow[leftarrow]{r}{S}
	&
	\DotC
	&
	\DotC
	\end{tikzcd}
	\Bigg]
	\] 
	However, for the particular type D structure under consideration, one can easily construct a corresponding simply-faced precurve by hand and verify that it has the desired property: 
	First, observe that after performing the homotopies above, the type D structure \(X\) can be built out of the following two pieces (without the dotted arrows) and the same pieces with \(\DotB\) and \(\DotC\) exchanged:
	\[
	\begin{tikzcd}[row sep=35pt,column sep=45pt]
	\cdots
	\arrow[leftrightarrow,dashed]{r}{S^a}
	&
	\DotB
	\arrow{r}{D}
	\arrow[bend left=15]{d}{D}
	\arrow[dashed,bend right=15]{d}[swap]{S^2}
	&
	\DotB
	\arrow[leftrightarrow,dashed]{r}{S^b}
	\arrow[bend right=15]{d}[swap]{D}
	\arrow[dashed,bend left=15]{d}{S^2}
	&
	\cdots
	\\
	\cdots
	\arrow[leftrightarrow,dashed]{r}{S^a}
	&
	\DotB
	\arrow{r}{D}
	\arrow[dotted,leftarrow]{ru}[description]{1}
	&
	\DotB
	\arrow[leftrightarrow,dashed]{r}{S^b}
	&
	\cdots
	\end{tikzcd}
	\quad
	\begin{tikzcd}[row sep=35pt,column sep=45pt]
	\cdots
	\arrow[leftrightarrow,dashed]{r}{D}
	&
	\DotB
	\arrow{r}{S^n}
	\arrow[bend left=15]{d}{S^2}
	\arrow[dashed,bend right=15]{d}[swap]{D}
	&
	\DotC
	\arrow[leftrightarrow,dashed]{r}{D}
	\arrow[bend right=15]{d}[swap]{S^2}
	\arrow[dashed,bend left=15]{d}{D}
	&
	\cdots
	\\
	\cdots
	\arrow[leftrightarrow,dashed]{r}{D}
	&
	\DotB
	\arrow{r}{S^n}
	\arrow[leftarrow,dotted]{ru}[description]{S^{2-n}}
	&
	\DotC
	\arrow[leftrightarrow,dashed]{r}{D}
	&
	\cdots
	\end{tikzcd}
	\]
	Here, \(n,a,b\in\{1,2\}\). 
	Note that we have replaced each vertical component \(H\) of the differential by two components \(D\) and \(S^2\), using the identity \(H=D+S^2\). 
	In the case of the second piece with \(n=1\), we can apply the Clean-Up Lemma to the dotted arrow to remove the vertical components \(S^2\) of the differential. In the case of the first piece and the second piece with \(n=2\), the dotted arrow is labelled by \(1\in\BNAlgH\); in each of these two cases, we do a base change along this dotted arrow, thereby simultaneously adding a crossover arrow to the precurve and removing the vertical arrows \(D\) and \(S^2\), respectively. The result is a simply-faced precurve, and its differential only consists of components that are linear combinations of \(D,S,S^2\in\BNAlgH\). 
\end{proof}

\begin{figure}[t]
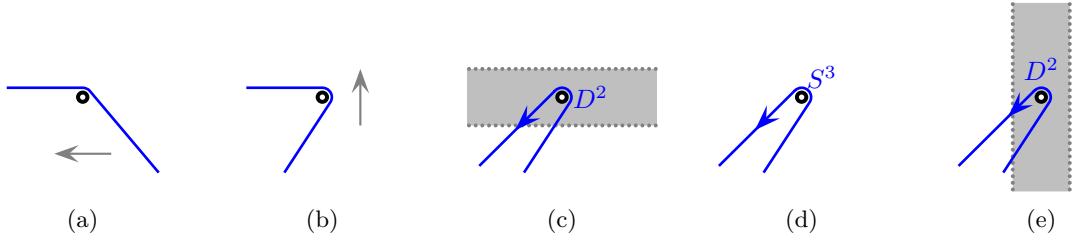

	\centering
	\begin{subfigure}{0.19\textwidth}
		\centering
		\(\WindingNonSpecialA\)
		\caption{}\label{thm:almost_no_wrapping:a}
	\end{subfigure}
	\begin{subfigure}{0.19\textwidth}
		\centering
		\(\WindingNonSpecialB\)
		\caption{}\label{thm:almost_no_wrapping:b}
	\end{subfigure}
	\begin{subfigure}{0.19\textwidth}
		\centering
		\(\WindingNonSpecialC\)
		\caption{}\label{thm:almost_no_wrapping:c}
	\end{subfigure}
	\begin{subfigure}{0.19\textwidth}
		\centering
		\(\WindingNonSpecialD\)
		\caption{}\label{thm:almost_no_wrapping:d}
	\end{subfigure}
	\begin{subfigure}{0.19\textwidth}
		\centering
		\(\WindingNonSpecialE\)
		\caption{}\label{thm:almost_no_wrapping:e}
	\end{subfigure}
	\caption{An illustration for the proof of Theorem~\ref{thm:almost_no_wrapping}. The arrows in (a) and (b) indicate the directions of the two shearing transformations to get from (a) to (b) and from (b) to (c), (d), or (e).}\label{fig:almost_wrapping}
\end{figure} 

\begin{proof}[Proof of Theorem~\ref{thm:almost_no_wrapping}]
	Suppose a curve changes its direction at a non-special puncture. 
	Then, by naturality under twisting (Theorem~\ref{thm:Kh:Twisting}), we can assume that one of the linear curve segments adjacent to it has slope 0, as shown in Figure~\ref{thm:almost_no_wrapping:a}. 
	Moreover, by adding twists corresponding to the shearing transformation
	\[
	\begin{bmatrix}
	1 & n \\
	0 & 1
	\end{bmatrix}\in\SL_2(\Z)
	\]
	for sufficiently large \(n\), we can achieve that the curve changes its direction by more than $90^\circ$, as shown in Figure~\ref{thm:almost_no_wrapping:b}. 
	By adding one additional twist corresponding to
	\[
	\begin{bmatrix}
	1 & 0 \\
	1 & 1
	\end{bmatrix}\in\SL_2(\Z)
	\] 
	we can make the first curve segment have slope 1 and the angle of the change of direction be more than $135^\circ$.
	Parts (c-e) of Figure~\ref{fig:almost_wrapping} show the three cases depending on which non-special puncture the curve wraps around. 
	In all three cases, the differential of the corresponding type~D structure either contains a component \(D^n\) for \(n>1\) or a component \(S^m\) for \(m>2\). 
	This contradicts Proposition~\ref{prop:higher_powers_when_wrapping}.
\end{proof}

%!TEX root = ../main.tex
\section{Wrapping around the special tangle end}\label{sec:Kh:geography:special}\label{sec:no_wrapping_around_special}

\begin{theorem}\label{thm:no_wrapping_around_special}
	A curve that wraps around the special puncture cannot be a component of \(\BNr(T)\) or \(\Khr(T)\) for any pointed Conway tangle \(T\). 
\end{theorem}
The proof is based on a certain extension property of the complexes \(\DD(T)^{\B}\) and \(\DD_1(T)^{\B}\).
\subsection{The extension property}
Consider the two arcs \(\arcVer\) and \(\arcHor\) in \(\FourPuncturedSphereKh\) that are labelled in Figure~\ref{fig:wrap_subcat} by \(\DotC\) and \(\DotB\), respectively. Let \(\End_{\W(\FourPuncturedSphereKh)}(\arcVer\oplus \arcHor)\) denote the full subcategory of the wrapped Fukaya category \(\W(\FourPuncturedSphereKh)\) generated by these arcs. 
It is well-known that such a subcategory admits a quiver description in terms of chords; see for example~\cite[Theorem~7.6]{Bocklandt}. 
In fact, we obtain the algebra \(\BNAlgH\) from Equation~\eqref{eq:B_quiver}:
\[
\End_{\W(\FourPuncturedSphereKh)}(\arcVer\oplus \arcHor) \cong \B
\]
If \(\ThreePuncturedSphere\) denotes the three-punctured sphere obtained by removing the special puncture \(*\) from \(\FourPuncturedSphereKh\), the arcs \(\arcVer\) and \(\arcHor\) also define a subcategory \(\End_{\W(\ThreePuncturedSphere)}(\arcVer\oplus \arcHor)\) of \(\W(\ThreePuncturedSphere)\). 
The product structure on \(\End_{\W(\ThreePuncturedSphere)}(\arcVer\oplus \arcHor)\) again agrees with the one on \(\B\), but (unlike the four-punctured case) there are also non-trivial higher products.  We describe the resulting \(A_\infty\) algebra \(\Binf\) below. The fact that \(\End_{\W(\ThreePuncturedSphere)}(\arcVer\oplus \arcHor) \cong \Binf\) follows from~\cite[Section~3.3]{HKK} and~\cite[Theorem~4.1]{AAEKO} applied to the three arcs $\arcVer, \arcHor, \arcDiag$ from Figure~\ref{fig:wrap_subcat}.
\begin{definition}\label{def:Binf_algebra}
	Let \(\Binf\) be an \(A_\infty\) algebra with the same generators and product \(\mu_2\) as the algebra \(\B\), and with higher products defined as follows. Define \emph{disk sequences} of algebra elements in \(\B\)  inductively: First, declare the shortest disk sequences to be the following cyclic permutations:
  \[E_4=\{({}_{\circ}S_{\bullet},{}_{\bullet}D_{\bullet},{}_{\bullet}S_{\circ},{}_{\circ}D_{\circ}),~ ({}_{\bullet}D_{\bullet},{}_{\bullet}S_{\circ},{}_{\circ}D_{\circ},{}_{\circ}S_{\bullet}), ~ ({}_{\bullet}S_{\circ},{}_{\circ}D_{\circ},{}_{\circ}S_{\bullet},{}_{\bullet}D_{\bullet}), ~({}_{\circ}D_{\circ},{}_{\circ}S_{\bullet},{}_{\bullet}D_{\bullet},{}_{\bullet}S_{\circ}) \}\]
  Next, given a set  \(E_{2m}\) of disk sequences of length \(2m\geq4\), the elements of the set \(E_{2m+2}\) are constructed by interposing sequences from $E_{4}$ into any disk sequence in \(E_{2m}\) as follows:
	\begin{align*}
		(\ldots,D^k_\sol,{}_\sol S^\ell,\ldots) &\mapsto (\ldots,D^{k+1}_\sol,{}_\sol S_\hol,{}_\hol D_\hol,{}_\hol S^{\ell+1},\ldots) \\  (\ldots,D^k_\hol,{}_\hol S^\ell,\ldots) &\mapsto (\ldots,D^{k+1}_\hol,{}_\hol S_\sol,{}_\sol D_\sol,{}_\sol S^{\ell+1},\ldots)\\ 
		(\ldots,S^k_\sol,{}_\sol D^\ell,\ldots) &\mapsto (\ldots,S^{k+1}_\hol,{}_\hol D_\hol,{}_\hol S_\sol,{}_\sol D^{\ell+1},\ldots)  \\ (\ldots,S^k_\hol,{}_\hol D^\ell,\ldots) &\mapsto (\ldots,S^{k+1}_\sol,{}_\sol D_\sol,{}_\sol S_\hol,{}_\hol D^{\ell+1},\ldots) 
	\end{align*}
	Now, each disk sequence \((a_1,\ldots,a_{2m}) \in E_{2m}\) defines higher products 
	\[
	\mu^{\Binf}_{2m}(a_1,\ldots,a_{2m}b)=b
	\qquad\text{and}\qquad 
	\mu^{\Binf}_{2m}(b a_1,\ldots,a_{2m})=b
	\]
	for all \(b\in\{\id,S^n,D^n\mid n>1\}\) such that \(a_{2m}b \neq 0\) and \(ba_{1}\neq 0\), respectively. Finally, extend these higher multiplications multilinearly to maps \(\mu^{\Binf}_{2m}:(\Binf)^{\otimes 2m}\rightarrow\Binf\)
  % where the tensor product is taken over $\I=\langle \iota_{\bullet},\iota_{\circ} \rangle$, the idempotent subalgebra.
\end{definition}
The described disk sequences correspond to the geometric disk sequences from~\cite[Section~3.3]{HKK}, and thus the $A_\infty$ relations are satisfied. To illustrate the resulting higher products, all length four and some of the length six products in \(\Binf\) are (suppressing the idempotents)
\[ 
\begin{aligned}
&\mu_4(D,S,D,S^k)=S^{k-1}, \ \mu_4(D^k,S,D,S)=D^{k-1}, \\
&\mu_4(S,D,S,D^k)=D^{k-1}, \ \mu_4(S^k,D,S,D)=S^{k-1}, \\
&\mu_6(S,D,S^2,D,S,D^{k+1})=D^{k-1}, \ \mu_6(S^k,D,S^2,D,S,D^2 )=S^{k-1}, \ldots \\ 
\end{aligned}
\]
for any \(k\geq1\).

The usual grading on \(\Binf\) coming from the Fukaya category is not the one we will need. Thus, we define the following \emph{bigraded} deformation algebras. 

\begin{figure}[t]
	\begin{subfigure}{0.3\textwidth}
		\centering
    \(\FukayaFigureEightArcAlg\)
		\caption{}\label{fig:wrap_subcat}
	\end{subfigure}
	\begin{subfigure}{0.3\textwidth}
		\centering
		\(\LittleSpecialWrapping\)
		\caption{}\label{fig:special_wrapping_before_mcg}
	\end{subfigure}
	\begin{subfigure}{0.3\textwidth}
		\centering
		\(\MuchSpecialWrapping\)
		\caption{}\label{fig:special_wrapping_after_mcg}
	\end{subfigure}
	\caption{Figure (a) shows the arcs \(\arcVer\) and \(\arcHor\) that generate the full subcategory of \(\W(\FourPuncturedSphereKh)\) that corresponds to the algebra \(\B\). Figures (b) and (c) illustrate the proof of Theorem~\ref{thm:no_wrapping_around_special}. As in Figure~\ref{fig:almost_wrapping}, the grey arrow in (b) indicates the direction of the shearing transformation to get from (b) to (c).}\label{fig:special_wrapping}
\end{figure}

\begin{definition}\label{def:deformed_Binf_algebras}
	Define \(A_\infty\) algebras \(\BinfU\) and \(\BinfUU\) to be equal to \(\Binf \otimes \fieldTwoElements[U]\) and \(\Binf \otimes \fieldTwoElements[U,U^{-1}]\) as vector spaces over \(\fieldTwoElements\), with the \(A_\infty\) operations 
	\[
	\mu_k( a_1 \otimes U^{i_1}, a_2 \otimes U^{i_2}, \ldots , a_k \otimes U^{i_k})
	=
	\mu_k^{\Binf}(a_1,\ldots,a_k) \otimes U^{k-2} \cdot U^{i_1 +\cdots +i_k}
	\] 
	Setting 
	\[\text{gr}(U)=q^{-3}h^{-1}, \quad \text{gr}(S)=q^{-1}h^{0}, \quad \text{gr}(D)=q^{-2}h^{0}  \] 
	where the homological grading is defined by \(h=\tfrac 1 2 q -\delta\), the algebras \(\BinfU\) and \(\BinfUU\) become bigraded, that is \(\text{gr}(\mu_k)= q^0 h^{2-k}\).  
\end{definition}
Looking at Figure~\ref{fig:wrap_subcat}, one can think of the algebra \(\BinfU\) as a geometric deformation of the Fukaya category \(\Binf\) of the three-punctured sphere, where every polygon picks up \(U^{2\ell}\) if it covers the special puncture \(\ell\) times; this is reflected in our notation by using the asterisk in the superscript. 

There is an obvious quotient map  \(\BinfU \rightarrow \B\) which sends \(U\) to zero. This induces a functor between the corresponding categories of type D structures over the respective algebras, which allows us to state the following result about the objects $\DD(\Diag_T)$, $\DD_1(\Diag_T)$, $\DD^c(T)$, and $\DD_1^c(T)$ from Section~\ref{sec:review:Kh:definition}:

\begin{theorem}[Extension property]\label{thm:extension}
  Given a diagram $\Diag_T$ of an oriented pointed Conway tangle,
	there exist type~D structures \(\DD(\Diag_T)^{\BinfU}\) and  \(\DD^c(T)^{\BinfU}\) such that
	\[
	\DD(\Diag_T)^{\BinfU}\Big|_{U=0} = \DD(\Diag_T)^{\B} \quad \text{ and } \quad \DD^c(T)^{\BinfU}\Big|_{U=0} = \DD^c(T)^{\B} \]

	The same extension-existence statements hold for \(\DD_1(\Diag_T)^{\B}\) and \(\DD_1^c(T)^{\B}\).
\end{theorem}

From the viewpoint of symplectic geometry, one should think of the result above as saying that the curves \(\BNr(T)\) and \(\Khr(T)\)---corresponding to \(\DD^c(T)\) and \(\DD^c_1(T)\) respectively---are unobstructed as curves in the three-punctured sphere, ie the fishtails enclosing the special puncture cancel out. Before proceeding to the lengthy proof of the extension property, let us explain how it implies the geography result of this section. 

\begin{proof}[Proof of Theorem~\ref{thm:no_wrapping_around_special}]
	Suppose the curve \(\BNr(T)\) has a component that wraps around the special puncture. For simplicity, let us assume that the local system is trivial; the general case follows similarly. We will prove that an extension $\DD^c(T)^{\BinfU}$ of $\DD^c(T)^{\B}$ cannot exist, contradicting Theorem~\ref{thm:extension}. 
	
	Using the naturality of \(\BNr(T)\) under the mapping class group action (Theorem~\ref{thm:Kh:Twisting}), we may assume that one of the two segments where the curve changes direction is horizontal. If the direction changes by less than \(90^\circ\), as for example in Figure~\ref{fig:special_wrapping_before_mcg}, we can add some twists to the bottom two tangle ends such that the angle becomes greater than \(90^\circ\), as in Figure~\ref{fig:special_wrapping_after_mcg}. In the corresponding type~D structure \(\DD^c(T)\), we then see a sequence of differentials
	\begin{equation}\label{eq:sequence}
	\begin{tikzcd}[column sep=15pt]
		\DotB
		\arrow{r}{D}
		&
		\DotB
		\arrow{r}{S}
		&
		\DotC
		\arrow{r}{D}
		&
		\DotC
		\arrow{r}{S}
		&
		\DotB
		\arrow{r}{D}
		&
		\DotB
	\end{tikzcd}
	\end{equation}
	Since \(\mu_4(S,D,S,D)=U^2\), the first four arrows contribute \(U^2\) to the compatibility relation (the analogue of \(d^2=0\)) of any extension \(\DD^c(T)^{\BinfU}\) of \(\DD^c(T)\). % Note: We read the higher multiplications from right to left.
	There are only two other sequences of differentials that may contribute \(U^2\), namely 
	\[
	\begin{tikzcd}[column sep=15pt]
		\DotB
		\arrow{r}{S}
		&
		\DotC
		\arrow{r}{D}
		&
		\DotC
		\arrow{r}{S}
		&
		\DotB
		\arrow{r}{D}
		&
		\DotB
	\end{tikzcd}
	\quad\text{and}\quad
	\begin{tikzcd}[column sep=15pt]
		\DotB
		\arrow{r}{U}
		&
		\DotB
		\arrow{r}{U}
		&
		\DotB
	\end{tikzcd}
	\]
	The first sequence cannot contribute the same \(U^2\) term, since the last arrow in the sequence~\eqref{eq:sequence} points out of the penultimate generator. 
	The second sequence does not appear in \(\DD^c(T)^{\BinfU}\), because the differential of \(\DD^c(T)^{\BinfU}\) contains no component that is labelled by \(U\). 
	This follows from considering quantum gradings: 
	Since \(q(U)=-3\), the quantum gradings of the start and end generators of such a component of the differential would have to differ by 3. 
	However, the quantum gradings of generators in \(\DD^c(T)^\B\) have the same parity if they belong to the same idempotent. 
	This can be seen directly from the cube-of-resolution construction, in the same way that all quantum gradings are even in reduced Khovanov homology of a link. 

  In conclusion, the $U^2$ term in the compatibility relation cannot be canceled, and so the extension \(\DD^c(T)^{\BinfU}\) cannot exist. The same argument works for components of \(\Khr(T)\), since \(\DD_1^c(T)^{\B}\) also extends to some type~D structure over \(\BinfU\) by Theorem~\ref{thm:extension}.
\end{proof}

The rest of Section~\ref{sec:no_wrapping_around_special} is devoted to proving the extension property (Theorem~\ref{thm:extension}). The matrix factorization framework of Khovanov-Rozansky, as well as the homological mirror symmetry for \(\ThreePuncturedSphere\) will play the two central roles.

\subsection{Multifactorizations for tangle diagrams} \label{sec:mf}
Recall that the curve-invariant \(\BNr(T)\) comes from the type~D structure \(\DD(T)^{\B}\), which in turn is equivalent to Bar-Natan's tangle invariant \(\KhTl{T}\). 
A different Khovanov-theoretic tangle invariant was developed by Khovanov and Rozansky in~\cite{KR_mf_I,KR_mf_II}, in the form of a matrix factorization over a certain ring. 
Rasmussen used this construction to obtain the higher differential \(d_{-1}\) on Khovanov homology~\cite{Some_diff}, which---similar to Lee and Bar-Natan \(d_1\) differentials---results in a trivial homology theory. Topological applications of this differential were obtained by Ballinger~\cite{Ballinger}. Our goal is to understand what the existence of the \(d_{-1}\) differential implies for \emph{tangle} invariants; the answer turns out to be precisely the extension property for \(\DD(T) \cong \KhTl{T}\).

We will develop the matrix factorization invariant in the \(sl(2)\) case, closely following Ballinger's setup~\cite{Ballinger} (which, in turn, is the \(n=2\) case of Rasmussen's work~\cite{Some_diff}). In fact, our setups are so similar, that in the interest of brevity we will refer to~\cite{Ballinger} for most of the definitions and results, highlighting only the differences. The main distinction is that we tweak the potential in order to match the reduced version of the Frobenius extension \(\mathcal F_7\) (in the numbering of~\cite{Kh_frob}), while Ballinger works with the Frobenius extension \(\mathcal F_3\) in the unreduced setting. 

Despite the fact that we work over \(\fieldTwoElements\) throughout the paper, we will keep using some signs in our formul\ae, since they make many of the choices more natural. Grading-wise, we will sometimes use the \emph{internal} grading \(i=q-3h\), in addition to the homological and quantum gradings. 
%When shifting and tensoring factorizations, the internal grading will be used in the Koszul sign rule.

The main objects for this subsection are \emph{matrix factorizations} and their generalizations called \emph{multifactorizations} \cite[Definitions~2.1, 2.2]{Ballinger}. The latter should be thought of as filtered chain complexes of matrix factorizations.

\begin{definition}\label{def:diagram}
\textbf{Elementary diagrams} are the diagrams on the left of Figures~\ref{fig:elementary_mf} and~\ref{fig:crossings_mf}. A general \textbf{diagram} is an oriented diagram of a tangle, together with some thick and dotted arcs whose neighbourhoods look like elementary diagrams.
\end{definition}
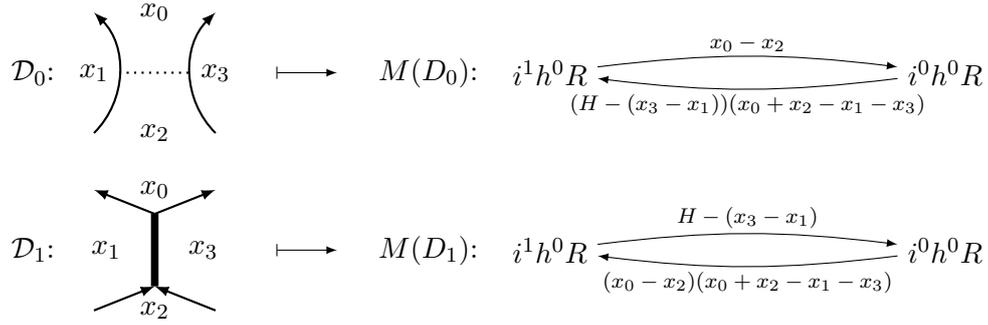
\begin{figure}[t]
\begin{gather*}
\begin{tikzpicture}[scale=0.8]
  % diagram
  \node at (-1,1){\(\Diag_0\):};
  \draw[thick,->] (2,0) arc (225:135:1.4142);
  \draw[thick,->] (0,0) arc (-45:45:1.4142);
  \draw[thick, dotted] (0.4142,1) to (1.5858,1);
  \node[](x0) at (1,2){\(x_0\)};
  \node[](x1) at (0,1){\(x_1\)};
  \node[](x3) at (2,1){\(x_3\)};
  \node[](x2) at (1,0){\(x_2\)};
  % arrow |->
  \draw[|->] (3,1) to (4,1);
  % matrix factorization
  \node at (5.5,1){\(M(D_0)\):};
  \node[](R1) at (7.5,1){\(i^1h^0 R\)};
  \node[](R2) at (14,1){\(i^0 h^0 R\)};
  % bottom d_0
  \draw[->] (R1) to[bend left=7,above] node[tight]{\(x_0-x_2\)} (R2);
  \draw[->] (R2) to[bend left=7,below] node[tight,fill=white]{\((H-(x_3-x_1)) (x_0+x_2-x_1-x_3)\)} (R1);
\end{tikzpicture} 
\\
%\begin{tikzpicture}[scale=0.8]
% % diagram
% \node at (-1,1){\(\Diag'_0\):};
% \draw[thick,<-] (2,0) arc (225:135:1.4142);
% \draw[thick,->] (0,0) arc (-45:45:1.4142);
% \draw[thick, dotted] (0.4142,1) to (1.5858,1);
% \node[](x0) at (1,2){\(x_0\)};
% \node[](x1) at (0,1){\(x_1\)};
% \node[](x3) at (2,1){\(x_3\)};
% \node[](x2) at (1,0){\(x_2\)};
% % arrow |->
% \draw[|->] (3,1) to (4,1);
% % matrix factorization
% \node at (5.5,1){\(M(D'_0)\):};
% \node[](R1) at (7.5,1){\(i^1h^0 R\)};
% \node[](R2) at (14,1){\(i^0 h^0 R\)};
% % bottom d_0
% \draw[->] (R1) to[bend left=7,above] node[tight]{\(x_0-x_2\)} (R2);
% \draw[->] (R2) to[bend left=7,below] node[tight,fill=white]{\((x_3-x_1) (x_0+x_2-x_1-x_3-H)\)} (R1);
%\end{tikzpicture} 
%\\
%\begin{tikzpicture}[scale=0.8]
% % diagram
% \node at (-1,1){\(\Diag''_0\):};
% \draw[thick,->] (2,0) arc (225:135:1.4142);
% \draw[thick,<-] (0,0) arc (-45:45:1.4142);
% \draw[thick, dotted] (0.4142,1) to (1.5858,1);
% \node[](x0) at (1,2){\(x_0\)};
% \node[](x1) at (0,1){\(x_1\)};
% \node[](x3) at (2,1){\(x_3\)};
% \node[](x2) at (1,0){\(x_2\)};
% % arrow |->
% \draw[|->] (3,1) to (4,1);
% % matrix factorization
% \node at (5.5,1){\(M(D''_0)\):};
% \node[](R1) at (7.5,1){\(i^1h^0 R\)};
% \node[](R2) at (14,1){\(i^0 h^0 R\)};
% % bottom d_0
% \draw[->] (R1) to[bend left=7,above] node[tight]{\(x_0-x_2\)} (R2);
% \draw[->] (R2) to[bend left=7,below] node[tight,fill=white]{\((x_3-x_1) (x_0+x_2-x_1-x_3+H)\)} (R1);
%\end{tikzpicture}  
%\\
\begin{tikzpicture}[scale=0.8]
  % diagram
  \node at (-1,1){\(\Diag_1\):};
  \draw[thick,->] (0,0) to (1,0.4);
  \draw[thick,->] (2,0) to (1,0.4);
  \draw[thick,->] (1,1.6) to (0,2);
  \draw[thick,->] (1,1.6) to (2,2);
  \draw[line width=0.1cm] (1,0.4) to (1,1.6);
  \node[](x0) at (1,2){\(x_0\)};
  \node[](x1) at (0.2,1){\(x_1\)};
  \node[](x3) at (1.8,1){\(x_3\)};
  \node[](x2) at (1,0){\(x_2\)};
  % arrow |->
  \draw[|->] (3,1) to (4,1);
  % matrix factorization
  \node at (5.5,1){\(M(D_1)\):};
  \node[](R1) at (7.5,1){\(i^1h^0 R\)};
  \node[](R2) at (14,1){\(i^0 h^0 R\)};
  % bottom d_0
  \draw[->] (R1) to[bend left=7,above] node[tight]{\(H-(x_3-x_1)\)} (R2);
  \draw[->] (R2) to[bend left=7,below] node[tight,fill=white]{\((x_0-x_2) (x_0+x_2-x_1-x_3)\)} (R1);
\end{tikzpicture} 
\end{gather*}
\caption{The matrix factorizations associated with dotted and thick arcs}
\label{fig:elementary_mf} 
\end{figure}

\begin{figure}[t]
\tikzstyle{tight}=[font=\scriptsize, inner sep=1pt, outer sep=1pt]
\begin{gather*}
\begin{tikzpicture}[scale=0.8]
  % diagram
  \node at (-1,1){\(\Diag_+\):};
  \draw[thick,->] (0,0) to (2,2);
  \draw[thick] (2,0) to (1.1,0.9);
  \draw[thick,->] (0.9,1.1) to (0,2);
  \node(x0) at (1,1.8){\(x_0\)};
  \node(x1) at (0.2,1){\(x_1\)};
  \node(x3) at (1.8,1){\(x_3\)};
  \node(x2) at (1,0.2){\(x_2\)};
  % arrow |->
  \draw[|->] (3,1) to (4,1);
  % matrix factorization
  \node at (5.5,1){\(M(D_+)\):};
  \node(R1) at (7.5,2.75){\(i^2 h^0R\)};
  \node(R2) at (14,2.75){\(i^1h^0 R\)};
  \node(R3) at (7.5,-0.75){\(i^0h^1R\)};
  \node(R4) at (14,-0.75){\(i^{-1}h^1R\)};
  % bottom d_0
  \draw[->] (R3) to[bend left=7,above] node[tight]{\(H-(x_3-x_1)\)} (R4);
  \draw[->] (R4) to[bend left=7,below] node[tight]{\((x_0-x_2)(x_0+x_2-x_1-x_3)\)} (R3);
  % diagonals d_1
  \draw[->] (R1) to[out=-90,in=90,looseness=0.5] node[tight,fill=white,pos=0.7]{\(-\id\)} (R4);
  \draw[->] (R2) to[out=-90,in=90,looseness=0.5] node[tight,fill=white,pos=0.7]{\(x_0+x_2-x_1-x_3\)} (R3);
  % top d_0
  \draw[->] (R1) to[bend left=7,above] node[tight]{\(x_0-x_2\)} (R2);
  \draw[->] (R2) to[bend left=7,below] node[tight,fill=white]{\((H-(x_3-x_1))(x_0+x_2-x_1-x_3)\)} (R1);
\end{tikzpicture} 
\\
\begin{tikzpicture}[scale=0.8]
  % diagram
  \node at (-1,1){\(\Diag_-\):};
  \draw[thick,->] (2,0) to (0,2);
  \draw[thick] (0,0) to (0.9,0.9);
  \draw[thick,->] (1.1,1.1) to (2,2);
  \node(x0) at (1,1.8){\(x_0\)};
  \node(x1) at (0.2,1){\(x_1\)};
  \node(x3) at (1.8,1){\(x_3\)};
  \node(x2) at (1,0.2){\(x_2\)};
  % arrow |->
  \draw[|->] (3,1) to (4,1);
  % matrix factorization
  \node at (5.5,1){\(M(D_-)\):};
  \node(R1) at (7.5,2.75){\(i^2 h^{-1}R\)};
  \node(R2) at (14,2.75){\(i^1h^{-1} R\)};
  \node(R3) at (7.5,-0.75){\(i^0h^0R\)};
  \node(R4) at (14,-0.75){\(i^{-1}h^0R\)};
  % bottom d_0
  \draw[->] (R3) to[bend left=7,above] node[tight]{\(x_0-x_2\)} (R4);
  \draw[->] (R4) to[bend left=7,below] node[tight]{\((H-(x_3-x_1))(x_0+x_2-x_1-x_3)\)} (R3);
  % diagonals d_1
  \draw[->] (R1) to[out=-90,in=90,looseness=0.5] node[tight,fill=white,pos=0.7]{\(-\id\)} (R4);
  \draw[->] (R2) to[out=-90,in=90,looseness=0.5] node[tight,fill=white,pos=0.7]{\(x_0+x_2-x_1-x_3\)} (R3);
  % top d_0
  \draw[->] (R1) to[bend left=7,above] node[tight]{\(H-(x_3-x_1)\)} (R2);
  \draw[->] (R2) to[bend left=7,below] node[tight,fill=white]{\((x_0-x_2)(x_0+x_2-x_1-x_3)\)} (R1);
\end{tikzpicture}  
\end{gather*}
\caption{The multifactorizations associated with crossings
}
\label{fig:crossings_mf} 
\end{figure}
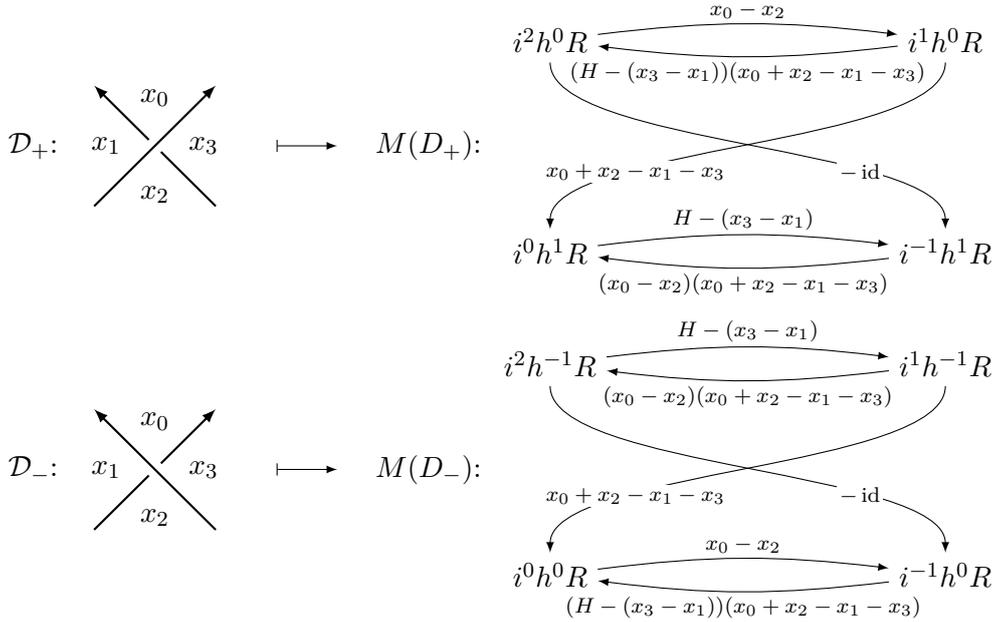

Examples of general diagrams are depicted in Figure~\ref{fig:general_diagram}. Setting 
\begin{equation}\label{eq:base_ring}
R=\fieldTwoElements[H,x_0-x_1,x_0-x_2,x_0-x_3], \quad \text{gr}(H)=\text{gr}(x_i)=i^{-2}h^0q^{-2}
\end{equation}
we first assign multifactorizations over \(R\) to elementary diagrams according to Figures~\ref{fig:elementary_mf} and~\ref{fig:crossings_mf}. 
In these multifactorizations, the horizontal arrows indicate the differential \(d_0\), which preserves homological grading and decreases both internal and quantum gradings by 3; in short, \(\text{gr}(d_0)=i^{-3}h^0q^{-3}\). The non-horizontal arrows in the multifactorizations \(M(D_+)\) and \(M(D_-)\) indicate the differential \(d_1\), whose grading is given by \(\text{gr}(d_1)=i^{-3}h^1q^{0}\). 
Note that \(M(D_+)\) and \(M(D_-)\) can be viewed as mapping cones
\begin{equation}\label{eq:mcs}
M(D_+)=\left[i^1 h^0 M(D_0) \xrightarrow{d_1} i^{-1} h^{1} M(D_1)\right]
\qquad 
M(D_-)=\left[ i^1 h^{-1} M(D_1) \xrightarrow{d_1}i^{-1} h^{0} M(D_0)\right]
\end{equation}

\begin{figure}[t]
	\begin{subfigure}{0.3\textwidth}
		\centering
		\labellist 
		\pinlabel \(\scriptstyle x_0\) at 37 52
		\pinlabel \(\scriptstyle x_2\) at 37 5
		\pinlabel \(\scriptstyle x_4\) at 36 29
		\pinlabel \(\scriptstyle x_5\) at 23 29
		\pinlabel \(\scriptstyle x_6\) at 45 18
		\pinlabel \(\scriptstyle x_1\) at 9 29
		\pinlabel \(\scriptstyle x_3\) at 58 29
		\endlabellist
		\includegraphics[scale=1.2]{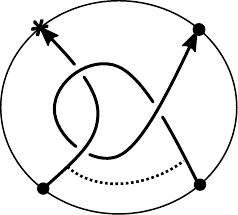}
		\caption{}\label{fig:general_diagram_tangle}
	\end{subfigure}
	\begin{subfigure}{0.3\textwidth}
		\centering
		\labellist 
		\pinlabel \(\scriptstyle x_0\) at 37 52
		\pinlabel \(\scriptstyle x_2\) at 37 4
		\pinlabel \(\scriptstyle x_4\) at 36 31
		\pinlabel \(\scriptstyle x_5\) at 23 29
		\pinlabel \(\scriptstyle x_6\) at 45 18
		\pinlabel \(\scriptstyle x_1\) at 9 29
		\pinlabel \(\scriptstyle x_3\) at 58 29
		\endlabellist
		\includegraphics[scale=1.2]{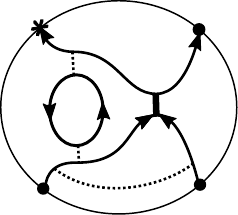}
		\caption{}\label{fig:general_diagram_resolution}
	\end{subfigure}
	\caption{A diagram of a tangle (a) and its $(1,1,1)$-resolution (b)}\label{fig:general_diagram}
\end{figure}
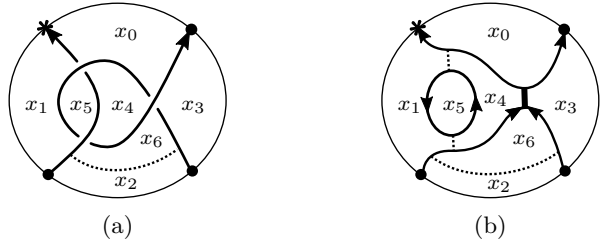

One can check that, for each elementary multifactorization from Figures~\ref{fig:elementary_mf} and~\ref{fig:crossings_mf}, the potential can be computed using the following formula:
\begin{equation}\label{eq:potential}
d_0^2=w\cdot \id, \qquad w=\sum_{e \in \partial\Diag} \varepsilon_e \left(\frac{X_e^3}{3} - H \frac{X_e^2}{2}\right) \in R
\end{equation}
Here, the sum is over all oriented edges \(e \in \partial\Diag\) on the boundary of the diagram, the \emph{edge variable} \(X_e=x_r-x_\ell\) is the difference between the variable  to the right (\(x_r\)) and left (\(x_\ell\)) of the edge \(e\), and \(\varepsilon_e=-1\) if the edge \(e\) points into the diagram and \(\varepsilon_e=+1\) if \(e\) points out of the diagram. Note that despite the denominators, the potential is indeed a linear combination of monomials in the variables \(x_i\) with integer coefficients.

Given a Conway tangle \(T\), fix a tangle diagram \(\Diag_T\). 
\begin{assumptions}\label{assums:tangle_diagram}
We will make the following assumptions: 
  \begin{itemize}
  \item All regions of  \(\Diag_T\) are labeled by unique variables \(x_i\), such that \(x_0\), \(x_1\), \(x_2\), and \(x_3\) are the open regions at the top, left, bottom, and right, respectively.
  \item The diagram \(\Diag_T\) is oriented such that the top left and top right are the outward pointing tangle ends, as in Figure~\ref{fig:general_diagram_tangle}. This can be achieved by changing the orientation on components and rotating \(\Diag_T\) by $90^\circ$.
  \item The diagram \(\Diag_T\) is connected in the sense that the projection of $T$ is a connected planar graph. This can be achieved by applying some Reidemeister moves.
  \item There is a single dotted arc that is parallel to the boundary of the region labelled \(x_2\).
  \item The region \(x_3\) is adjacent to a single crossing of the tangle, which can be achieved by applying the Reidemeister~II move. 
\end{itemize}
\end{assumptions}
An example of a diagram \(\Diag_T\) meeting all of the assumptions is given in Figure~\ref{fig:general_diagram_tangle}. The last two assumptions ensure that the open components in every crossingless resolution of \(\Diag_T\) are connected by some thick or dotted arc, which will simplify the proof of Lemma~\ref{lem:delooping:mf} below. Note that all resolutions of crossings inherit additional arcs, either thick (if the resolution is unoriented) or dotted (if the resolution is oriented). For example, the $(1,1,1)$ full resolution of \(\Diag_T\) from Figure~\ref{fig:general_diagram_tangle} is depicted in  Figure~\ref{fig:general_diagram_resolution}. 

\begin{definition}\label{def:mf_of_tangle}
	Define the \textbf{edge ring of \(\Diag_T\)} as the subring of the free polynomial ring in \(H\) and all \(x_i\) that is generated over \(\fieldTwoElements[H]\) by the differences of all variables:
	\[
		R(\Diag_T)
		\coloneqq
		\fieldTwoElements[H][\{x_i-x_j\}_{x_i,x_j \text{  regions in }\Diag_T}]
	\]
	Define the multifactorization \(M(\Diag_T)\) over \(R(\Diag_T)\) as
	\[
		M(\Diag_T) = \bigotimes_\mathrm{v} M(\Diag_\mathrm{v}) 
	\]
	where the tensor product is taken over \(R(\Diag_T)\), the index \(\mathrm{v}\) runs through crossings and arcs in  \(\Diag_T\), and \(\Diag_\mathrm{v}\) is the corresponding local elementary diagram from Figures~\ref{fig:elementary_mf} and~\ref{fig:crossings_mf} with \(R=R(\Diag_T)\) and labels changed accordingly.
\end{definition}

\begin{remark}\label{rem:mf:cube_of_resolutions}
	Similarly, we can define matrix factorizations of complete resolutions \(\Diag=\Diag_T(v)\) for \(v\in\{0,1\}^n\); we set
	\[
	M(\Diag) = \bigotimes_a M(\Diag_a) 
	\]
	where the index \(a\) runs through all arcs in \(\Diag\), and \(\Diag_a\) is the corresponding local elementary diagram \(\Diag_0\) or \(\Diag_1\) from Figure~\ref{fig:elementary_mf} with \(R=R(\Diag_T)\).
	Since \(M(D_+)\) and \(M(D_-)\) are mapping cones of these elementary diagrams, we can regard \(M(\Diag_T)\) as a cube of resolutions, at whose vertices there are matrix factorizations of diagrams with dotted and thick edges, but no crossings. 
\end{remark}

The potential of matrix factorizations is additive in the sense that for any two matrix factorizations \(M\) and \(M'\) with potentials \(w\) and \(w'\), respectively, the potential of \(M \otimes M'\) is given by \(w+w'\). Therefore, the potential of \(M(\Diag_T)\) is also given by formula~\eqref{eq:potential}. Note that this potential is an element of the \textbf{boundary edge ring} 
\[
R^\partial
=
\fieldTwoElements[H][x_0-x_1,x_0-x_2,x_0-x_3] 
%\cong 
%\fieldTwoElements[H][\{X_e\}_{e \in \partial\Diag_T}] / \left(\sum_{e\in\partial\Diag_T}X_e = 0\right)
\]
As such, we can consider \(M(\Diag_T)\) as a multifactorization over \(R^\partial\). 

\begin{definition}
	Given a diagram of a Conway tangle \(\Diag_T\)  meeting the Assumptions~\ref{assums:tangle_diagram}, recall that \(x_0\) and \(x_1\) are the labels of the two regions adjacent to the tangle end at the top left. We mark this tangle end by an asterisk \(\ast\) (see Figure~\ref{fig:general_diagram_tangle}), and define the reduced edge and boundary edge rings \(\tR(\Diag_T)\) and \(\tR^\partial=\tR^\partial(\Diag_T)\) by imposing the relation \(x_{0}=x_1\) and define the multifactorization associated with \(\Diag_T\) by
	\[
		\tM(\Diag_T)
		\coloneqq
		M(\Diag_T)\otimes_{R(\Diag_T)} \tR(\Diag_T)
	\]
	Again, we can consider this as a matrix factorization over \(\tR^\partial=\fieldTwoElements[H][x_0-x_2,x_0-x_3]\). 
%	We may need to change to $H-(x_0 - x_0)=0$.
\end{definition}

\begin{example}\label{exa:elementary_mf}
	We can change the basis in the reduced boundary edge ring so that
	\[\tR^\partial=\fieldTwoElements[x,y,z] \quad \text{ where } \quad
		x=x_0-x_2
		\qquad
		y=H-(x_3-x_0)
		\qquad
		z= x_2-x_3 
	\]
	In particular, \(H=x+y+z\in\tR^\partial\). 
	Note that the edge ring for the elementary diagrams from Figure~\ref{fig:elementary_mf} agrees with the boundary edge ring \(\tR(\Diag_0)=\tR^\partial,~\tR(\Diag_1)=\tR^\partial\). With respect to the new basis for \(\tR^\partial\), the reduced matrix factorizations can be rewritten as follows: 
	\[
		\tM(\Diag_0)=
		\begin{tikzcd}[column sep=1cm, ampersand replacement=\&]
			q^1h^0 \tR^\partial
			\arrow[r,bend left=10, "x"above]
			\&
			\tR^\partial
			\arrow[l,bend left=10, "yz"below]
		\end{tikzcd}
		\qquad
		\tM(\Diag_1)=
		\begin{tikzcd}[column sep=1cm, ampersand replacement=\&]
			q^1 h^0\tR^\partial
			\arrow[r,bend left=10, "y"above]
			\&
			\tR^\partial
			\arrow[l,bend left=10, "xz"below]
		\end{tikzcd}
	\]
\end{example}

We now begin to set the stage for the main result of this section, Lemma~\ref{lem:delooping:mf}, which describes the matrix factorization $\tM(\Diag_T)$ over $\tR^\partial$. 

Given two matrix factorizations \(M\) and \(M'\) over \(R\) with the same potential, we can consider the space of morphisms \(\Mor(M,M')\). By morphism, we mean any \(R\)-linear map between the underlying modules of \(M\) and \(M'\), irrespective of the differentials on \(M\) and \(M'\). \(\Mor(M,M')\) can be equipped with a differential defined by \(\partial(f)=(d_0)_{M'} \circ f + f\circ (d_0)_{M}\). Grading preserving morphisms in the kernel of this differential define \emph{homomorphisms} from \(M\) to \(M'\). 
Two maps are called \emph{homotopic} $f\simeq g$ if $f-g=\partial(h)$. Two matrix factorizations are called \emph{homotopy equivalent} $M\simeq M'$ if there exist homomorphisms $f\co M\to M'$ and $g\co M'\to M$ such that $f\circ g\simeq \id_{M'}$ and $g\circ f\simeq \id_{M}$. 

\begin{definition}\label{def:dg_algebra}
  Consider the two elementary matrix factorizations \(\tM(\Diag_0)\) and \(\tM(\Diag_1)\) from Example~\ref{exa:elementary_mf}.
  Let \(\A\) be their endomorphism dg algebra defined by
  \[
  \A=\End_{\tR^\partial}( \tM(\Diag_0) \oplus \tM(\Diag_1))=\bigoplus_{0\leq i,j \leq 1}\Mor_{\tR^\partial}(\tM(\Diag_i),\tM(\Diag_j)), \quad d^{\A}(f)=f\circ d_0 + d_0 \circ f
  \]
  The bigrading on the elements of \(\A\) is defined by \(\text{gr}(x \xrightarrow{a} y)=\text{gr}(y)+\text{gr}(a)-\text{gr}(x)\). The sum \(\id_{\tM(\Diag_0)} + \id_{\tM(\Diag_1)}\) of the two indecomposable idempotents is a unit. 
\end{definition}

The dg algebra \(\A\) is not bigraded in the usual sense, because \(\text{gr}(d^{\A})=q^{-3}h^0\), while in Khovanov homology, the grading of differentials is usually \(q^{0}h^1\). We will return to this issue in Definition~\ref{def:deformed_dg_algebra}.

\begin{lemma}\label{lem:iso_A_B}
  The elements of \(\A\) that are shown in Figure~\ref{fig:saddles_and_dots_in_A} represent algebra generators of \(\Hast(\A)\). Moreover, there is an isomorphism of \(\fieldTwoElements[H]\)-algebras 
  \(
  \Phi\co
  \Hast(\A)
  \rightarrow
  \B
  \)
  which is uniquely determined by
  \begin{align*}
    \id_{\tM(\Diag_0)} 
    &
    \mapsto 
    \iota_{\circ}
    &
    \id_{\tM(\Diag_1)} 
    &
    \mapsto 
    \iota_{\bullet}
    &
    s_{0}
    &
    \mapsto
    S_{\circ}
    &
    s_{1}
    &
    \mapsto
    S_{\bullet}
    &
    d_{0}
    &
    \mapsto
    D_{\circ}
    &
    d_{1}
    &
    \mapsto
    D_{\bullet}
  \end{align*}
\end{lemma}

\begin{proof}
  This follows from a routine computation using the \(H\)-action on both sides. 
\end{proof}

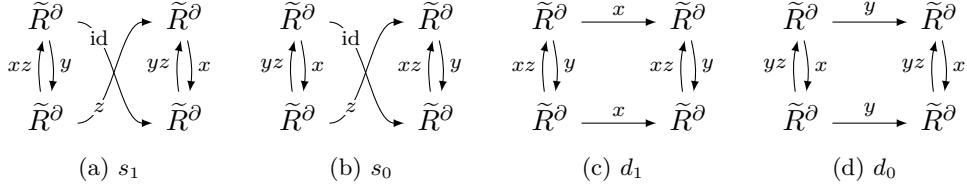
\begin{figure}[H]
  \centering
  \begin{subfigure}{0.2\textwidth}
    \centering
    \(
    \begin{tikzcd}[column sep=1cm, ampersand replacement=\&]
      \tR^\partial
      \arrow[d,bend left=10, "y"right]
      \arrow[in=180,out=0,looseness=0.7]{rd}[inner sep=1pt, outer sep=1pt, description,fill=white, near start]{\id}
      \&
      \tR^\partial
      \arrow[d,bend left=10, "x"right]
      \\
      \tR^\partial
      \arrow[u,bend left=10, "xz"left]
      \arrow[in=180,out=0,looseness=0.7]{ru}[inner sep=1pt, outer sep=1pt, description,fill=white, near start]{z}
      \&
      \tR^\partial
      \arrow[u,bend left=10, "yz"left]
    \end{tikzcd}
    \)
    \caption{\(s_{1}\)}
  \end{subfigure}
  \begin{subfigure}{0.2\textwidth}
    \centering
    \(
    \begin{tikzcd}[column sep=1cm, ampersand replacement=\&]
      \tR^\partial
      \arrow[d,bend left=10, "x"right]
      \arrow[in=180,out=0,looseness=0.7]{rd}[inner sep=1pt, outer sep=1pt, description,fill=white, near start]{\id}
      % \arrow[rd, "\id"above, near start]
      \&
      \tR^\partial
      \arrow[d,bend left=10, "y"right]
      \\
      \tR^\partial
      \arrow[u,bend left=10, "yz"left]
      \arrow[in=180,out=0,looseness=0.7]{ru}[inner sep=1pt, outer sep=1pt, description,fill=white, near start]{z}
      \&
      \tR^\partial
      \arrow[u,bend left=10, "xz"left]
    \end{tikzcd}
    \)
    \caption{\(s_{0}\)}
  \end{subfigure}
  \begin{subfigure}{0.2\textwidth}
    \centering
    \(
    \begin{tikzcd}[column sep=1cm, ampersand replacement=\&]
      \tR^\partial
      \arrow[d,bend left=10, "y"right]
      \arrow[r, "x"above]
      \&
      \tR^\partial
      \arrow[d,bend left=10, "y"right]
      \\
      \tR^\partial
      \arrow[u,bend left=10, "xz"left]
      \arrow[r, "x"above]
      \&
      \tR^\partial
      \arrow[u,bend left=10, "xz"left]
    \end{tikzcd}
    \)
    \caption{\(d_{1}\)}
  \end{subfigure}
  \begin{subfigure}{0.2\textwidth}
    \centering
    \(
    \begin{tikzcd}[column sep=1cm, ampersand replacement=\&]
      \tR^\partial
      \arrow[d,bend left=10, "x"right]
      \arrow[r, "y"above]
      \&
      \tR^\partial
      \arrow[d,bend left=10, "x"right]
      \\
      \tR^\partial
      \arrow[u,bend left=10, "yz"left]
      \arrow[r, "y"above]
      \&
      \tR^\partial
      \arrow[u,bend left=10, "yz"left]
    \end{tikzcd}
    \)
    \caption{\(d_{0}\)}
  \end{subfigure}
  \caption{Four elements of \(\A\) generating the algebra \(\Hast(\A)\)}\label{fig:saddles_and_dots_in_A}
\end{figure}

We will be manipulating the multifactorization \(\tM(\Diag_T)\) over \(\tR^\partial\) using a notion of a \emph{special deformation retract}~\cite[Definition~2.6]{Ballinger}. This is a special case of a \emph{0-homotopy equivalence}, which is defined using the notion of \emph{chain maps between multifactorizations}~\cite[Definition~2.3]{Ballinger} and the notion of \emph{0-homotopy} between the chain maps~\cite[Definition~2.4]{Ballinger}. The reason to work with special deformation retracts is the ability to ‘‘extend'' them using a version of the homological perturbation lemma~\cite[Proposition~2.7]{Ballinger}.

\begin{remark}\label{prop:matrix_factorization_is_an_invariant}
The multifactorization \(\tM(\Diag_T)\) over \(\tR^\partial\) depends not only on the tangle \(T\), but also on the diagram \(\Diag_T\). To obtain a tangle invariant, it is natural to consider multifactorizations up to \emph{1-homotopy equivalence}, a notion which is slightly weaker than the 0-homotopy equivalence, see~\cite[Definition~2.4]{Ballinger}. In analogy with~\cite[Proposition~3.3]{Ballinger}, we expect that the 1-homotopy equivalence class of the multifactorization \(\tM(\Diag_T)\) over \(\tR^\partial\) is a tangle invariant, ie does not depend on the choice of diagram for \(T\). For the purposes of this paper the invariance of \(\tM(\Diag_T)\) is not needed, and thus we choose not to pursue this direction.
\end{remark}

The following lemmas, both of which are consequences of~\cite[Theorem~2.2]{KR_convolutions}, will be our sources of special deformation retracts. For the term \emph{Koszul factorization} see~\cite[Section~2.2]{Ballinger}.

\begin{lemma}\label{lem:lin_sp_def_retr}
Suppose a Koszul matrix factorization \(K((a_1,\ldots,a_n),(b_1,\ldots,b_n))\) over the polynomial ring \(R=R_0[x]\) has potential \(W-q x  \) where 
\(W\in R_0,~q\in R,~\text{gr}(x)=i^{-2}h^0q^{-2},~\text{gr}(q)=i^{-4}h^0q^{-4}\). Then there is a special deformation retract
\[K((a_1,\ldots,a_n,x),(b_1,\ldots,b_n,q)) \rightarrow
K\left((a_1\big|_{x=0},\ldots,a_n\big|_{x=0}),(b_1\big|_{x=0},\ldots,b_n\big|_{x=0})\right)
\]
where the matrix factorizations are considered over \(R_0\) with potential \(W \). In fact, the map is given by the horizontal quotient homomorphism below
\[
\begin{tikzcd}
i^{1}h^0  K((a_1,\ldots,a_n),(b_1,\ldots,b_n))
\arrow[bend left=10]{d}{x}
\\
K((a_1,\ldots,a_n),(b_1,\ldots,b_n))
\arrow[in=180,out=0]{r}{}
\arrow[bend left=10]{u}{q}
&
K\left((a_1\big|_{x=0},\ldots,a_n\big|_{x=0}),(b_1\big|_{x=0},\ldots,b_n\big|_{x=0})\right)
\end{tikzcd}
\]
\end{lemma}

\begin{lemma}\label{lem:quad_sp_def_retr}
Suppose a Koszul matrix factorization \(K((a_1,\ldots,a_n),(b_1,\ldots,b_n))\) over the polynomial ring \(R=R_0[x]\) has potential 
\(W\in R_0\). 
Let \(p(x)\in R_0[x]\) be a homogeneous monic polynomial of degree 2, say \(p(x)=(x^2 + C x + D)\), with \(\text{gr}(x)=\text{gr}(C)=i^{-2}h^0q^{-2},~\text{gr}(D)=i^{-4}h^0q^{-4}\). 
Then there is a special deformation retract 
\[
K((a_1,\ldots,a_n,0),(b_1,\ldots,b_n,p(x))) \to
K((a_1,\ldots,a_n),(b_1,\ldots,b_n)) \otimes i^{1}h^0 R_0[x]/(p(x))
\]
where the matrix factorizations are considered over \(R_0\) with potential \(W\).  
In fact, this map is given by the horizontal quotient homomorphism below
\[
\begin{tikzcd}
i^{1}h^0 K((a_1,\ldots,a_n),(b_1,\ldots,b_n))
\arrow[in=180,out=0]{r}{}
&
K(a_1,\ldots,a_n),(b_1,\ldots,b_n)) \otimes_{R_0[x]} i^{1}h^0 R_0[x]/(p(x))
\\
K((a_1,\ldots,a_n),(b_1,\ldots,b_n))
\arrow{u}{p(x)}
\end{tikzcd}
\]
\end{lemma}

\begin{definition}\label{def:induced_saddle_maps:mf}
	Given a circle \(c\) in the plane, we define
	\[
	V(c)= i^1 h^0 \fieldTwoElements[H,y]/(y^2 - H y)
	\]
	where \(\text{gr}(y)=i^{-2}h^0q^{-2}\). 
	Given a crossingless diagram \(\Diag\), such as the one in Figure~\ref{fig:general_diagram_resolution}, with \(n\) closed components \(c_1, \dots, c_n\), define
	\[
	V(\Diag)
	\coloneqq
	\bigotimes_{i=1}^n V(c_i)
	=
	i^{n}h^0\fieldTwoElements[H,y_1,\ldots,y_n]/(y_i^2 - H y_i)
	\]
	where the tensor product is taken over \(\fieldTwoElements[H]\). Recalling the elementary diagrams \(\Diag_{0}\) and \(\Diag_1\) from Figure~\ref{fig:elementary_mf}, define 
  \[\begin{aligned}
  \varepsilon(\Diag)&\coloneqq \begin{cases}
  0\text{ if the open strands of \(\Diag\) are connected as in \(\Diag_{0}\) }\\
  1\text{ if the open strands of \(\Diag\) are connected as in \(\Diag_{1}\) }
  \end{cases} \\
  o(\Diag) &\coloneqq \Diag_{\varepsilon(\Diag)}
  \end{aligned}
  \] 
	Let \(\Diag'\) be another diagram with closed components \(c'_1,\dots, c'_{m}\) for some integer \(m\). 
	We pick the basis on 
	\begin{equation}\label{eqn:mf:only_circles}
		\Mor_{\fieldTwoElements[H]}\big(V(\Diag), V(\Diag')\big)
		= 
		\Big(
		\bigotimes_{i=1}^{n}
		V^*(c_i)
		\Big)
		\otimes_{\fieldTwoElements[H]}
		\Big(
		\bigotimes_{j=1}^{m}
		V(c'_j)
		\Big)
	\end{equation}
	that is induced by the standard basis on each of tensor factor for \(i=1,\dots,n\) and \(j=1,\dots,m\):
	\begin{align*}
		V^*(c_i)
		= 
		\fieldTwoElements[H]
		\Big\langle
		1^*_i,y^*_i
		\Big\rangle 
		\qquad \quad
		V(c_j)
		= 
		\fieldTwoElements[H]
		\Big\langle
		1_j,y_j
		\Big\rangle
	\end{align*}
	If \(\Diag'\) differs from \(\Diag\) in a single dotted/thick arc, we define a map of multifactorizations over $\tR^\partial$
	\[
	\mathcal{S}(\Diag,\Diag')
	\co
	V(\Diag)
	\otimes_{\fieldTwoElements[H]}
	\tM(o(\Diag))
	\rightarrow
	V(\Diag')
	\otimes_{\fieldTwoElements[H]}
	\tM(o(\Diag'))
	\]
	as follows. If \((\varepsilon(\Diag),\varepsilon(\Diag'))=(0,1)\) or \((1,0)\), \(V(\Diag)=V(\Diag')\), we respectively define
	\[
	\mathcal{S}(\Diag,\Diag')=\id\otimes s_{0}
	\quad\text{or}\quad
	\mathcal{S}(\Diag,\Diag')=\id\otimes s_{1}
	\]
	Suppose \(\varepsilon(\Diag)=\varepsilon(\Diag')\). 
	Let \(o^*\) denote the special open component of \(\Diag\) and \(o\) the non-special one. 
	We also write 
	$I_{\bcancel{i},\bcancel{j}}\coloneqq\sum_{\ell\neq i,j} 1^*_\ell 1_\ell+y^*_\ell y_\ell$ and \(I_{\bcancel{i}}\coloneqq\sum_{\ell\neq i} 1^*_\ell 1_\ell+y^*_\ell y_\ell\). 
	Then, if \(\Diag'\) is obtained from \(\Diag\) by merging two components, we define
	\[
	\mathcal{S}(\Diag,\Diag') \coloneqq
	\begin{cases*}
		I_{\bcancel{i},\bcancel{j}} \big(
		1_i^*1_j^*1_k + y_i^*1_j^*y_k + 1_i^*y_j^*y_k +H\cdot y_i^*y_j^*y_k
		\big) 
		\otimes \id
		&
		if \(c_i\) and \(c_j\) merge to \(c_k\)
		\\
		 I_{\bcancel{i}} 1_i^* \otimes \id + I_{\bcancel{i}} y_i^*\otimes s_{\varepsilon(\Diag)+1}s_{\varepsilon(\Diag)}
		&
		if \(c_i\) merges with \(o\)
%		\\
%		I_{\bcancel{i}} 1_i^* \otimes \id + I_{\bcancel{i}} y_i^*\otimes s_{0}s_{1} 
%		&
%		if \(\varepsilon(\Diag)=1\) and \(c_i\) merges with \(o\)
		\\
		I_{\bcancel{i}} 1_i^* \otimes \id + H\cdot  I_{\bcancel{i}} y_i^*\otimes \id 
		&
		if \(c_i\) merges with \(o^*\)
	\end{cases*}
	\]
	where \(s_2\coloneqq s_0\). 
	If \(\Diag'\) is obtained from \(\Diag\) by splitting two components, we define 
	\[
	\mathcal{S}(\Diag,\Diag') \coloneqq
	\begin{cases*}
		 I_{\bcancel{k}} \big(
		1_k^*y_i1_j + 1_k^*1_iy_j + H\cdot 1_k^*1_i1_j+y_k^* y_iy_j
		\big)
		\otimes \id
		&
		if \(c_k\) splits into \(c_i\) and \(c_j\)
		\\
    I_{\bcancel{i}} 
		1_i\otimes d_{\varepsilon(\Diag)}+ I_{\bcancel{i}} y_i\otimes \id 
		&
		if \(c_i\) splits from \(o\)
%		\\
%		I_{\bcancel{i}} 1_i\otimes d_{1}+ I_{\bcancel{i}} y_i\otimes \id 
%		&
%		if \(\varepsilon(\Diag)=1\) and \(c_i\) splits from \(o\)
		\\
		I_{\bcancel{i}} y_i\otimes \id 
		&
		if \(c_i\) splits from \(o^*\)
	\end{cases*}
	\]	
	
\end{definition}

The following lemma describes a process known as delooping, which allows one to eliminate closed components of a crossingless tangle diagram, such as the one in Figure~\ref{fig:general_diagram_resolution}, at the expense of tensoring with a two-dimensional vector space. 
This lemma is well-known in the setting of closed link diagrams, however we require both a version for tangles as well as some control over morphisms induced by this process. 
The latter represents the bulk of the work; the proof is the most technically involved in this section and may be safely skipped on the first read. 

\begin{lemma}\label{lem:delooping:mf}
	Let \(\Diag_T\) be a diagram of a Conway tangle as in the Assumptions~\ref{assums:tangle_diagram} and \(\Diag\) some complete resolution of \(\Diag_T\). Then there exists a special deformation retract
	\[
	\varphi_\Diag
	\co 
	\tM(\Diag)
	\rightarrow
	V(\Diag)
	\otimes_{\fieldTwoElements[H]}
	\tM(o(\Diag))
	\]
  of matrix factorizations over $\tR^\partial$.

	Furthermore, if \(\Diag'\) is another complete resolution of \(\Diag_T\) which differs from \(\Diag\) in a single dotted/thick arc, and \(f\co\tM(\Diag)\rightarrow\tM(\Diag')\) is the map induced by the vertical maps in \(M(\Diag_+)\) or \(M(\Diag_-)\) (see Figure~\ref{fig:crossings_mf}), then the induced map
	\[
	V(\Diag)
	\otimes_{\fieldTwoElements[H]}
	\tM(o(\Diag))
	\rightarrow
	V(\Diag')
	\otimes_{\fieldTwoElements[H]}
	\tM(o(\Diag'))
	\]
	is equal to \(\mathcal{S}(\Diag,\Diag')\) up to homotopy of maps between matrix factorizations over $\tR^\partial$.
\end{lemma}

\newcommand{\mybfheading}[1]{\medskip\noindent\textbf{#1.}}%
\newcommand{\myitheading}[1]{\medskip\noindent\textit{#1:}}%

\begin{proof}
	The proof of Lemma~\ref{lem:delooping:mf} proceeds in three steps: First, we construct a candidate \(\varphi_A\) for the special deformation retract \(\varphi_\Diag\) which depends on some additional data \(A\) on \(\Diag\), which we call arc system. In the second step, we show that while the maps \(\varphi_{A_1}\) and \(\varphi_{A_2}\) are indeed distinct for different arc systems \(A_1\) and \(A_2\) on \(\Diag\), their homotopy classes agree. In the last step, we show the second part of the lemma, using choices of arc systems of the diagrams \(\Diag\) and \(\Diag'\) that are adapted to the map \(f\). 	
	
	\mybfheading{0) Notation and preliminary observations}
	Before we start describing the construction of the maps \(\varphi_A\), we introduce some notation that we will be using throughout this proof. The symbol $\tR$ will be used for the reduced edge ring $\tR(\Diag)$.	Given any arc \(a\) in \(\Diag\), we define 
	\[
	\lambda(a)\coloneqq
	\begin{cases*}
	(x^a_0-x^a_2)
	&
	if \(\Diag_a = \Diag_0\)
	\\
	H-(x^a_3-x^a_1)
	&
	if \(\Diag_a = \Diag_1\)
	\end{cases*}
	\]
	where \(x^a_{i}\) for \(i=1,2,3,4\) are the labels of the local regions for \(\Diag_a\). 
	Similarly, we define
	\[
	\kappa(a)\coloneqq
	\begin{cases*}
	(H-(x^a_3-x^a_1))(x^a_0+x^a_2-x^a_1-x^a_3)
	&
	if \(\Diag_a = \Diag_0\)
	\\
	(x^a_0-x^a_2) (x^a_0+x^a_2-x^a_1-x^a_3)
	&
	if \(\Diag_a = \Diag_1\)
	\end{cases*}
	\]
	so that
	\[
		\tM(\Diag_a)=
		\begin{tikzcd}[column sep=1cm, ampersand replacement=\&]
		\tR
		\arrow[r,bend left=10, "\lambda(a)"above]
		\&
		\tR
		\arrow[l,bend left=10, "\kappa(a)"below]
		\end{tikzcd}
	\] 
	Next, given an arc \(a^*\) which connects the two open components of the diagram, let \(\mathfrak{a}\coloneqq \mathfrak{a}(a^*)\) be the sum of the two ideals \(\mathfrak{e}\) and \(\mathfrak{o}\) in \(\tR\) that are defined as follows:
	\(\mathfrak{e}\coloneqq\mathfrak{e}(a^*)\) is generated by all expressions \((x_i-x_j)\), where \(x_i\) and \(x_j\) vary over the labels of all pairs of regions that can be connected by a path \(\gamma\) which avoids all diagram components as well as the arc \(a^*\) and intersects the thick arcs an even number of times. 
	Similarly, \(\mathfrak{o}\coloneqq\mathfrak{o}(a^*)\) is generated by the expressions \((H-(x_i-x_j))\), where \(x_i\) and \(x_j\) are as before, except that the path \(\gamma\) should intersect the thick arcs an odd number of times. 
	In particular, \(\lambda(a)\in\mathfrak{a}\) for every arc \(a\neq a^*\). 
	Note that these expressions are independent of the particular path connecting \(x_i\) and \(x_j\). 
	(Also note that if we worked with signs, the signs in these expressions would depend on the direction of the path \(\gamma\) and how it intersects the thick edges.) An example of a diagram and the corresponding ideals is given in Figure~\ref{fig:ideals}.

  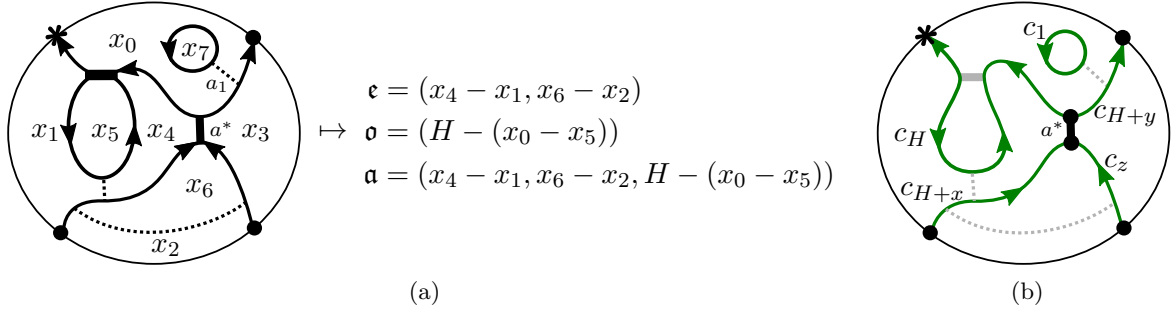
\begin{figure}[t]
  \begin{subfigure}{0.69\textwidth}
  % \hspace*{-2cm}
  \labellist 
  \pinlabel \(x_0\) at 27 52
  \pinlabel \(x_7\) at 44 50
  \pinlabel \(x_2\) at 37 4
  \pinlabel \(x_4\) at 36 31
  \pinlabel \(\scriptstyle a^*\) at 50 32
  \pinlabel \(\scriptstyle a_1\) at 49 42
  \pinlabel \(x_5\) at 23 31
  \pinlabel \(x_6\) at 45 18
  \pinlabel \(x_1\) at 9 31
  \pinlabel \(x_3\) at 58 31
  \pinlabel \(\mapsto\) at 76 31
  \pinlabel \(\mathfrak{e}=(x_4-x_1,x_6-x_2)\) at 116 41
  \pinlabel \(\mathfrak{o}=(H-(x_0-x_5))\) at 113 31
  \pinlabel \(\mathfrak{a}=(x_4-x_1,x_6-x_2,H-(x_0-x_5))\) at 138 21
  \endlabellist
  \includegraphics[scale=1.6]{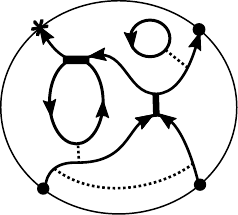} 
  \caption{}\label{fig:ideals}
  \end{subfigure}
  \begin{subfigure}{0.29\textwidth}
  \hspace*{0.2cm}
  \labellist 
  \pinlabel \(\scriptstyle a^*\) at 41 32
  \pinlabel \(c_H\) at 8 30
  \pinlabel \(c_1\) at 37 55
  \pinlabel \(c_{H+x}\) at 13 17
  \pinlabel \(c_{z}\) at 56 24
  \pinlabel \(c_{H+y}\) at 58 35
  \endlabellist
  \includegraphics[scale=1.6]{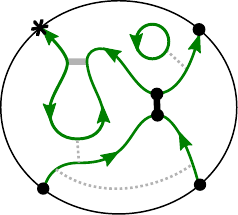}
  \caption{}
  \label{fig:segments}
  \end{subfigure}
  \caption{(a) A diagram and three associated ideals corresponding to the distinguished arc \(a^*\); (b)~The same diagram, but with the labels \(c_\star\) corresponding to the five diagram segments relative to the arc \(a^*\), which are drawn in green}
  \end{figure}
	
  Ends of \emph{all} arcs in $\Diag$ divide the diagram strands into \textbf{edges}, eg the diagram in Figure~\ref{fig:ideals} has $10$ edges. We now introduce some notation for keeping track of elements in \(\tR\) corresponding to edges in \(\Diag\).
	For each such edge~\(e\), consider the circle $c_e$ containing \(e\) in the \(\infty\)-closure \(\closureDiag\) of the diagram \(\Diag\). Define \(\sigma_{e}\in\ZZ\) as the mod 2 number of circles 
  in \(\closureDiag\) in which $c_e$ is nested in,  plus 1 if the orientation of \(e\) induces a counter-clockwise orientation on $c_e$. 
	We then set 
	\[
	Y_e=
	\begin{cases*}
	x_r-x_\ell & if \(\sigma_{e}=0\)
	\\
	H-(x_r-x_\ell) & if \(\sigma_{e}=1\)
	\end{cases*}
	\]
	where \(x_r\) and \(x_\ell\) are the labels of the two regions on the right and left of the edge \(e\), respectively. 
	The above is inspired by the definition of \textit{canonical generator} in~\cite{JakeSInvariant}.
	The key property of \(\sigma_e\) that we will be using in the last part of this proof is that its value stays invariant under merge and split operations.

	We will need to group edges together. Ends of the arc \(a^*\) divide the diagram strands  (without any arcs) into \textbf{diagram segments relative to \(\bm{a^*}\)}. For instance, the diagram in Figure~\ref{fig:ideals} has five segments, illustrated in Figure~\ref{fig:segments}. In other words, each circle amounts to one segment and each open diagram strand consists of two diagram segments that meet at an endpoint of \(a^*\). We will call the first type of segment closed, and the latter open.  We denote the open diagram segments, ordered by their ends on the boundary of \(\Diag\), starting at the top left and going in counter-clockwise direction, by \(c_H\), \(c_{H+x}\), \(c_z\), and \(c_{H+y}\), respectively. We enumerate the closed diagram segments and denote them by \(c_1,\dots,c_n\). 
	For each diagram segment \(c_\star\), we pick some oriented edge \(e\) on \(c_\star\) and write \(Y_{\star}\coloneqq Y_{e}\). 
	The element \([Y_\star]\in\tR/\mathfrak{a}\) is independent of the choice of edge \(e\), which can be seen inductively by following a path along \(c_\star\). In particular, if \(\star\in\{H,H+x,z,H+y\}\), \([Y_\star]=\star\in\tR/\mathfrak{a}\). 
	
  \bigskip

	Next, extend \(\{a^*\}\) to a minimal subset \(A\) of dotted or thick arcs that makes \(\Diag\) connected. Call such a subset \(A\) an  \textbf{arc system} for \(\{a^*\}\). 
	For example, the set \(\{a^*,a_1\}\) constitutes an arc system for the diagram in Figure~\ref{fig:ideals}.
	In general, we claim: 
  \begin{itemize}
  \item The ideal in $\tR$ generated by \(\lambda(a)\) for \(a\not\in A\) is equal to the ideal \(\mathfrak{a}\).
  % DETAILS: one inclusion for the ideals is obvious. So for the other inclusion, first recall that the definition is independent of the particular path chosen. So given two regions that can be connected by a path as above, then we may choose this path to avoid any arc in A, since the A is minimal. Then proceed by induction on the number of arcs that this path crosses. 
  \item The algebra \(\tR/\mathfrak{a}\) is freely generated by \(Y_1,\dots,Y_n\) as an \(\tR^\partial\)-algebra. 
  % step 1: a free basis of \R is given by {H} union x_0-x_i for all i\neq 0. 
  % step 2: fix a root of each component of the collection of trees considered previously with roots x_0, x_1, x_2 and x_3 for the open regions. We can find a new basis {H}, union {x_0-x_i} for each root x_i \neq x_0, union {x_i -x_j} where x_i is a root and x_j a region in the same component of the collection of trees as x_i\neq x_j. 
  % step 3: inductively, do base changes of {x_i -x_j} starting at the outer leaves of each tree to the generators of the ideal a. Therefore R/a is freely generated by {H} union {x_0-x_i} for each root x_i \neq x_0. 
  % step 4: for each closed component, chose some pair of adjacent regions x_r and x_l defining Y_i. Together with x_0-x_3 and x_3-x_2 and x_0-x_1, these form a maximal tree in the dual graph. Now apply the same argument as in step 3 to see that R/a is freely generated by {H,x_0-x_3,x_3-x_2,x_0-x_1,Y_1,...,Y_n}. 
  % step 5: quotient out by x_0-x_1 and observe that tRpartial is freely generated by {H,x_0-x_3,x_3-x_2}.
  \end{itemize}   
	Both facts follow from graph combinatorics, by considering trees in the dual graph of \(\Diag\), which is formed by the edges dual to the arcs not in \(A\).
	
	A straightforward computation shows that 
	\[
	[\kappa(a)]
	=
	\big(H-[x^a_0-x^a_1]\big)[x^a_0-x^a_1]
	-
	\big(H-[x^a_3-x^a_2]\big)[x^a_3-x^a_2]
	\in\tR/(\lambda(a))
	\]
	%	If \(\Diag_a=\Diag_0\), then as elements in \(\tR/(\lambda(a))\), 
	%	\begin{align*}
	%	[\kappa(a)]
	%	&
	%	=
	%	(H-([x^a_2-x^a_1]+[x^a_3-x^a_2])([x^a_2-x^a_1]-[x^a_3-x^a_2])
	%	\\
	%	&
	%	=
	%	H([x^a_2-x^a_1]-[x^a_3-x^a_2])-[x^a_2-x^a_1]^2+[x^a_3-x^a_2]^2
	%	\\
	%	&
	%	=
	%	(H-[x^a_2-x^a_1])[x^a_2-x^a_1]-(H-[x^a_3-x^a_2])[x^a_3-x^a_2]
	%	\end{align*}
	%	Likewise, if \(\Diag_a=\Diag_1\), then as elements in \(\tR/(\lambda(a))\), 
	%	\begin{align*}
	%	[\kappa(a)]
	%	&
	%	=
	%	([x^a_0-x^a_1]+[x^a_1-x^a_2]) (([x^a_0-x^a_1]-[x^a_1-x^a_2])-H)
	%	\\
	%	&
	%	=
	%	(H-[x^a_0-x^a_1])[x^a_0-x^a_1]-(H-[x^a_1-x^a_2])[x^a_1-x^a_2]
	%	\end{align*}
	So, in particular, if \(c_\star\) and \(c_\vartriangle\) are two segments connected by an arc \(a\neq a^*\), then
	\[
	[\kappa(a)]
	=
	[Y_\star]\big(H-[Y_\star]\big)
	-
	[Y_{\vartriangle}]\big(H-[Y_{\vartriangle}]\big)
	\in\tR/\mathfrak{a}
	\]
	For \(\star\in\{H,H+x,z,H+y\}\), observe that
	\[
	[Y_{\star}]\big(H-[Y_{\star}]\big)
	=
	\Delta_\star\in\tR/\mathfrak{a}
	\quad
	\text{where }
	\quad
	\Delta_\star
	\coloneqq
	\begin{cases}
	0
	&
	\star=H
	\\
	x(y+z)
	&
	\star=H+x
	\\
	z(x+y)
	&
	\star=z
	\\
	y(x+z)
	&
	\star=H+y
	\end{cases}
	\]	
%	in the subset of the plane defined as the union of the tangle strands and the arcs in \(A\smallsetminus\{a^*\}\)
%	Any arc \(a\in A\smallsetminus\{a^*\}\) is connected to exactly one open component \(c_a\) to exactly one of  define \(\Delta_\)
%	Let \(\mathfrak{A}(A)\)

	\mybfheading{1) Construction of the maps \(\bm{\varphi_A}\)}
	Pick some arc \(a^*\) connecting the two open tangle strands.  
	By the Assumptions~\ref{assums:tangle_diagram}, such an arc always exists.  
	Then pick some arc system \(A\) for \(a^*\). 
	We define the special deformation retract \(\varphi_A\) in three steps. 
	First, we define a special deformation retract 
	\[
		\varphi_A^1
		\co
		\tM(\Diag) = \bigotimes_{a}\tM(\Diag_a)
		\longrightarrow
		\bigotimes_{a\in A} \tM(\Diag_a) 
		\otimes 
		\bigotimes_{a\not\in A}\tR/(\lambda(a))
		\cong
		\bigotimes_{a\in A} 
		\left(
		\tM(\Diag_a) 
		\otimes 
		\tR/\mathfrak{a}
		\right)
	\]
  of matrix factorizations over $\tR/\mathfrak{a}=\tR^\partial[Y_1,\dots,Y_n]$.
	On the tensor factors for \(a\not\in A\), the map is induced by the special deformation retractions from Lemma~\ref{lem:lin_sp_def_retr}; on the other tensor factors, it is the identity. 
	For each \(a\in A\smallsetminus\{a^*\}\), \(\lambda(a)\in\mathfrak{a}\), so 
	\[
	\left(
	\tM(\Diag_a) 
	\otimes 
	\tR/\mathfrak{a}
	\right)
	\cong
	\tR/\mathfrak{a}\xrightarrow{\kappa(a)}\tR/\mathfrak{a}
	\]
	Let us now turn to \(\tM(\Diag_{a^*})\). 
	Suppose \(o(\Diag)=\Diag_0\). 
	Then by considering a path connecting the regions \(x_0\) and \(x_2\), we see that
	\([\lambda(a^*)]=x\in\tR/\mathfrak{a}\). 
	Note that this identity holds regardless of whether \(\Diag_{a^*}=\Diag_0\), \(\Diag_{a^*}=\Diag_1\), or how \(\Diag_{a^*}\) is connected to the rest of the diagram \(\Diag\). 
	Likewise, if \(o(\Diag)=\Diag_1\), we see that \([\lambda(a^*)]=y\in\tR/\mathfrak{a}\). 
	In summary,
	\[
	\left(
	\tM(\Diag_{a^*}) 
	\otimes 
	\tR/\mathfrak{a}
	\right)
	=
	\tM(o(\Diag))=
	\begin{tikzcd}[column sep=1cm, ampersand replacement=\&]
	\tR
	\arrow[r,bend left=10, "\lambda^{*}"above]
	\&
	\tR
	\arrow[l,bend left=10, "xyz/\lambda^{*}"below]
	\end{tikzcd}
	\quad
	\text{where } 
	\lambda^{*}=
	\begin{cases}
	x & \text{ if } o(\Diag)=\Diag_0\\
	y & \text{ if } o(\Diag)=\Diag_1
	\end{cases}
	\]
	as a matrix factorization over \(\tR/\mathfrak{a}\), since the overall potential is \(xyz\). 
	
	Let \(\mathfrak{A}\coloneqq\mathfrak{A}(A)\supseteq\mathfrak{a}\) be the sum of \(\mathfrak{a}\) with the ideal generated by the the expressions \(\kappa(a)\) for \(a\in A\smallsetminus\{a^*\}\). Applying Lemma~\ref{lem:quad_sp_def_retr}, we obtain the special deformation retraction 
	\[
		\varphi_A^2
		\co
		\bigotimes_{a\in A} 
		\left(
		\tM(\Diag_a) 
		\otimes 
		\tR/\mathfrak{a}
		\right)
		\rightarrow
		\tR/\mathfrak{A}\otimes \tM(o(\Diag))
	\]
  of matrix factorizations over $\tR^\partial$.
	Since the expressions \(\kappa(a)\) are quadratic polynomials in \(\tR/\mathfrak{a} \cong\tR^\partial[Y_1,\dots,Y_n]\), there exists an \(\tR^\partial\)-module isomorphism 
	\[
		\tR/\mathfrak{A}
		\rightarrow
		V(\Diag)\otimes_{\fieldTwoElements[H]} \tR^\partial
	\]
	which is uniquely defined by \(\prod_j Y_j^{\varepsilon_j}\mapsto \prod_j y_j^{\varepsilon_j}\) for any tuple \(\varepsilon = (\varepsilon_1,\dots,\varepsilon_n)\in\{0,1\}^n\). This isomorphism induces an isomorphism 
	\[
	\varphi_A^3
	\co
	\tR/\mathfrak{A}\otimes \tM(o(\Diag))
	\rightarrow
	V(\Diag)\otimes \tM(o(\Diag))
	\]
	We now define the special deformation retract \(\varphi_\Diag\coloneqq\varphi_A\) as the composition \(\varphi_A^3\circ\varphi_A^2\circ\varphi_A^1\). As the notation suggests, this map depends on the choice of an arc system \(A\). However, as we will show next, its homotopy type is independent of \(A\).

	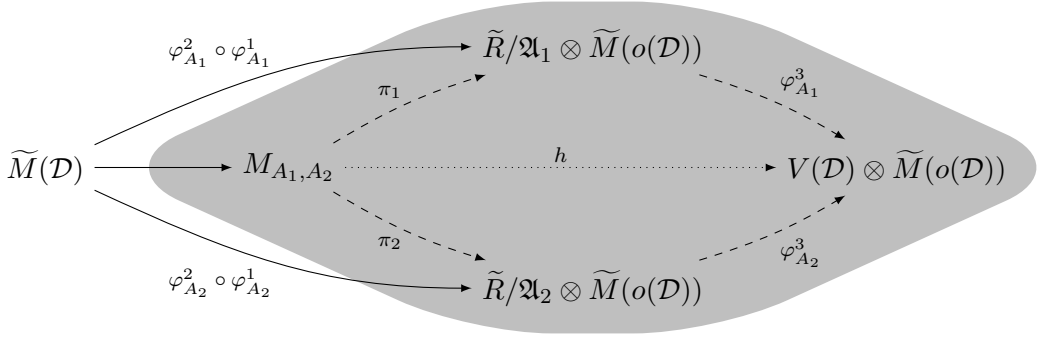
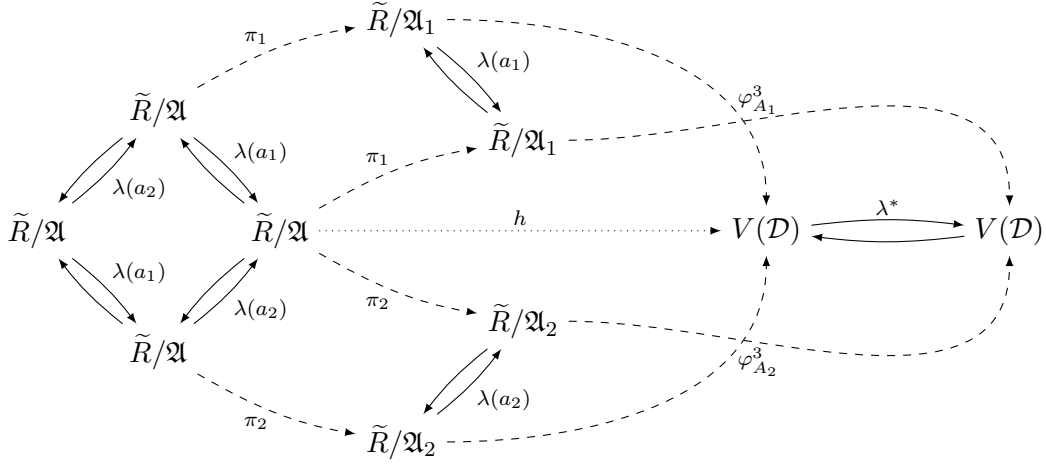
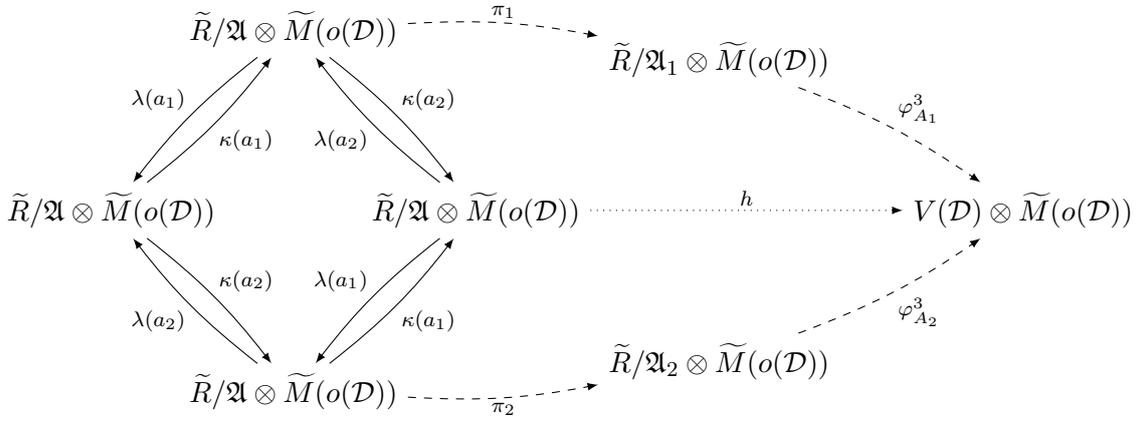
\begin{figure}[p]
		\begin{subfigure}{\textwidth}	
			\centering
			\begin{tikzpicture}[scale=0.8]
			\path[fill,lightgray,rounded corners=30pt] (-8,0) -- (-2,2.75) -- (2,2.75) -- (8,0) -- (2,-2.75) -- (-2,-2.75) -- cycle;
			\node(L) at (-9,0){\(\tM(\Diag)\)};
			\node(C) at (-5,0){\(M_{A_1,A_2}\)};
			\node(T) at (0,2){\(\tR/\mathfrak{A}_1\otimes\tM(o(\Diag))\)};
			\node(B) at (0,-2){\(\tR/\mathfrak{A}_2\otimes\tM(o(\Diag))\)};
			\node(R) at (5,0){\(V(\Diag)\otimes\tM(o(\Diag))\)};
			\draw[->,dotted] (C) to [tight,above]node{\(h\)} (R);
			\draw[->,dashed,bend left=7] (T) to [tight,above right]node{\(\varphi^3_{A_1}\)} (R);
			5		\draw[->,dashed,bend left=7] (C) to [tight,above left]node{\(\pi_1\)} (T);
			\draw[->,dashed,bend right=7] (C) to [tight,below left]node{\(\pi_2\)} (B);
			\draw[->,dashed,bend right=7] (B) to [tight,below right]node{\(\varphi^3_{A_2}\)} (R);
			\draw[->,in=180,out=25] (L) to [tight,above left]node{\(\varphi^2_{A_1}\circ\varphi^1_{A_1}\)} (T);
			\draw[->] (L) to node{} (C);
			\draw[->,in=180,out=-25] (L) to [tight,below left]node{\(\varphi^2_{A_2}\circ\varphi^1_{A_2}\)} (B);
			\end{tikzpicture}
			\caption{The general strategy for constructing the homotopy between \(\varphi_{A_1}\) and \(\varphi_{A_2}\)}\label{fig:delooping:mf:homotopy:overview}
		\end{subfigure}
		\begin{subfigure}{\textwidth}	
			\centering
			\begin{tikzpicture}[scale=0.8]
			% diagram
			\begin{scope}[shift={(-4,0)}]% left
			\node(Lt) at (0,2){\(\tR/\mathfrak{A}\)};
			\node(Lb) at (0,-2){\(\tR/\mathfrak{A}\)};
			\node(Lr) at (2,0){\(\tR/\mathfrak{A}\)};
			\node(Ll) at (-2,0){\(\tR/\mathfrak{A}\)};
			\draw[->] (Lt) to[bend left=7,above right,tight]node{\(\lambda(a_1)\)} (Lr);
			\draw[->] (Lr) to[bend left=7,above right,tight]node{} (Lt);
			\draw[->] (Ll) to[bend left=7,above right,tight]node{\(\lambda(a_1)\)} (Lb);
			\draw[->] (Lb) to[bend left=7,above right,tight]node{} (Ll);
			\draw[->] (Ll) to[bend right=7,below right,tight]node{\(\lambda(a_2)\)} (Lt);
			\draw[->] (Lt) to[bend right=7,below right,tight]node{} (Ll);
			\draw[->] (Lb) to[bend right=7,below right,tight]node{\(\lambda(a_2)\)} (Lr);
			\draw[->] (Lr) to[bend right=7,below right,tight]node{} (Lb);
			\end{scope}
			\begin{scope}[shift={(0,+1.5)}]% top
			\node(Tt) at (0,2){\(\tR/\mathfrak{A}_1\)};
			\node(Tr) at (2,0){\(\tR/\mathfrak{A}_1\)};
			\draw[->] (Tt) to[bend left=7,above right,tight]node{\(\lambda(a_1)\)} (Tr);
			\draw[->] (Tr) to[bend left=7,above right,tight]node{} (Tt);
			\end{scope}
			\begin{scope}[shift={(0,-1.5)}]% bottom
			\node(Bb) at (0,-2){\(\tR/\mathfrak{A}_2\)};
			\node(Br) at (2,0){\(\tR/\mathfrak{A}_2\)};
			\draw[->] (Bb) to[bend right=7,below right,tight]node{\(\lambda(a_2)\)} (Br);
			\draw[->] (Br) to[bend right=7,below right,tight]node{} (Bb);
			\end{scope}
			\begin{scope}[shift={(+8,0)}]% right
			\node(Rr) at (2,0){\(V(\Diag)\)};
			\node(Rl) at (-2,0){\(V(\Diag)\)};
			\draw[->] (Rr) to[bend left=7,below,tight]node{} (Rl);
			\draw[->] (Rl) to[bend left=7,above,tight]node{\(\lambda^{*}\)} (Rr);
			\end{scope}
			% map for A₁:
			\draw[dashed,->,out=30,in=-170] (Lt) to [tight,above left]node{\(\pi_1\)} (Tt);
			\draw[dashed,->,out=0,in=90] (Tt) to node{} (Rl);
			\draw[dashed,->,out=30,in=-170] (Lr) to [tight,above left]node{\(\pi_1\)} (Tr);
			\draw[dashed,->,out=0,in=90] (Tr) to [tight,above,pos=0.33]node{\(\varphi^3_{A_1}\)} (Rr);
			% map for A₂:
			\draw[dashed,->,out=-30,in=170] (Lb) to [tight,below left]node{\(\pi_2\)} (Bb);
			\draw[dashed,->,out=0,in=-90] (Bb) to node{} (Rl);
			\draw[dashed,->,out=-30,in=170] (Lr) to [tight,below left]node{\(\pi_2\)} (Br);
			\draw[dashed,->,out=0,in=-90] (Br) to [tight,below,pos=0.33]node{\(\varphi^3_{A_2}\)} (Rr);
			% homotopy h:
			\draw[->,dotted] (Lr) to [tight,above]node{\(h\)} (Rl);
			\end{tikzpicture}
			\caption{%
				The homotopy \(h\)
				between \(\varphi^3_{A_1}\circ\pi_1\)
				and \(\varphi^3_{A_2}\circ\pi_2\)
				in case 1}
			\label{fig:delooping:mf:homotopy:i}
		\end{subfigure}
		\begin{subfigure}{\textwidth}	
			\centering
			\begin{tikzpicture}[scale=0.8]
			% diagram
			\begin{scope}[shift={(-4,0)}]% left
			\node(Lt) at (0,3){\(\tR/\mathfrak{A}\otimes\tM(o(\Diag))\)};
			\node(Lb) at (0,-3){\(\tR/\mathfrak{A}\otimes\tM(o(\Diag))\)};
			\node(Lr) at (3,0){\(\tR/\mathfrak{A}\otimes\tM(o(\Diag))\)};
			\node(Ll) at (-3,0){\(\tR/\mathfrak{A}\otimes\tM(o(\Diag))\)};
			\draw[->] (Lt) to[bend left=7,above right,tight]node{\(\kappa(a_2)\)} (Lr);
			\draw[->] (Lr) to[bend left=7,below left,tight]node{\(\lambda(a_2)\)} (Lt);
			\draw[->] (Ll) to[bend left=7,above right,tight]node{\(\kappa(a_2)\)} (Lb);
			\draw[->] (Lb) to[bend left=7,below left,tight]node{\(\lambda(a_2)\)} (Ll);
			\draw[->] (Ll) to[bend right=7,below right,tight]node{\(\kappa(a_1)\)} (Lt);
			\draw[->] (Lt) to[bend right=7,above left,tight]node{\(\lambda(a_1)\)} (Ll);
			\draw[->] (Lb) to[bend right=7,below right,tight]node{\(\kappa(a_1)\)} (Lr);
			\draw[->] (Lr) to[bend right=7,above left,tight]node{\(\lambda(a_1)\)} (Lb);
			\end{scope}
			\node(T) at (3,2.5){\(\tR/\mathfrak{A}_1\otimes\tM(o(\Diag))\)};
			\node(B) at (3,-2.5){\(\tR/\mathfrak{A}_2\otimes\tM(o(\Diag))\)};
			\node(R) at (8,0){\(V(\Diag)\otimes \tM(o(\Diag))\)};
			% map for A₁:
			\draw[dashed,->,bend left=7] (Lt) to [tight, above]node{\(\pi_1\)}  (T);
			\draw[dashed,->,bend left=7] (T) to [tight, above right]node{\(\varphi^3_{A_1}\)} (R);
			% map for A₂:
			\draw[dashed,->,bend right=7] (Lb) to [tight, below]node{\(\pi_2\)} (B);
			\draw[dashed,->,bend right=7] (B) to [tight, below right]node{\(\varphi^3_{A_2}\)} (R);
			% homotopy h:
			\draw[->,dotted] (Lr) to [tight,above]node{\(h\)} (R);
			\end{tikzpicture}
			\caption{%
				The homotopy \(h\)
				between \(\varphi^3_{A_1}\circ\pi_1\) 
				and \(\varphi^3_{A_2}\circ\pi_2\)
				in case 2}
			\label{fig:delooping:mf:homotopy:ii}
		\end{subfigure}
		\caption{Commutative diagrams illustrating the second step in the proof of Lemma~\ref{lem:delooping:mf}%: The homotopy class of the map \(\varphi_{\Diag}\) does not depend on any choices in its construction. %
		}\label{fig:delooping:mf:homotopy}
	\end{figure}

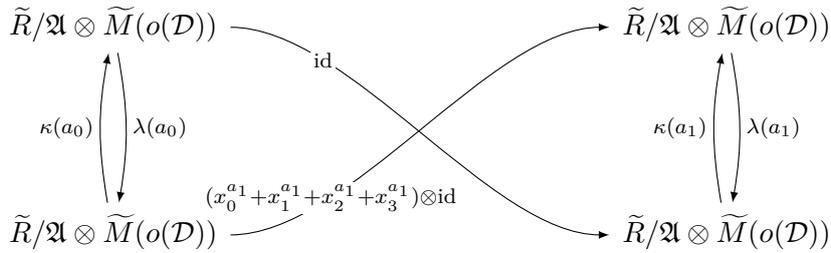
\begin{figure}[p]
	\centering
	\(
	\begin{tikzcd}[column sep=5cm,row sep=2cm]
	\tR/\mathfrak{A}\otimes\tM(o(\Diag))
	\arrow[d,bend left=10, "\lambda(a_0)"right]
	\arrow[in=180,out=0,looseness=0.7]{rd}[inner sep=1pt, outer sep=1pt, description,fill=white, near start]{\id}
	&
	\tR/\mathfrak{A}\otimes\tM(o(\Diag))
	\arrow[d,bend left=10, "\lambda(a_1)"right]
	\\
	\tR/\mathfrak{A}\otimes\tM(o(\Diag))
	\arrow[u,bend left=10, "\kappa(a_0)"left]
	\arrow[in=180,out=0,looseness=0.7]{ru}[inner sep=1pt, outer sep=1pt, description,fill=white, pos=0.27]{(x^{a_1}_0+x^{a_1}_1+x^{a_1}_2+x^{a_1}_3)\otimes\id}
	&
	\tR/\mathfrak{A}\otimes\tM(o(\Diag))
	\arrow[u,bend left=10, "\kappa(a_1)"left]
	\end{tikzcd}
	\)
	% NOTE: o(D)=o(D')
	\caption{The edge map in the third step, case 2, in the proof of Lemma~\ref{lem:delooping:mf} 
	}\label{fig:delooping:mf:edge_maps}
\end{figure}

	\mybfheading{2) Independence of \(\bm{\varphi_\Diag}\) up to homotopy}
	Let \(A_1\) and \(A_2\) be two arc systems extending \(a_1^*\) and \(a_2^*\), respectively. We claim that the maps \(\varphi_{A_1}\) and \(\varphi_{A_2}\) are homotopic. One can easily see by induction on \(|A_1\smallsetminus (A_2\cap A_1)|=|A_2\smallsetminus (A_2\cap A_1)|\) that it suffices to show this claim in the case that \(A_1\) and \(A_2\) differ in a single arc. 
	So let \(A\coloneqq A_1 \cap A_2\) so that \(A_1=A\cup \{a_1\}\) and \(A_2=A\cup \{a_2\}\) for some arcs \(a_1\neq a_2\). We distinguish two cases.

	\myitheading{Case 1} Suppose the arc \(a_1\) connects the two open components of \(\Diag\); then so does \(a_2\), and both of them are the distinguished arcs \(a_1^*=a_1,~a_2^*=a_2\) (because of the minimality assumption of an arc system). Let \(\mathfrak{a}\) be the ideal generated by \(\lambda(a)\) where \(a\not\in A_1\cup A_2\). Then 
	\[
	\mathfrak{a}_1\coloneqq\mathfrak{a}(a^*_1)=\mathfrak{a}+(\lambda(a_2))
	\quad
	\text{ and }
	\quad
	\mathfrak{a}_2\coloneqq\mathfrak{a}(a^*_2)=\mathfrak{a}+(\lambda(a_1))
	\]
	Likewise, define \(\mathfrak{A}\) as the sum of \(\mathfrak{a}\) and the ideal generated by \(\kappa(a)\) where \(a\in A=A_1 \cap A_2\), so that 
	\[
	\mathfrak{A}_1\coloneqq\mathfrak{A}(A_1)=\mathfrak{A}+(\lambda(a_2))
	\quad
	\text{ and }
	\quad
	\mathfrak{A}_2\coloneqq\mathfrak{A}(A_2)=\mathfrak{A}+(\lambda(a_1))
	\]
	If \(o(\Diag)=\Diag_0\), there exists a path between the regions labelled \(x_2\) and \(x_0\), which avoids the arcs in \(A\) and crosses each of the arcs \(a_1\) and \(a_2\) exactly once. Hence, \([\lambda(a_1)]-[\lambda(a_2)]=x_0-x_2=x\in\tR/\mathfrak{a}\). If \(o(\Diag)=\Diag_1\), we similarly see that \([\lambda(a_1)]-[\lambda(a_2)]=H-(x_3-x_1)=y\in\tR/\mathfrak{a}\). So in either case, 
	\begin{equation}\label{eq:delooping:mf:delta_x_a:case1}
	[\lambda(a_1)]-[\lambda(a_2)]=\lambda^{*}\in\tR/\mathfrak{a}
	\end{equation}	
	For \(i=1,2\), Lemma~\ref{lem:lin_sp_def_retr} gives us a map 
	\[
	\pi_i\co
	M_{A_1,A_2}
	\coloneqq
	\tR/\mathfrak{A}
	\otimes
	\tM(\Diag_{a_1})
	\otimes
	\tM(\Diag_{a_2})
	\rightarrow
	\tR/\mathfrak{A}_i
	\otimes
	\tM(o(\Diag))
	\]
	and the map \(\varphi^2_{A_i}\circ\varphi^1_{A_i}\) factors through this map, as shown on the left of Figure~\ref{fig:delooping:mf:homotopy:overview}. 
	It therefore suffices to construct a homotopy \(h\) between \(\varphi^3_{A_1}\circ\pi_1\) and \(\varphi^3_{A_2}\circ\pi_2\).
	For this, we define \(b_{k,\varepsilon}\coloneqq(\lambda(a_2))^k\prod_j[Y_j]^{\varepsilon_j}\) for any integer \(k\geq0\) and any tuple \(\varepsilon = (\varepsilon_1,\dots,\varepsilon_n)\in\{0,1\}^n\). 
	The set \(\{b_{k,\varepsilon}\}\) freely generates \(\tR/\mathfrak{A}\) over \(\tR^\partial\). 
	Now consider Figure~\ref{fig:delooping:mf:homotopy:i}, which shows a more detailed version of the diagram contained in the shaded region of Figure~\ref{fig:delooping:mf:homotopy:overview}. The two pairs of consecutive dashed arrows in the top half of the diagram are equal to \(\varphi^3_{A_1}\circ\pi_1\). Similarly, the dashed arrows in the bottom half represent the map \(\varphi^3_{A_2}\circ\pi_2\). 
	We calculate the compositions of the four consecutive dashed arrows in this diagram as follows: 
	\[
	\left.(\varphi^3_{A_1}\circ\pi_1)\right|_{\tR/\mathfrak{A}}(b_{k,\varepsilon})
	=
	\begin{cases}
	\prod_j y_j^{\varepsilon_j}
	&
	k=0
	\\
	0 
	&
	k>0
	\end{cases}
	\quad
	\text{and}
	\quad
	\left.(\varphi^3_{A_2}\circ\pi_2)\right|_{\tR/\mathfrak{A}}(b_{k,\varepsilon})
	=
	(\lambda^{*})^k\textstyle\prod_j y_j^{\varepsilon_j}
	\]
	For the first identity, we use the fact that \([\lambda(a_2)]=0\in\tR/\mathfrak{A}_1\); for the second identity, we use \([\lambda(a_1)]=0\in\tR/\mathfrak{A}_2\) together with relation~\eqref{eq:delooping:mf:delta_x_a:case1}.
	A routine calculation shows that the map \(h\) represented by the dotted arrow in Figure~\ref{fig:delooping:mf:homotopy:i} and defined by 
	\[
	h(b_{k,\varepsilon})
	=
	\begin{cases}
	0
	&
	k=0
	\\
	(\lambda^{*})^{k-1} \prod_j y_j^{\varepsilon_j}
	&
	k>0
	\end{cases}
	\]
	establishes the desired homotopy \(h\). 
	
	\myitheading{Case 2} Suppose \(a_1\) does not connect the two open components; then neither does \(a_2\). Define \(\mathfrak{a}\), \(\mathfrak{a}_1\), \(\mathfrak{a}_2\), and \(\mathfrak{A}\) as in the previous case. Then
	\[
	\mathfrak{A}_1\coloneqq\mathfrak{A}(A_1)=\mathfrak{A}+(\lambda(a_2),\kappa(a_1))
	\quad
	\text{ and }
	\quad
	\mathfrak{A}_2\coloneqq\mathfrak{A}(A_2)=\mathfrak{A}+(\lambda(a_1),\kappa(a_2))
	\]
	By considering a closed path in \(\Diag\) that avoids all diagram strands and arcs in \(A\) and crosses each of the arcs \(a_1\) and \(a_2\) exactly once, we see that 
	\begin{equation}\label{eq:delooping:mf:delta_x_a:case2}
	[\lambda(a_1)]=[\lambda(a_2)]\in\tR/\mathfrak{a}
	\end{equation}
	Let \((a_{i_1},\dots,a_{i_m})\) be a finite sequence of arcs in \(A_1\cup A_2\) containing \(a_1\) and \(a_2\) such that for \(j=1,\dots,m-1\), the arcs \(a_{i_j}\) and \(a_{i_{j+1}}\) have ends on the same closed segment, and \(a_{i_1}\) and  \(a_{i_m}\) have an end on an open segment \(c_\star\) and \(c_\vartriangle\), respectively, for some \(\star,\vartriangle\in\{H,H+x,z,H+y\}\). 
	(Note that the \(a_{i_j}\) may repeat themselves, in which case the sum below is zero.)
	Then the telescoping sum results in
	\[
	\Delta
	\coloneqq
	\Delta_{\star}+\Delta_{\vartriangle}
	=
	\sum_{j=1}^m[\kappa(a_{i_j})]\in\tR/\mathfrak{a}_1
	\]
	so that 
	\begin{equation}\label{eq:delooping:mf:delta_y_a:case2}
	[\kappa(a_1)]+[\kappa(a_2)]=\Delta+\lambda(a_2)\alpha\in\tR/\mathfrak{A}
	\end{equation}
	for some \(\alpha\in\tR\).
	For \(i=1,2\), the map \(\varphi^2_{A_i}\circ\varphi^1_{A_i}\) factors through some map 
	\[
	\pi_i\co
	M_{A_1,A_2}
	\coloneqq
	\tR/\mathfrak{A}
	\otimes
	\tM(\Diag_{a_1})
	\otimes
	\tM(\Diag_{a_2})
	\otimes
	\tM(o(\Diag))
	\rightarrow
	\tR/\mathfrak{A}_i\otimes \tM(o(\Diag))
	\]
	As in Case~1, it therefore suffices to construct a homotopy between \(\varphi^3_{A_1}\circ\pi_1\) and \(\varphi^3_{A_2}\circ\pi_2\).
	For this, we define \(b_{k,\ell,\varepsilon}\coloneqq(\lambda(a_2))^k(\kappa(a_1))^\ell\prod_j[Y_j]^{\varepsilon_j}\) for any integers \(k,\ell\geq0\) and any tuple \(\varepsilon = (\varepsilon_1,\dots,\varepsilon_n)\in\{0,1\}^n\). 
	The set \(\{b_{k,\ell,\varepsilon}\}\) freely generates \(\tR/\mathfrak{A}\) over \(\tR^\partial\). 
  Figure~\ref{fig:delooping:mf:homotopy:ii} shows a more detailed version of the diagram contained in the shaded region of Figure~\ref{fig:delooping:mf:homotopy:overview}. 
  On the tensor factor $\tM(o(\Diag))$, the maps \(\varphi^3_{A_i}\circ\pi_i\) are equal to the identity. When restricted to the first tensor factors, they are determined by
	\[
	(\varphi^3_{A_1}\circ\pi_1)\co 
	b_{k,\ell,\varepsilon}
	\mapsto
	\begin{cases}
	\prod_j y_j^{\varepsilon_j}
	&
	k=0=\ell
	\\
	0 
	&
	k>0\text{ or } \ell>0
	\end{cases}
	\quad
	\text{and}
	\quad
	(\varphi^3_{A_2}\circ\pi_2)
	\co
	b_{k,\ell,\varepsilon}
	\mapsto
	\begin{cases}
	\Delta^\ell\prod_j y_j^{\varepsilon_j}
	&
	k=0
	\\
	0 
	&
	k>0
	\end{cases}
	\]
	For the first identity, we use the fact that \([\lambda(a_2)]=0\in\tR/\mathfrak{A}_1\); for the second identity, we also use the relations~\eqref{eq:delooping:mf:delta_x_a:case2} and~\eqref{eq:delooping:mf:delta_y_a:case2}.
	The desired homotopy \(h\) is now given by the map represented by the dotted arrow in Figure~\ref{fig:delooping:mf:homotopy:i} and defined by 
	\[
	h
	\co
	b_{k,\ell,\varepsilon}
	\mapsto
	\begin{cases}
	\Delta^{\ell-1}\prod_j y_j^{\varepsilon_j}
	&
	k=0\text{ and }\ell\neq0
	\\
	0 
	&
	\text{otherwise}
	\end{cases}
	\]
	on the first tensor fact and by the identity on the second tensor factor.

	\mybfheading{3) Identification of edge maps}
	For the second part of the lemma, let \(a_0\) and \(a_1\) be the arcs in \(\Diag\) and \(\Diag'\), respectively, in whose neighbourhoods these two diagrams differ. We distinguish the following two cases: 
	
	\myitheading{Case 1}
	Suppose one of \(a_0\) and \(a_1\) connects two open components. 
	Then so does the other. 
	So we can choose minimal subsets \(A\) and \(A'\) for \(\Diag\) and \(\Diag'\) such that \(A\smallsetminus\{a_0\}=A'\smallsetminus\{a_1\}\). 
	In this case, \(\mathfrak{a}=\mathfrak{a}(A)=\mathfrak{a}(A')\).
	Furthermore, by considering a pair of paths intersecting the diagram strands in a neighbourhood of the arc \(a_0\) and connecting the region \(x_0\) with \(x_1\) and \(x_2\) with \(x_3\), respectively, we see that \([x^a_0+x^a_1+x^a_2+x^a_3]=(x_0-x_1)+ (x_2-x_3)=z\in\tR/\mathfrak{a}\). 
	So clearly, the maps induced by \(M(\Diag_+)\) and \(M(\Diag_-)\) agree with the maps \(\mathcal{S}(\Diag,\Diag')\). 
	
	\myitheading{Case 2}
	Suppose that neither \(a_0\) nor \(a_1\) connect two open components. 
	Let \(a^*\) be an arc that connects two open components in \(\Diag\), and hence also in \(\Diag'\).
	Depending on whether \(\Diag\) has one more or one fewer component than \(\Diag'\), the edge corresponds to a merge or split map.
	We treat these two cases separately:
	
	\myitheading{Merge maps} 
	Suppose \(\Diag'\) is obtained from \(\Diag\) by merging two components. Let \(A_1\) be an arc system extending \(\{a^*\}\) for \(\Diag'\); note that \(a_1 \not\in A_1\).
	Then \(A_0=A_1\cup\{a_0\}\) is an arc system for~\(\Diag\). 
	Let \(\mathfrak{A}\) be the ideal generated by \(\lambda(a)\) for arcs \(a\not\in A_0\) of \(\Diag\) and \(\kappa(a)\) for all \(a\in A_1\smallsetminus \{a^*\}\).
	Then the left hand side of Figure~\ref{fig:delooping:mf:edge_maps} shows the special deformation retract for \(\tM(\Diag)\) obtained by applying the same procedure as in the construction of the map \(\varphi_{A_0}\), except that we do not simplify the matrix factorization for the arc \(a_0\). 
	The right hand side is a special deformation retract of \(\tM(\Diag')\). 
	It can be obtained from the first by replacing the matrix factorization for the arc \(a_0\) by the one for \(a_1\), noting that \(o(\Diag')=o(\Diag)\).
	The map induced by \(f\) is indicated by the two diagonal arrows. Only one of them matters to us, namely the diagonal arrow from the top left to the bottom right labelled by the identity. 
	This is because, in order to obtain \(V(\Diag)\otimes\tM(o(\Diag))\) and \(V(\Diag')\otimes\tM(o(\Diag))\), we need to eliminate \(\kappa(a_0)\) and \(\lambda(a_1)\), respectively, using Lemma~\ref{lem:lin_sp_def_retr}. To express the induced map in terms of our basis given by products of \(y_j\), we do a case analysis. 
	
	Suppose the saddle map merges two closed segments \(c_i\) and \(c_j\) of \(\Diag\) to some closed segment \(c_k\) of \(\Diag'\). Let \(C\) be the index set of closed segments not involved in the merge operation.  Observe that as elements of \(\tR/\mathfrak{a}(A_1)\), \([Y_i]=[Y_k]=[Y_j]\) and \([Y_k]^2=H\cdot[Y_k]+\Delta_\star\), where \(c_\star\) is the open segment that \(c_k\) is connected to via arcs in \(A_1\). Thus, the induced map is given by 
	\[
	y_i^{\varepsilon_i}y_j^{\varepsilon_j}
	\textstyle\prod_{\ell\in C} y_\ell^{\varepsilon_\ell}
	\mapsto
	\begin{cases}
	\textstyle\prod_{\ell\in C} y_\ell^{\varepsilon_\ell}
	&
	\varepsilon_i=0=\varepsilon_j
	\\
	y_k\textstyle\prod_{\ell\in C} y_\ell^{\varepsilon_\ell}
	&
	\{\varepsilon_i,\varepsilon_j\}=\{0,1\}
	\\
	(H\cdot y_k+\Delta_\star)\textstyle\prod_{\ell\in C} y_\ell^{\varepsilon_\ell}
	&
	\varepsilon_i=1=\varepsilon_j
	\end{cases}
	\]
	Observe that the endomorphism of \(\tM(o(\Diag))\) given by \(\Delta_\star\cdot\id\) is null-homotopic. So the induced map is homotopic to \(\mathcal{S}(\Diag,\Diag')\). 
	
	Suppose the saddle map merges a closed segment \(c_i\) with an open segment \(c_\star\). Let \(C\) be as above. Then \([Y_i]=[Y_\star]\in\tR/\mathfrak{a}(A_1)\) and the induced map is given by 
	\begin{equation}\label{eq:delooping:mf:merge_with_open}
	y_i^{\varepsilon_i}\textstyle\prod_{\ell\in C} y_\ell^{\varepsilon_\ell}
	\mapsto
	\begin{cases}
	\prod_{\ell\in C} y_\ell^{\varepsilon_\ell}
	&
	\varepsilon_i=0
	\\
	[Y_\star]\prod_{\ell\in C} y_\ell^{\varepsilon_\ell}
	&
	\varepsilon_i=1
	\end{cases}
	\end{equation}
	where we regard \([Y_\star]\) as an element of \(\tR^\partial\). The desired and actual values of \([Y_\star]\) are shown in the second and third column of Table~\ref{tab:delooping:mf:edge_maps}. Observe that they either agree or differ by the summand \(x\) for \(o(\Diag)=\Diag_0\) and the summand \(y\) for \(o(\Diag)=\Diag_1\). The endomorphism of \(\tM(\Diag_0)\) given by \(x\cdot\id\) is nullhomotopic, and so is the endomorphism of \(\tM(\Diag_1)\) given by \(y\cdot\id\). So the induced map is homotopic to \(\mathcal{S}(\Diag,\Diag')\). 
	
	\begin{table}[t]
		\centering
		\newcommand{\DW}[2]{\(#1\)\,\colorbox{lightgray}{\(+#2\)}}
			\begin{tabular}{c|c|c|c|c}
				\toprule
				&
				\multicolumn{2}{c|}{merge maps}
				&
				\multicolumn{2}{c}{split maps}
				\\
				\(\star\)
				&
				\(o(\Diag)=\Diag_0\)
				&
				\(o(\Diag)=\Diag_1\)
				&
				\(o(\Diag)=\Diag_0\)
				&
				\(o(\Diag)=\Diag_1\)
				\\
				\midrule
				\(H\)
				&
				\(H\)
				&
				\(H\)
				&
				\(y_i\)
				&
				\(y_i\)
				\\
				\(H+x\)
				&
				\DW{H}{x}
				&
				\DW{z}{y}
				&
				\DW{y_i}{x}
				&
				\(y_i+x\)
				\\
				\(z\)
				&
				\(z\)
				&
				\(z\)
				&
				\DW{y_i+y}{x}
				&
				\DW{y_i+x}{y}
				\\
				\(H+y\)
				&
				\DW{z}{x}
				&
				\DW{H}{y}
				&
				\(y_i+y\)
				&
				\DW{y_i}{y}
				\\
				\bottomrule
			\end{tabular}
		\medskip
		\caption{%
			Comparison of the edge maps induced by merging/splitting with an open component. 
			The first column indicates the open diagram segment relative to \(a^*\) that is involved in the merge/split operation. 
			The non-highlighted entries in the second and third column show the desired values of \([Y_\star]\in\tR^\partial\) for the map in \eqref{eq:delooping:mf:merge_with_open} to agree with \(\mathcal{S}(\Diag,\Diag')\);
			the complete entries show the actual values of \([Y_\star]\in\tR^\partial\). 
			Similarly, the last two columns show the desired values of \((y_i+[Y_\star]+H)\) for the map in \eqref{eq:delooping:mf:split_from_open} to agree with \(\mathcal{S}(\Diag,\Diag')\); the full entries show their actual values.}
		\label{tab:delooping:mf:edge_maps}
	\end{table}
	
	\myitheading{Split maps}
	Suppose \(\Diag'\) is obtained by splitting a component of \(\Diag\) into two.
	Let \(A_0\) be an arc system extending \(\{a^*\}\) for \(\Diag\); note that $a_0 \not\in A_0$.
	Then \(A_1=A_0\cup\{a_1\}\) is an arc system for \(\Diag'\). 
	Let \(\mathfrak{A}\) be the ideal generated by \(\lambda(a)\) for arcs \(a\not\in A_1\) of \(\Diag'\) and \(\kappa(a)\) for all \(a\in A_0\smallsetminus\{a^*\}\).
	As in the case of the merge maps, we apply special deformation retractions to all factors but the ones for \(a^*\), \(a_0\), and \(a_1\). We obtain the same diagram as before, namely the one in Figure~\ref{fig:delooping:mf:edge_maps}.
	However, when passing to \(V(\Diag)\otimes\tM(o(\Diag))\) and \(V(\Diag')\otimes\tM(o(\Diag))\), we now need to eliminate \(\lambda(a_0)\) and \(\kappa(a_1)\), respectively. So the edge map is induced by the map indicated by the diagonal arrow from the bottom left to the top right labelled \((x^{a_1}_0+x^{a_1}_1+x^{a_1}_2+x^{a_1}_3)\otimes\id\). 
	By inspection of the diagrams \(\Diag_0\) and \(\Diag_1\), one can see that  \((x^{a_1}_0+x^{a_1}_1+x^{a_1}_2+x^{a_1}_3)=[Y_\diamond]+[Y_{\vartriangle}]+H\in\tR/\mathfrak{a}(A_1)\), where \(c_\diamond\) and \(c_\vartriangle\) are the two segments connected by the arc \(a_1\). 
	Again, we distinguish two subcases:
	
	Suppose the saddle map splits a closed segment \(c_k\) of~\(\Diag\) into two closed segments \(c_i\) and \(c_j\) of~\(\Diag'\). 
	Let \(C\) be the index set of closed segments not involved in the merge operation. 
	Let \(c_\star\) be the open segment that \(c_k\) is connected to via arcs in \(A\). Since $[Y_k]=[Y_i]=[Y_j]\in \tR/\mathfrak{a}(A_0)$, we have $y_k\mapsto y_i(y_i+y_j+H)=y_j(y_i+y_j+H)=(y_i y_j+\Delta_\star)$, and the full induced map is given by 
\[
y_k^{\varepsilon_k}
\textstyle\prod_{\ell\in C} y_\ell^{\varepsilon_\ell}
\mapsto
\begin{cases}
(y_i+y_j+H)\prod_{\ell\in C} y_\ell^{\varepsilon_\ell}
&
\varepsilon_k=0
\\
(y_iy_j+\Delta_\star)\textstyle\prod_{\ell\in C} y_\ell^{\varepsilon_\ell}
&
\varepsilon_k=1
\end{cases}
\]
As in the case of the merge maps, we now use the fact that the endomorphism of \(\tM(o(\Diag))\) given by \(\Delta_\star\cdot\id\) is null-homotopic to conclude that the induced map is homotopic to \(\mathcal{S}(\Diag,\Diag')\). 

Suppose the saddle map splits a closed segment \(c_i\) off an open segment \(c_\star\). Let \(C\) be as above. Then \([Y_i]=[Y_\star]\in\tR/\mathfrak{a}(A_0)\) and the induced map is given by 
\begin{equation}\label{eq:delooping:mf:split_from_open}
\textstyle\prod_{\ell\in C} y_\ell^{\varepsilon_\ell}
\mapsto
(y_i+[Y_\star]+H)
\textstyle\prod_{\ell\in C} y_\ell^{\varepsilon_\ell}
\end{equation}
The desired and actual values of \((y_i+[Y_\star]+H)\) are shown in the last two columns of Table~\ref{tab:delooping:mf:edge_maps}. We now conclude as for merge maps. 
\end{proof}

In order to relate the multifactorization \(\tM(\Diag_T)\) to the invariant \(\BNr(T)\), we need to recall some details about the dotted cobordism category. There are several different versions of this category, which correspond to various different Frobenius extensions. Bar-Natan's original version corresponds to the $\mathcal F_1$ Frobenius extension \cite[Section~11.2]{BarNatanKhT}. 
In~\cite{KWZ}, we worked with a more general version that corresponds to $\mathcal F_7$. The objects of this category, which we will denote by \(\Cobb\), are crossingless diagrams in a disk, possibly with some number of endpoints on the boundary. 
For example, any of those crossingless diagrams \(\Diag\) considered above gives rise to an object in \(\Cobb\) simply by forgetting all dotted and thick arcs. 
The morphisms in \(\Cobb\) are \(\Z[H]\)-linear combinations of orientable abstract cobordisms that are decorated by dots, considered up to boundary preserving homeomorphisms, moving dots freely on their components, and the following relations: (cf~\cite[Page~1493]{BarNatanKhT})
\begin{equation}\label{eq:cobb_relations}
	\SpherePic=0,~  \Spheredot=1,~ \planedotdot=H\cdot\planedot\,,~ \tube=\DiscLdot\DiscR+\DiscL\DiscRdot -H\cdot \DiscL \DiscR
\end{equation}
We work in the reduced setting, which means that we mark one distinguished tangle end (the same one for all tangles) by an asterisk \(\ast\) and set any cobordism equal to 0 if the component containing this tangle end is marked by a dot: 
\[
\planedotstar =0
\] 
Given a diagram $\Diag_T$ of a tangle \(T\), the usual cube-of-resolutions construction results in a chain complex \(\KhTb{\Diag_T}\) over \(\Cobb\). The homotopy equivalence class of \(\KhTb{\Diag_T}\) denoted by \(\KhTb{T}\) is a tangle invariant, and the proof follows from Bar-Natan's original arguments \cite{BarNatanKhT}. When comparing this to Section~\ref{sec:review:Kh:definition} or~\cite{KWZ}, note that by \cite[Remark~4.12]{KWZ}, the category that we denote here by \(\Cobb\) is equivalent to the undotted cobordism category denoted by \(\Cobl\). 
\begin{example}
	The endomorphism algebra \(\End_{\Cob_{\bullet}}(\Li \oplus \Lo)\) is equal to the algebra \(\B\) from Equation~\eqref{eq:B_quiver}. Under this isomorphism, the saddle cobordisms correspond to \(S_{\bullet}\) and \(S_{\circ}\) and the identity cobordisms with a single dot on the non-special components correspond to \(D_{\bullet}\) and \(D_{\circ}\).
\end{example}
Just as in Section~\ref{sec:review:Kh:definition}, the isomorphism \(\End_{\Cob_{\bullet}}(\Li \oplus \Lo) \cong \BNAlgH\), in conjunction with delooping, allows to view \(\KhTb{T}\) as a type D structure $\DD(T)^{\BNAlgH}$.

\begin{definition}\label{def:induced_saddle_maps:cob}
	Given two crossingless diagrams \(\Diag\) and \(\Diag'\) as in Definition~\ref{def:induced_saddle_maps:mf}, consider them as objects in \(\Cobb\). Then, using the same notation as in Definition~\ref{def:induced_saddle_maps:mf}, define the map 
	\[
	\mathcal{S}'(\Diag,\Diag')\co
	V(\Diag)
	\otimes
	o(\Diag)
	\rightarrow
	V(\Diag')
	\otimes
	o(\Diag')
	\]
	as the image of \(\mathcal{S}(\Diag,\Diag')\) under the map \(\id_{[V(\Diag)^*\otimes V(\Diag')]}\otimes \Phi\), where \(\Phi\) is the isomorphism from Lemma~\ref{lem:iso_A_B}. 
	In other words, it is obtained from the rules for the map \(\mathcal{S}(\Diag,\Diag')\) in Definition~\ref{def:induced_saddle_maps:mf} by replacing \(s_0\), \(s_1\), \(d_0\), and \(d_1\) by \(S_{\bullet}\), \(S_{\circ}\), \(D_{\bullet}\), and \(D_{\circ}\), respectively. 
\end{definition}

The following lemma is the analogue of Lemma~\ref{lem:delooping:mf} in the setting of Bar-Natan's cobordism category. Its proof is much simpler.

\begin{lemma}\label{lem:delooping:cob}
	There exists an isomorphism of objects in $\Cobb$:
	\[
	\psi_\Diag
	\co 
	\Diag
	\rightarrow
	V(\Diag)
	\otimes
	o(\Diag)
	\]
	Moreover, if \(\Diag'\) is a diagram which differs from \(\Diag\) in a single dotted/thick arc, the saddle map \((\Diag\rightarrow\Diag')\in\ \Mor_{\Cobb}(\Diag,\Diag')\) induces the map \(\mathcal{S}'(\Diag,\Diag')\in \Mor_{\Cobb}(V(\Diag) \otimes o(\Diag), V(\Diag') \otimes o(\Diag')) \).
\end{lemma}

\begin{proof}
	The isomorphism \(\psi_\Diag\) is constructed using Bar-Natan's delooping procedure, see~\cite[Observation~4.18]{KWZ}. If \(c_i\) is a closed component of \(\Diag\), the following two maps are inverse to each other:
	\[
	\left(
		\begin{tikzcd}[row sep=0.3cm, column sep=0.5cm]
			\Circle_i
			\arrow{r}{\mathbf{f}_i}
			&
			\big(V(c_i)\otimes\varnothing\big)
			\arrow{r}{\mathbf{g}_i}
			&
			\Circle_i
		\end{tikzcd}
		\right)
		=
		\left(
		\begin{tikzcd}[row sep=0.3cm, column sep=0.5cm]
			&
			i^{1}h^0 (y_i\otimes\varnothing)
			\arrow[dr,"\DiscRdot-H\cdot\DiscR", "\mathbf{g}_i" below]
			\\
			\Circle_i
			\arrow[ur,"\DiscL","\mathbf{f}_i" below]
			\arrow[dr, swap, "\DiscLdot", "\mathbf{f}_i" above]
			&&
			\Circle_i
			\\
			&
			i^{1}h^0(1_i\otimes\varnothing)
			\arrow[ur, swap, "\DiscR", "\mathbf{g}_i" above]
		\end{tikzcd}
		\right)
	\]
	The isomorphism \(\psi_\Diag\) is constructed by tensoring these isomorphisms together.
	
	To prove the second statement, let us assume for a moment the more general setup in which \(\Diag\) and \(\Diag'\) are any diagrams of crossingless Conway tangles, ie not necessarily connected by a saddle move. 
	In~\cite{KWZ}, we called a cobordism representing an element in \(\Cobb(\Diag,\Diag')\) \textbf{simple} if all its components are disks, the distinguished component carries no dot and any other component carries at most one dot. By applying the relations \eqref{eq:cobb_relations}, it is easy to see that any element in \(\Cobb(\Diag,\Diag')\) can be written as an \(\fieldTwoElements[H]\)-linear combination of simple cobordisms; in fact, simple cobordisms form an \(\fieldTwoElements[H]\)-linear basis of \(\Cobb(\Diag,\Diag')\) \cite[Proposition~4.15]{KWZ}. From this decomposition into simple cobordisms, it is clear that \(\Cobb(\Diag,\Diag')\) can be written as the tensor product of \(\Cobb(o(\Diag),o(\Diag'))\) over \(\fieldTwoElements[H]\) with
		\begin{equation}\label{eqn:cob:only_circles}
		\Big(
		\bigotimes_{i=1}^{n}
			\Cobb\big(\Circle_i,\varnothing\big)
		\Big)
		\otimes_{\fieldTwoElements[H]}
		\Big(
			\bigotimes_{j=1}^{m}
			\Cobb\big(\varnothing,\Circle_j\big)
		\Big)
	\end{equation}
	On each of these tensor factors, we modify the basis as follows:
	\begin{align*}
		\Cobb\big(\Circle_i,\varnothing\big)
		&= 
		\fieldTwoElements[H]
		\Big\langle
			\prescript{}{i}{\DiscLdot}, \prescript{}{i}{\DiscL}
		\Big\rangle 
		\quad\text{for }i=1,\dots,n
		\\
		\Cobb\big(\varnothing,\Circle_j\big)
		&= 
		\fieldTwoElements[H]
		\Big\langle
			\DiscR_j, \DiscRdot_j-H\cdot\DiscR_j
		\Big\rangle 
		\quad\text{for }j=1,\dots,m
	\end{align*}
	This new basis has the advantage that it is compatible with the isomorphism \(\psi_\Diag\) in the sense that the basis on the space of \(\fieldTwoElements[H]\)-linear maps
	\[
	V(\Diag)
	\rightarrow
	V(\Diag')
	\]
	given by tensor products of \(y_i^*\), \(1_i^*\), \(y_j\), and \(1_j\)
	corresponds to the new basis on \eqref{eqn:cob:only_circles}
	via the following dictionary:
	\[
		\prescript{}{i}{\DiscL} \leftrightarrow y_i^*
		\quad
		\prescript{}{i}{\DiscLdot} \leftrightarrow 1_i^*
		\quad
		\DiscR_j \leftrightarrow 1_j
		\quad
		\DiscRdot_j-H\cdot\DiscR_j \leftrightarrow y_j
	\]
	It now remains to verify that in the case that \(\Diag\) and \(\Diag'\) are related by a single saddle cobordism \(C\), the map \(\psi_{\Diag'}^{-1}\circ C\circ \psi_{\Diag}\) agrees with \(\mathcal{S}'(\Diag,\Diag')\). Given the above dictionary, this computation is straightforward and similar to the one carried out in \cite[Proposition~4.31]{KWZ}, so we leave it to the reader.
\end{proof}

\subsection{Reframing: from multifactorizations to type~D structures}\label{sec:reframing}

As noted in Remark~\ref{rem:mf:cube_of_resolutions}, we may regard the multifactorization \(\tM(\Diag_T)\) as a cube of resolutions at whose vertices \(v\in \{0,1\}^{n}\) there are matrix factorizations \((\tM(\Diag_T(v)),d_0)\) and at whose edges, there are differentials induced by the vertical maps in Figure~\ref{fig:crossings_mf}. 
The special deformation retracts \(\varphi_{\Diag_T(v)}\) from Lemma~\ref{lem:delooping:mf} give rise to a special deformation retract of the whole multifactorization \(\tM(\Diag_T)\), see~\cite[Proposition~2.7]{Ballinger}. 
Let us denote the resulting ``delooped'' multifactorization by \(\tM_{\varnothing}(\Diag_T)\). 
At each vertex \(v\) of the cube for \(\tM_{\varnothing}(\Diag_T)\), there is now a matrix factorization \(V(\Diag_T(v))\otimes_{\fieldTwoElements[H]}\tM(o(\Diag_T(v)))\). Moreover, the differentials along each edge \(v\rightarrow v'\) of the cube are morphisms that agree with the maps \(\mathcal{S}(\Diag_T(v),\Diag_T(v'))\) from Definition~\ref{def:induced_saddle_maps:mf} up to homotopy. However, there are possibly also higher differentials \(d_{k}\) for \(k>1\). 

The matrix factorization \((\tM_{\varnothing}(\Diag_T),d_0)\) is equal to a direct sum of copies of either \(\tM(\Diag_0)\) or \(\tM(\Diag_1)\) (possibly shifted in bigrading).
We can therefore regard all components of the differentials \(d_k\) for \(k\geq1\) as elements of the endomorphism algebra \(\A\) from Definition~\ref{def:dg_algebra}. 
Hence, by replacing each \(\tM(o(\Diag_T(v)))\) by a generator in idempotent \(\id_{\tM(o(\Diag_T(v)))} \in \A\), we may regard \(\tM_{\varnothing}(\Diag_T)\) as a type~D structure \(\tM(\Diag_T)^{\A}\) over \(\A\).
If we unwrap the \(D^2=0\) relations in the multifactorization \(\tM_{\varnothing}(\Diag_T)\) (where the differential is \(D=d_0+d_1+d_2+\cdots\)), we obtain exactly the compatibility relation for the type~D structure  \(\tM(\Diag_T)^{\A}\) (where the differential is \(\delta^1=d_1+d_2+\cdots\)). Moreover, chain maps and $1$-homotopies between multifactorizations translate into type~D structure homomorphisms and homotopies.

Considering the first order of the differential \(\delta^1=d_1+d_2+\cdots\) in \(\tM(\Diag_T)^{\A}\), the compatibility relation \((d_1)^2=d_0 \circ d_2 + d_2 \circ d_0= d^{\A}(d_2)\) implies that components of \((d_1)^2\) are null-homologous in \(\A\). Thus, if we consider the type~D structure \(\tM(\Diag_T)^{\Hast(\A)}\) with the same generators and the differential \(\delta^1=[d_1]\), the compatibility relation \([d_1]^2=0\) holds.

% (The homological perturbation lemma (see for example~\cite[Section~1i]{Seidel}) implies that there is a well-defined \(A_\infty\) structure on the algebra \(\Hast(\A)\). It will play a key role in Section~\ref{sec:quasi-iso}, but for now we consider the homology \(\Hast(\A)\) as an ordinary algebra, with the multiplication induced from \(\A\).)

\begin{corollary}\label{cor:identification_of_type_D_structures}
	Under the isomorphism \(\Phi\) from Lemma~\ref{lem:iso_A_B}, the type~D structures  \(\DD(T)^{\B}\) and \(\tM(\Diag_T)^{\Hast(\A)}\) are isomorphic. 
\end{corollary}

\begin{proof}
This is an immediate consequence of Lemmas~\ref{lem:delooping:mf} and~\ref{lem:delooping:cob}.
\end{proof}

We now address the issue of gradings. 
\begin{definition}\label{def:deformed_dg_algebra}
Define a deformation $\AnUU$ of the algebra \(\A\) via adjoining a variable~\(U\), \(\text{gr}(U)=i^0h^{-1}q^{-3}\):
\[\AnUU=\A\otimes \F[U,U^{-1}], \qquad d^{\AnUU}(x \otimes U^k)=d^{\A}(x) \otimes U^{k-1}\]
\end{definition}
Note that the differential of this algebra is deformed by \(U^{-1}\), compared to the algebra \(\A\). This makes the algebra \(\AnUU\) a \emph{bigraded dg algebra} in the usual sense, that is the multiplication preserves the bigrading, and the differential preserves the quantum and raises the homological grading by one.

Next, note that so far we worked with filtered differentials in multifactorizations, according to \cite[Definition~2.2]{Ballinger}. We now adjoint a variable \(U\), and require the differentials \(d_k\) to pick up \(U^{k-1}\): 
\[(d_0,d_1,d_2,d_3,\ldots) \ \mapsto \ (U^{-1} d_0, d_1,U d_2,U^2 d_3,\ldots),\]
We have 
\[\text{gr}(d_k)=i^{-3}h^k q^{-3+3k} \ \implies \ \text{gr}(U^{k-1}d_k )=i^{-3}h^1q^0\]
In other words, gradings now behave in the same way as in Khovanov homology, preserving the quantum and raising by one the homological. The chain maps and $1$-homotopies between multifactorizations are changed accordingly, namely \[(F_0,F_1,F_2,\ldots)\ \mapsto \ (F_0,U F_1,U^2 F_2,\ldots), \quad ( H_{-1}, H_0, H_1,\ldots) \ \mapsto \ (H_{-1},U  H_0, U^2 H_1,\ldots)\] 

We will call a map \emph{positive} if it does not pick up a negative power of \(U\). Note that all the differentials (with the exception of \(d_0\)), chain maps, and $1$-homotopies are positive.

With the above changes in mind, we claim that the delooped multifactorization \(\tM_\varnothing(T)\) can be interpreted as a \emph{bigraded} type~D structure \(\tM(\Diag_T)^{\AnUU}\). The differential \(\delta^1=d_1+U d_2 + U^2 d_3 + \cdots\) preserves the quantum and raises the homological grading by one. Moreover, the differential is positive, since the \(d_0\) is not included in \(\delta^1\) (because it is absorbed into the generators of \(\tM(\Diag_T)^{\AnUU}\)).

In conclusion, given a pointed Conway tangle \(T\), we have now constructed a bigraded type~D structure 
\[\tM(\Diag_T)^{\AnUU}\] 
with a positive differential. 
% which is a tangle invariant up to positive homotopy equivalence of type~D structures, which means that the bigraded homomorphisms and bigraded homotopies are positive.

\subsection{Quasi-isomorphic algebras via homological mirror symmetry} \label{sec:quasi-iso} 
	\begin{theorem} \label{thm:from_hms}
	There is a quasi-isomorphism between the following two \(A_\infty\) algebras:
	\[\text{(Definition~\ref{def:dg_algebra})}\quad \A \  \simeq \  \Binf \quad \text{(Definition~\ref{def:Binf_algebra})}\]
	\end{theorem}
	\begin{proof} 
	This is a consequence of homological mirror symmetry for the three-punctured sphere. Consider the following categories: 
	\begin{itemize}
	\item The \(\Z/2\)-graded \(A_\infty\) category of twisted complexes over the wrapped Fukaya category of the three-punctured sphere 
	\(\Tw \W (\ThreePuncturedSphere)\), defined according to~\cite{AAEKO}.
	\item The \(\Z\)-graded dg-enhancement \(\Coh(X_0)\) of the derived category of bounded coherent complexes of sheaves on \(X_0=\{xyz=0\}\subset \mathbb A^3\). 
	Because \(X_0\) is an affine variety, \(\Coh(X_0)\) can be defined as the dg category of dg modules \(N\) over \(\F[x,y,z]/(xyz)\) that are coherent and bounded (ie whose homologies \(\text{H}_i(N)\) are finitely generated and vanish for \(i\ll0\) and \(i\gg 0\)),
	quotiented out by the full dg subcategory of acyclic modules, using the quotient construction for dg categories~\cite{Drinfeld, Keller}. 
	There are other models for \(\Coh(X_0)\), including \(A_\infty\) and projective modules (see \cite[Proposition~2.4.1]{LOTBimodules}), and all of them are quasi-equivalent~\cite{Lunts_Orlov}. 

	\item The full dg subcategory \(\Perf(X_0) \subset \Coh(X_0)\) consisting of perfect dg modules (ie quasi-isomorphic to a dg module \(N\) such that \(N_i\) are finitely generated and projective, and \(N_i = 0\) for \(i\gg0\) and \(i\ll0\)).

	\item The \(\Z/2\)-graded dg quotient \(D_{\sg}(X_0)=\Coh(X_0)/\Perf(X_0)\) called the category of singularities of \(X_0\), introduced in~\cite{Orlov}. The \(\Z/2\)-grading on \(D_{\sg}(X_0)\) comes from the two-periodicity of \(\Z\)-grading, see~\cite[Sections~2.1 and~2.2]{Nadler}.

	\item The \(\Z/2\)-graded dg category \(\MF(\mathbb A^3,xyz)\)  of matrix factorizations over \(\F[x,y,z]\), ie free modules  \(V_0 \oplus V_1\) with maps \(p_0:V_0\to V_1\) and \(p_1:V_1\to V_0\) such that \(p_0p_1 + p_1 p_0=xyz\cdot \id\). 
	\end{itemize}
	The following are two quasi-equivalences of $\Z/2$-graded $A_\infty$ categories: 
	\begin{align}
	&\Tw\W(\ThreePuncturedSphere) \simeq D_{\sg}(X_0) & \quad &\text{\cite{AAEKO}} \label{eq:HMS}\\
	&\MF(\mathbb A^3,xyz) \simeq D_{\sg}(X_0) & \quad &\text{\cite{Orlov}} \label{eq:Orlov}
	\end{align}
	The first describes the homological mirror symmetry between the three puncture sphere \(\ThreePuncturedSphere\) and the Landau-Ginzburg model \((\mathbb A^3, W=xyz)\). The second establishes the matrix factorization model for the category of singularities. 
	Combining the quasi-equivalences above, we obtain \[\MF(\mathbb A^3,xyz) \simeq \Tw\W(\ThreePuncturedSphere)\] 
	Moreover, from~\cite[Theorem~6.1]{AAEKO} it follows that under quasi-equivalence~\eqref{eq:HMS} the two generating arcs \(\arcVer\) and \(\arcHor\) in Figure~\ref{fig:wrap_subcat} correspond to dg modules with vanishing differentials \(\F[x,y,z]/(x)\) and \(\F[x,y,z]/(y)\) over the ring \(\F[x,y,z]/(xyz)\), respectively.
	Under quasi-equivalence~\eqref{eq:Orlov}, these two modules correspond to the two basic matrix factorizations from Example~\ref{exa:elementary_mf}, respectively (see~\cite[Section~3.2]{Orlov}). 
	Thus, we obtain the quasi-isomorphism of the corresponding endomorphism algebras: 
	\begin{equation*}
		\End_{\MF(\mathbb A^3,xyz)}(M(\Diag_0) \oplus M(\Diag_1))=\A \simeq \Binf = \End_{\W(\ThreePuncturedSphere)}(\arcVer \oplus \arcHor)
		\qedhere
	\end{equation*}
	\end{proof}

\begin{remark}
	We also note that 
	\(\MF(\mathbb A^3,xyz)\simeq \Coh(\Spec \F[x,y]/(xy))\)~\cite[Proposition~2.3]{Nadler}, and the corresponding version of homological mirror symmetry for \(\ThreePuncturedSphere\) was proved in~\cite{LekPol}: \(\text{H}_0(\Tw\W(\ThreePuncturedSphere)) \simeq D^b \Coh(\Spec \F[x,y]/(xy)).\)
\end{remark}
Theorem~\ref{thm:from_hms} holds in the bigraded deformed setting as well:
	\begin{theorem} \label{thm:quasi-iso}
	There is a quasi-isomorphism between the following two bigraded \(A_\infty\) algebras
	\[\text{(Definition~\ref{def:deformed_dg_algebra})}\quad \AnUU \  \simeq \  \BinfUU \quad \text{(Definition~\ref{def:deformed_Binf_algebras})}\]
	Moreover, the quasi-isomorphism can be given by two \(A_\infty\) homomorphisms \(F^U\) and \(G^U\) that are positive on \(\A =\A\otimes 1 \subset \AnUU\).
	\end{theorem}
	\begin{proof}
	Suppose the quasi-isomorphism between \(\A\) and \(\Binf\) from Theorem~\ref{thm:from_hms} is given by 
	\begin{align*}
		F&=(F_0,F_1,F_2,\ldots),\quad F_k\co\A^{k+1}\to\Binf
		\\
		G&=(G_0,G_1,G_2,\ldots),\quad G_k\co(\Binf)^{k+1}\to \A
	\end{align*}
	Because the homological grading of all elements in \(\A\) and \(\Binf\) vanishes, we can promote \(F\) and \(G\) to \(h\)-graded (ie \(\text{gr}(F^U_k)=\text{gr}(G^U_k)=h^{-k}\)) quasi-isomorphisms of deformed algebras
	\[
	F^U\co\AnUU \xrightarrow{\simeq} \BinfUU, \ F^U_{k}(U^{i_{0}} a_0,\ldots,U^{i_{k}}a_{k})=F_{k}( a_0,\ldots,a_{k}) \otimes U^{i_0 +\cdots +i_k} \cdot U^k  \]
	and the same formula for \(G^U\co\BinfUU \xrightarrow{\simeq} \AnUU\)

	We now make the quasi-isomorphism \(F^U\) preserve the quantum grading. First, note that we know the map on cycles \(F_0:\mathrm{Z}(\A) \to \Binf\) is the projection \(\mathrm{Z}(\A) \to \Hast(\A)\cong \Binf \cong \B\), which preserves the quantum grading, since the isomorphism \(\Phi\) from Lemma~\ref{lem:iso_A_B} does. 
	This implies that excluding all the components of \(F^U\) which do not preserve the quantum grading makes \(F^U\) preserve the quantum grading.

	The same procedure works to make \(G^U\) preserve the quantum grading.
	\end{proof}
	
The quasi-isomorphism \(\AnUU \simeq \BinfUU\) allows us to rewrite the type~D structure \(\tM(\Diag_T)^{\AnUU}\) from Section~\ref{sec:reframing} as a type D structure \(\tM(\Diag_T)^{\BinfUU}\) over \(\BinfUU\). 
The positivity of the differential in \(\tM(\Diag_T)^{\AnUU}\), together with the positivity of the quasi-isomorphisms \(F^U\) and \(G^U\) on \(\A\subset \AnUU\), implies the positivity of the differential in \(\tM(\Diag_T)^{\BinfUU}\). This makes it possible to view the type~D structure \(\tM(\Diag_T)^{\BinfUU}\) as a type~D structure over \(\BinfU\), which we denote by 
\(\DD(\Diag_T)^{\BinfU}\).

\begin{remark}
If \(M(\Diag_T)\) is indeed a tangle invariant, see Remark~\ref{prop:matrix_factorization_is_an_invariant}, then so is the type~D structure \(\tM(\Diag_T)^{\AnUU}\). Moreover, positive homotopy equivalences up to which \(\tM(\Diag_T)^{\AnUU}\) is defined induce positive homotopy equivalences for \(\tM(\Diag_T)^{\BinfUU}\). Therefore, if the multifactorization \(M(\Diag_T)\) is a tangle invariant, then so is \(\DD(\Diag_T)^{\BinfU}\).
\end{remark}

We are now ready to prove the extension property.
\begin{proof}[Proof of Theorem~\ref{thm:extension}]

The fact that \(\DD(\Diag_T)^{\BinfU}\big|_{U=0} = \DD(\Diag_T)^{\B}\) is a consequence of Corollary~\ref{cor:identification_of_type_D_structures}, together with the fact that \(F^U_0\) projects the cycles onto the homology: \(F^U_0: \mathrm{Z}(\A) \to \Hast(\A)\cong\B\). 
The statement about \(\DD_1(\Diag_T)^\B\) is completely analogous, except instead of \(\tM(\Diag_T)\) one should start with a multifactorization \([q^{-2}h^{-1}\tM(\Diag_T)\xrightarrow{H}\tM(\Diag_T)]\).

Finally, we need to show that the type~D structure \(\DD^c(T)^{\B}\) representing the curve \(\BNr(T)\) is extendable. Such a type~D structure is obtained from \(\DD(\Diag_T)\) via the so-called arrow sliding algorithm. This algorithm, which is described in~\cite[Section~5]{KWZ}, consists of a sequence of basic homotopies, namely cancellations and clean-ups. 
Lemma~\ref{lem:cancel_n_clean} below implies that extendability is preserved under such homotopies.
\end{proof}

\begin{lemma}\label{lem:cancel_n_clean}
	Suppose a type~D structure \(X^\B\) over \(\B\) is extendable to a type~D structure over \(\BinfU\), and \(Y^\B\) is obtained from \(X^\B\) by a single application of the Cancellation and Clean-Up Lemmas \cite[Lemmas~2.16 and~2.17]{KWZ} in the arrow sliding algorithm~\cite[Section~5]{KWZ}. Then \(Y^\B\) is also extendable to a type~D structure over \(\BinfU\).
\end{lemma}

\begin{proof}
	We can regard \(X^{\BinfU}\) as a type~D structure over \(\BinfUU\) and by Theorem~\ref{thm:quasi-iso} as a type~D structure over \(\AnUU\).
	By the positivity of the differential of \(X^{\BinfUU}\), the differential of \(X^{\AnUU}\) is also positive. 
	Note that since \(F^U_0\circ G^U_0=\id\), \(F^U(X^{\AnUU})|_{U=0}=F^U(G^U(X^{\BinfUU}))|_{U=0}=X^{\B}\), so let us replace  \(X^{\BinfUU}\) by  \(F^U(X^{\AnUU})\). 
	In both cases of this proof, the strategy is to find a homotopy of \(X^{\AnUU}\) to some type~D structure over \(\AnUU\) with the following properties: First, the corresponding type~D structure over \(\BinfUU\) should have a positive differential, so that it can be viewed as a type~D structure \(Y^{\BinfU}\) over \(\BinfU\). Secondly, \(Y^{\BinfU}\) should define an extension of \(Y^\B\). This is summarized in the following diagram: 
	\begin{equation}\label{eq:diagram_for_homotopies}
	\begin{tikzcd}[column sep = 2cm, row sep = 1cm, ampersand replacement = \&]
		X^{\AnUU}
		\arrow[d, "\simeq"left, "\substack{\text{cancellation}\\\text{or clean-up}}"right]
		\arrow[r,mapsto, "\text{Theorem~\ref{thm:quasi-iso}}"above]
		\&
		X^{\BinfUU}
		\arrow[r,mapsto, "\substack{\text{positivity of}\\\text{the differential}}"above]
		\&
		X^{\BinfU}
		\arrow[r,mapsto, "U=0"above]
		\&
		X^{\B}
		\arrow[d, "\simeq"right, "\substack{\text{cancellation}\\\text{or clean-up}}"left]
		\\
		Y^{\AnUU}
		\arrow[r,mapsto, "\text{Theorem~\ref{thm:quasi-iso}}"above]
		\&
		Y^{\BinfUU}
		\arrow[r,mapsto, "\substack{\text{positivity of}\\\text{the differential}}"above]
		\&
		Y^{\BinfU}
		\arrow[r,mapsto, "U=0"above]
		\&
		Y^{\B}
	\end{tikzcd}
	\end{equation}
	\noindent\textbf{Case 1: Clean-up. } 
	The clean-ups that are involved in the process of homotoping \(X^{\B}\) to a model associated with an immersed curve are quite simple: Each one consists of applying the Clean-Up Lemma to some \(g=(\mathbf{x}\xrightarrow{g_{\B}}\mathbf{y}) \in \End(X^{\B})\), where \(\mathbf{x}\) and \(\mathbf{y}\) are two distinct homogeneous generators of \(X^\B\) and \(\text{gr}(g)=q^0h^0\). The resulting type~D structure \(Y^{\B}\) has the same generators, while its differential is changed: 
	\(\delta^1_{Y}=\delta^1_{X}+ g \delta^1_{X}  + \delta^1_{X} g\). 
	Our task is to lift this process to the type~D structure \(Y^{\AnUU}\). 
	
	With the endomorphism \(g\) we associate \(\tg=(\mathbf{x}\xrightarrow{\tg_{\A}}\mathbf{y}) \in \End(X^{\AnUU})\), where \(\tg_{\A}\in \A \subset \AnUU\) corresponds to \(g_{\B}\) under the isomorphism \(\Phi\) from Lemma~\ref{lem:iso_A_B}. We now want to apply the Clean-Up Lemma to \(\tg\).  For this, we first check that the hypotheses are satisfied. By \cite[Remark~2.18]{KWZ}, it suffices to verify that \(\tg^2\), \(d_{\A}(\tg_{\A})\), and \(\tg \delta^1_X\tg\) vanish. The first identity is obvious and the second follows from the choice of \(\tg_{\A}\). For the third, observe that the homological grading of \(g\) vanishes, so \(h(\mathbf{x})=h(\mathbf{y})\). If there existed a non-zero component \(\mathbf{y} \xrightarrow{a\cdot U^k}\mathbf{x}\) of \(\delta^1_{X}\), then its homological grading is equal to \(1=h(\mathbf{x})-h(\mathbf{y})-k=-k\). So \(k=-1\), which contradicts our assumption that the differential of \(X^{\AnUU}\) is positive, ie \(k\geq 0\). 
	
	Let \(Y^{\AnUU}\) be the type~D structure obtained by applying the Clean-Up Lemma to \(\tg\). Its differential is positive since the additional components \(\tg \delta^1_{X}  + \delta^1_{X} \tg\) do not contain any negative powers of \(U\), and the same is true for the differential in \(Y^{\BinfUU}\). So we may view it as a type~D structure over \(\BinfU\), which, by construction, is an extension of \(Y^{\B}\).
	
	\medskip\noindent\textbf{Case 2: Cancellation. } 
	The cancellations that are involved in the process of homotoping \(X^{\B}\) to a model associated with an immersed curve are also quite simple: 
	Each one consists of applying the Cancellation Lemma to some component \((\mathbf{x}\xrightarrow{1}\mathbf{y})\) of the differential of \(X^{\B}\). 
	The resulting type~D structure \(Y^{\B}\) is obtained from \(X^{\B}\) by removing the generators \(\mathbf{x}\) and \(\mathbf{y}\), removing all components of the differential involving these generators, and then adding components \(\mathbf{z}_1 \xrightarrow{ba}\mathbf{z}_2\) for each zigzag \(\mathbf{z}_1 \xrightarrow{a} \mathbf{y} \xleftarrow{1} \mathbf{x} \xrightarrow{b} \mathbf{z}_2\). We now lift this process to \(X^{\AnUU}\). 
	
	The component of the differential of \(X^{\AnUU}\) that corresponds to \((\mathbf{x}\xrightarrow{1}\mathbf{y})\) in \(X^{\B}\) is given by an arrow \(\mathbf{x}\xrightarrow{e}\mathbf{y}\), where \([e]=1 \in \Hast(\A)\). Being part of the differential of \(X^{\AnUU}\), this arrow preserves quantum grading. By inspection, the quantum grading of any homogeneous endomorphism of \(\tM(\Diag_0)\) or \(\tM(\Diag_1)\) that does not contain any identity component is strictly less than 0. Since the differential of \(X^{\AnUU}\) is positive, this implies that in fact \(e=1\). 
	
	Let \(Y^{\AnUU}\) be the type~D structure obtained by cancelling \(e\). Then by construction, it fits into Diagram~\eqref{eq:diagram_for_homotopies}. 
\end{proof}
%!TEX root = ../main.tex
\section{\texorpdfstring{Further restrictions on \(\Khr(T)\)}{Further restrictions on Khr(T)}}\label{sec:further_restrictions}
\subsection{Geography of special components}
\begin{theorem}\label{thm:geography:special_curves}
	Every special component of \(\Khr(T)\) is equal to the curve \(\sKh_{2n}(\tfrac p q)\) for some $n\geq 1,~\tfrac p q \in \QPI$, equipped with the trivial local system.
\end{theorem}

\begin{lemma}\label{lem:H_is_nullhomotopic_on_Khr}
The morphism \((H\cdot\id)\) is null-homotopic on every direct summand of the type~D structure \(\DD^{c}_1(T)\) corresponding to a multicurve $\Khr(T)$.
\end{lemma}

\begin{proof}
	By definition, \(\DD^c_1(T)\) is obtained from
	\(
		\DD_1(T)\coloneqq\left[\DD(T)\xrightarrow{H\cdot\id} \DD(T)\right]
	\) by cancellations and clean-ups (see \cite[Section~5]{KWZ}). The morphism \(H\cdot\id_{\DD_1}\) is null-homotopic, since the null-homotopy can be given by the diagonal dashed arrow: 
  \[
  \begin{tikzcd}[column sep=2cm]
  \DD(T)
  \arrow{r}{H\cdot\id}
  \arrow{d}{H\cdot\id}
  &
   \DD(T)
   \arrow{d}{H\cdot\id}
   \arrow[dashed, ld, "\id",swap]
  \\
  \DD(T)
  \arrow{r}{H\cdot\id}
  &
   \DD(T)
  \end{tikzcd}
  \]
	Therefore, by Lemma~\ref{lem:morphisms_are_preserved}, \(H\cdot\id_{\DD^c_1(T)}\) is also null-homotopic. Let us write \(\DD^c_1(T)=\bigoplus_i (\DD_i,d_i)\), where each \((\DD_i,d_i)\) corresponds to the \(i^\text{th}\) component of \(\Khr(T)\). If \(F\) is a null-homotopy for \(H\cdot\id_{\DD^c_1(T)}\) and \(F_i\co(\DD_i,d_i)\rightarrow(\DD_i,d_i)\) denotes the restriction of \(F\) to the \(i^\text{th}\) component, \(F^\text{diag}\coloneqq\bigoplus_i F_i\) is also a null-homotopy for \(H\cdot\id_{\DD^c_1(T)}\), which restricts to null-homotopies for each \(H\cdot\id_{(\DD_i,d_i)}\).  
\end{proof}

\begin{proof}[Proof of Theorem~\ref{thm:geography:special_curves}]
	By naturality of \(\Khr\) under the action of the mapping class group (Theorem~\ref{thm:Kh:Twisting}), we may assume without loss of generality that the slope of the special component is zero. Denote the type D structure associated with this component by $\DD^\sKh$. Our task is to prove that $\DD^\sKh$ corresponds to the curve \(\sKh_{2n}(0)\) for some $n\geq 1$. 
	More explicitly, we need to show that $\DD^\sKh$ is equal to the following type~D structure containing \(4n\) generators in idempotent \(\iota_\bullet\):
  \[
  \begin{tikzcd}
    \DotC
    \arrow{r}{S}
    \arrow{d}[swap]{D}
    &
    \DotB
    \arrow{r}{D}
    &
    \DotB
    \arrow{r}{S^2}
    &
    \cdots
    &
    \cdots
    \arrow{r}{S^2}
    &
    \DotB
    \arrow{r}{D}
    &
    \DotB
    \arrow{r}{S}
    &
    \DotC
    \arrow{d}{D}
    \\
    \DotC
    \arrow{r}{S}
    &
    \DotB
    \arrow{r}{D}
    &
    \DotB
    \arrow{r}{S^2}
    &
    \cdots
    &
    \cdots
    \arrow{r}{S^2}
    &
    \DotB
    \arrow{r}{D}
    &
    \DotB
    \arrow{r}{D}
    &
    \DotC
  \end{tikzcd}
  \]

\begin{figure}[t]
  \centering
  \(\GeographyForSpecials\)
  \caption{Some curve segments of a special component of slope 0 in the covering space \(\PuncturedPlane\), illustrating the proof of Theorem~\ref{thm:geography:special_curves}}\label{fig:GeographyForSpecials}
\end{figure}

	First, suppose that the local system on the component is trivial. 
	Some generators of the type~D structure \(\DD^\sKh\) belong to the idempotent \(\iota_{\circ}\); otherwise, it would be a rational component of slope $0$. By Proposition~\ref{prop:higher_powers_when_wrapping}, the differential of \(\DD^\sKh\) only contains linear combinations of \(S\), \(S^2\), and \(D\). Therefore, it contains some differential \(\DotC\xrightarrow{D}\DotC\). It looks either like
	\begin{equation}\label{eq:geography:special_curves}
	\begin{tikzcd}
	\DotC
	\arrow{r}{S}
	\arrow{d}[swap]{D}
	&
	\DotB
	\arrow{r}{D}
	&
	\DotB
	\arrow{r}{}
	&
	\cdots
	\\
	\DotC
	\arrow{r}{S}
	&
	\DotB
	\arrow{r}{D}
	&
	\DotB
	\arrow{r}{}
	&
	\cdots
	\end{tikzcd}
	\quad
	\text{or}
	\quad
		\begin{tikzcd}
	\DotC
	\arrow[leftarrow]{r}{S}
	\arrow[leftarrow]{d}[swap]{D}
	&
	\DotB
	\arrow[leftarrow]{r}{D}
	&
	\DotB
	\arrow[leftarrow]{r}{}
	&
	\cdots
	\\
	\DotC
	\arrow[leftarrow]{r}{S}
	&
	\DotB
	\arrow[leftarrow]{r}{D}
	&
	\DotB
	\arrow[leftarrow]{r}{}
	&
	\cdots
	\end{tikzcd}
	\end{equation}
	Let us focus on the first case first. 
	A lift of the corresponding portion of the curve \(\Khr(T)\) to the covering space \(\PuncturedPlane\) is shown in the middle of Figure~\ref{fig:GeographyForSpecials}. 
	Let \(R\) and \(L\) be the number of consecutive generators in idempotent \(\iota_{\bullet}\) on the upper and lower legs of the type~D structure, respectively. These integers are equal to the number of intersection points of the vertical parametrizing arcs with the right and left segments of the curve up to the point where they intersect the horizontal parametrizing arcs again. In particular, this implies that \(R\) and \(L\) are even integers. 
	Suppose \(L\leq R\). If \(L=2\), there is a non-trivial \(\mu_4\)-action, which can be canceled if and only if \(R=2\) and the type~D structure looks like the one for the curve \(\sKh_2(0)\). If \(L>2\), the type~D structure looks as follows (without the dotted and dashed arrows): 
	\[
	\begin{tikzcd}
	\DotC
	\arrow{r}{S}
	\arrow{d}[swap]{D}
	&
	\DotB
	\arrow{r}{D}
	\arrow[dotted,in=135,out=-45]{drr}[description]{U^2}
	&
	\DotB
	\arrow{r}{S^2}
	\arrow[dotted,in=135,out=-45]{drr}[description]{U^2}
	&
	\cdots
	&
	\cdots
	\arrow{r}{S^2}
	\arrow[dotted,in=135,out=-45]{drr}[description]{U^2}
	&
	\DotB
	\arrow{r}{D}
	&
	\DotB
	\arrow[dashed]{r}{S}
	&
	\DotC
	\arrow[dashed]{d}{D}
	\\
	\DotC
	\arrow{r}{S}
	&
	\DotB
	\arrow{r}{D}
	&
	\DotB
	\arrow{r}{S^2}
	&
	\cdots
	&
	\cdots
	\arrow{r}{S^2}
	&
	\DotB
	\arrow{r}{D}
	&
	\DotB
	\arrow{r}{S}
	&
	\DotC
	\end{tikzcd}
	\]
	In this case, the \(\mu_4\)-action forces the existence of a \(U^2\)-differential in the extended type~D structure, namely the one indicated in the complex above by the first dotted arrow on the left. 
	As in the proof of Theorem~\ref{thm:no_wrapping_around_special}, the extended type~D structure contains no differential labelled \(U\) nor \(U\cdot S\), since the quantum gradings of generators in the same idempotent have the same parity, and generators in different idempotents different parity. This forces the existence of the other \(U^2\)-differentials in the complex above. 
	The contribution to the structure relation of the composition of the last \(U^2\)-differential with the differential \(\DotC\xrightarrow{S}\DotB\) on the lower leg of the complex can be cancelled if and only if \(R=L\) and the complex is equal to \(\sKh_{L}(0)\). 
	
	If the curve carries an \(n\)-dimensional local system \(X\), we may choose the corresponding complex to be the same as before, except that we tensor each generator by an \(n\)-dimensional vector space \(W\), replace the differential \(\DotC\xrightarrow{D}\DotC\) on the left by \(\DotC\otimes W\xrightarrow{D\otimes X}\DotC\otimes W\) and tensor all other differentials of \(\DD^\sKh\) by \(\id_{W}\). Then the \(U^2\)-differentials in the extended type~D structure need to be tensored by \(X\), and so does the final differential \(\DotC\xrightarrow{D}\DotC\) on the right. This means that the curve carries the local system \(X\cdot X^{-1}=\id\), thus proving the claim. 
	
	The second case in \eqref{eq:geography:special_curves} with \(L\leq R\) follows from reversing all arrows in the arguments above. So it remains to consider those two cases for \(L>R\). We claim that a curve containing a portion of this kind cannot be a component of \(\Khr(T)\). To see this, consider a shortest leg of length \(R\) of such a curve. By the previous arguments, the corresponding portion of the type~D structure looks as follows: 
	\[
	\begin{tikzcd}
		\cdots
		&
		\DotC
		\arrow{l}[swap]{S}
		&
		\DotC
		\arrow[dashed]{l}[swap]{D}
		\arrow{r}{S}
		&
		\DotB
		\arrow{r}{D}
		&
		\DotB
		\arrow{r}{}
		&
		\cdots
		\arrow{r}{}
		&
		\DotB
		\arrow{r}{D}
		&
		\DotB
		\arrow{r}{S}
		&
		\DotC
		&
		\DotC
		\arrow[dashed]{l}[swap]{D}
		&
		\cdots
		\arrow{l}[swap]{S}
	\end{tikzcd}
	\]
	(Note the direction of the two dashed arrows.)
	A simple application of the Clean-Up Lemma \cite[Lemma~2.17]{KWZ} shows that the mapping cone of such a complex is chain isomorphic to a complex containing a direct summand that corresponds to the curve \(\sKh_{R}\). By the classification of complexes over \(\BNAlgH\) \cite[Theorem~1.5]{KWZ}, this contradicts Lemma~\ref{lem:H_is_nullhomotopic_on_Khr}. 
	Any local system may be pushed outside of the relevant region of the type~D structure in which these isotopies are non-trivial, so this argument works in general. 
\end{proof}

\begin{figure}[t]
	\centering
	\(
	\begin{tikzcd}[column sep=30pt,row sep=30pt]
	% row 1
	&&
	\DotB
	\arrow{r}{S}
	&
	\DotC
	\arrow{d}{D}
	\\
	% row 2
	&&
	\DotB
	\arrow{u}{D}
	&
	\DotC
	\\
	% row 3
	\DotC
	\arrow[swap]{d}{D}
	&
	\DotC
	\arrow[swap]{l}{S^2}
	\arrow[dotted,swap]{d}{U^2}
	&
	\DotC
	\arrow{u}{S^3}
	\arrow[swap]{l}{D}
	\arrow[dotted]{r}{S\cdot U^2}
	\arrow[dotted,swap]{d}{U^2}
	&
	\DotB
	\arrow[swap]{u}{S}
	\arrow{r}{D}
	&
	\DotB
	\arrow{d}{S}
	\\
	% row 4
	\DotC
	\arrow[swap]{d}{S}
	&
	\DotC
	\arrow{d}{S}
	&
	\DotC
	\arrow{l}{D}
	\arrow[dotted,swap]{r}{U^2}
	&
	\DotC
	&
	\DotC
	\arrow{l}{D}
	\\
	% row 5
	\DotB
	\arrow[swap]{r}{D}
	&
	\DotB
	&
	\DotC
	\arrow{u}{S^2}
	\arrow[swap]{d}{D}
	\arrow[dotted,swap]{r}{U^2}
	&
	\DotC
	\arrow[swap]{u}{S^2}
	\arrow{d}{D}
	\\
	% row 6
	&&
	\DotC
	\arrow[swap]{d}{S^2}
	\arrow[dotted,swap]{r}{U^2}
	&
	\DotC
	\arrow{d}{S}
	\\
	% row 7
	&&
	\DotC
	\arrow[swap]{d}{D}
	&
	\DotB
	\\
	% row 8
	&&
	\DotC
	\arrow[swap]{r}{S}
	&
	\DotB
	\arrow[swap]{u}{D}
	\end{tikzcd}
	\)
\hspace*{40pt}
\labellist \tiny
	\pinlabel $D$ at 32 407 \pinlabel $S$ at 45 407 \pinlabel $D$ at 45 420 
	\pinlabel $D$ at 87 407 \pinlabel $S$ at 75 407 \pinlabel $D$ at 75 420 
	\pinlabel $D$ at 32 376 \pinlabel $S$ at 45 376 \pinlabel $D$ at 45 364
	\pinlabel $D$ at 87 376 \pinlabel $S$ at 75 376 \pinlabel $D$ at 75 364
	\pinlabel $S^3$ at 62 382
		\endlabellist
\raisebox{-170pt}{\includegraphics[scale=0.75]{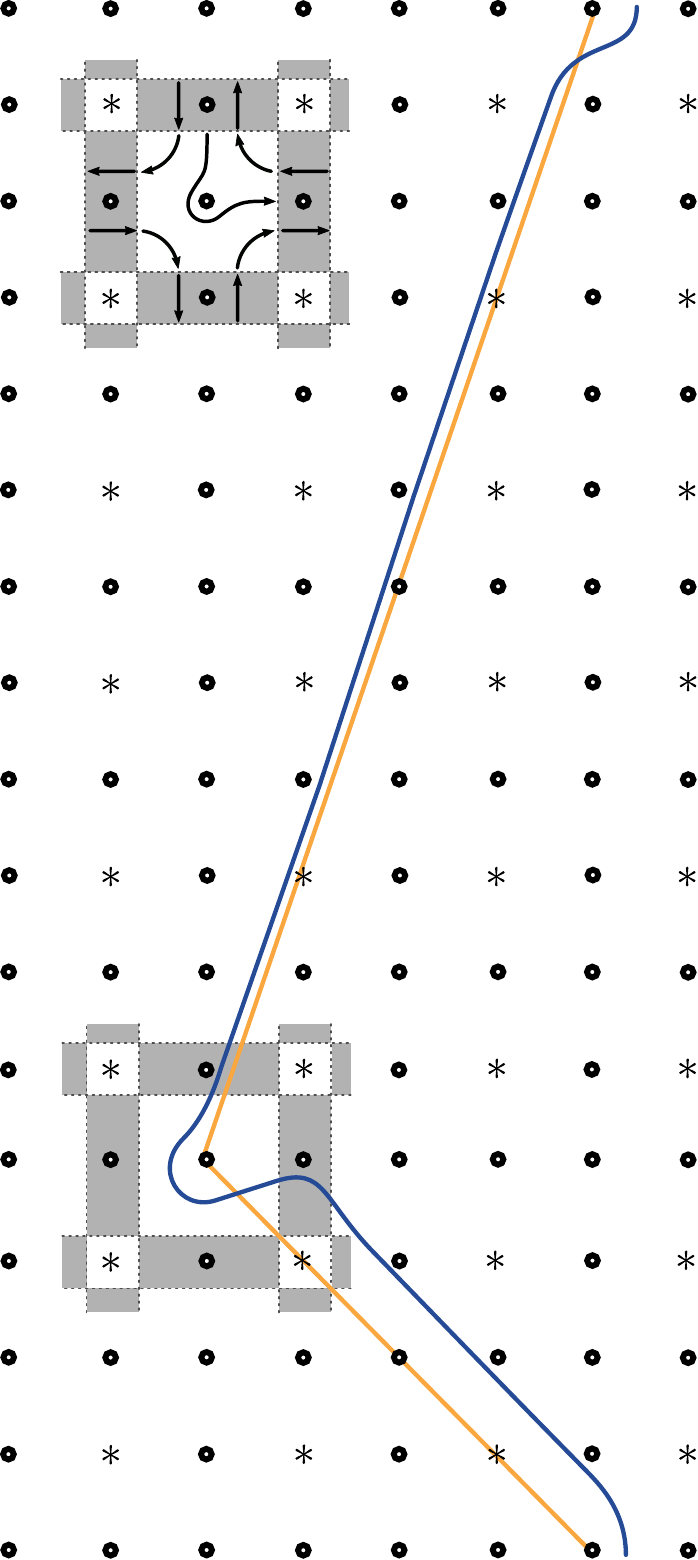}}
\caption{The extended type~D structure over \(\BinfU\) discussed in Remark~\ref{rem:comparison_extendibility} and the lift of its corresponding curve to \(\PuncturedPlane\). 
	At the top right, a shorthand for constructing curve segments in the cover associated with algebra elements arising in a type~D structure over~$\BNAlgH$ is given. Using this, one can check that the blue curve indeed describes the type~D structure on the left; on the right, the single instance of a differential labelled $S^3$ is highlighted. 	
}
\label{fig:CKMC}
\end{figure}

\begin{remark}\label{rem:comparison_extendibility}
  The above proof hinges on two properties: the existence of the extension  $\DD_1(T)^{\BNAlgH^*[U]}$ of $\DD_1(T)^{\BNAlgH}$ and the fact that $H\cdot \id_{\DD_1(T)} \simeq 0$.
	It is interesting to compare this to the proof given in~\cite{pqSym} for \(\HFT(T)\). 
	For the Heegaard Floer invariant, the extendibility property \emph{alone} suffices to show that every special component is equal to \(\sKh_{2n}(\tfrac p q)\) for some positive integer \(n\) and slope \(\tfrac p q\in\QPI\). 
	Remarkably, the same is not true for \(\Khr(T)\), as the extended type~D structure over \(\BinfU\) in Figure~\ref{fig:CKMC} illustrates.  Its restriction to a type~D structure over \(\B\) is in fact a component of \(\BNr\) of a Conway tangle, namely the tangle called \texttt{T\_CKMC} in \cite{tangle-atlas}. See Figure \ref{fig:CKMC} for this type D structure realized as a curve. 
	%, has a non-trivial self-morphism $H\cdot \id$, and corresponds to a special curve of slope $0$ which is not equal to \(\sKh_{2n}(0)\) for any integer~\(n\). 
%	It can also be bigraded which can be seen from the fact that the curve is a figure-eight curve which follows a certain immersed arc that starts and ends at the special puncture. 
%	In fact, we do not see any reason why this curve should not be a component of \(\BNr(T)\) for some tangle~\(T\). 
\end{remark}

\subsection{Geography of rational components}
\begin{theorem}\label{thm:geography:rational_curves}
  Every rational component of \(\Khr(T)\) is equal to the curve \(\rKh_{n}(\tfrac p q)\) for some $n\geq 1,~\tfrac p q \in \QPI$, equipped with the trivial local system.
\end{theorem}
\begin{proof}
As before, thanks to the naturality of \(\Khr\) under the action of the mapping class group, we may assume with out loss of generality that the slope of the rational component is 0. 
Let us denote the type~D structure associated with this component by $\DD^\rKh$.
It suffices to show that $\DD^\rKh$ is the type~D structure corresponding to the curve \(\rKh_{n}(0)\) from Figure~\ref{fig:geography} for some $n\geq 1$.

Up to homotopy, any slope $0$ rational curve can be chosen such that it does not intersect the horizontal parametrizing arc corresponding to the idempotent \(\DotC\).
Therefore, the type~D structure $\DD^\rKh$ only contains generators in the idempotent \(\DotB\) and we can consider $\DD^\rKh$ as a type~D structure over the subalgebra 
$
\F[D,S^2]/(DS^2)
\coloneqq
\iota_{\bullet} \BNAlgH \iota_{\bullet}
\subset 
\BNAlgH
$.
By Lemma~\ref{lem:H_is_nullhomotopic_on_Khr}, $H\cdot \id_{\DD^\rKh} \simeq 0$, and by Lemma~\ref{prop:higher_powers_when_wrapping}, there are no high powers of $D$ and $S^2$ in the differential of $\DD^\rKh$. 
Thus, the lemma below finishes the proof.
\end{proof}

\begin{lemma}\label{lem:cone_curves}
Let $\DD^\rKh$ be a type~D structure over the algebra $\F[D,S^2]/(DS^2)$ which is associated with an immersed curve and whose differential only contains linear combinations of \(D\) and \(S^2\). 
Suppose further that the homomorphism $H \cdot\id_{\DD^\rKh} \in \Mor(\DD^\rKh, \DD^\rKh)$ is null-homotopic. 
Then there exists some \(n\geq1\), such that $\DD^\rKh$ is equal to one or multiple copies of \(\DD^{\rKh_n}\), where 
%\begin{equation}\label{eq:rational_slope_zero_type_D_structures}
%\begin{aligned}
%&
%\DD^{\rKh_1}=\left[
%\begin{tikzcd}[column sep=1cm]
%\DotB
%\arrow[r,"D" above, bend left]
%\arrow[r,"S^2" below, bend right]
%&
%\DotB
%\end{tikzcd}
%\right], ~
%\DD^{\rKh_2}=\left[
%\begin{tikzcd}[row sep=5pt,column sep=1cm]
%&
%\DotB
%\arrow[rd,"S^2" above]
%&
%\\
%\DotB
%\arrow[ru,"D" above]
%\arrow[rd,"S^2" above]
%&
%&
%\DotB
%\\
%&
%\DotB
%\arrow[ru,"D" above]
%&
%\end{tikzcd}
%\right],~
%\DD^{\rKh_3}=
%\left[
%\begin{tikzcd}[row sep=5pt,column sep=1cm]
%&
%\DotB
%\arrow[r,"S^2" above]
%&
%\DotB
%\arrow[rd,"D" above]
%&
%\\
%\DotB
%\arrow[ru,"D" above]
%\arrow[rd,"S^2" above]
%&
%&
%&
%\DotB
%\\
%&
%\DotB
%\arrow[r,"D" above]
%&
%\DotB
%\arrow[ru,"S^2" above]
%&
%\end{tikzcd}
%\right], \\
%&
%\DD^{\rKh_4}=
%\left[
%\begin{tikzcd}[row sep=5pt,column sep=1cm]
%&
%\DotB
%\arrow[r,"S^2" above]
%&
%\DotB
%\arrow[r,"D" above]
%&
%\DotB
%\arrow[rd,"S^2" above]
%&
%\\
%\DotB
%\arrow[ru,"D" above]
%\arrow[rd,"S^2" above]
%&
%&
%&
%&
%\DotB
%\\
%&
%\DotB
%\arrow[r,"D" above]
%&
%\DotB
%\arrow[r,"S^2" above]
%&
%\DotB
%\arrow[ru,"D" above]
%&
%\end{tikzcd}
%\right] ,
%\qquad \cdots
%\end{aligned}
%\end{equation}

\[
\DD^{\rKh_1}=\left[
\begin{tikzcd}[column sep=1cm]
\DotB
\arrow[r,"D" above, bend left]
\arrow[r,"S^2" below, bend right]
&
\DotB
\end{tikzcd}
\right]
\quad
\text{and}
\quad
\DD^{\rKh_2}=\left[
\begin{tikzcd}[row sep=-5pt,column sep=1cm]
&
\DotB
\arrow[rd,"S^2"]
&
\\
\DotB
\arrow[ru,"D"]
\arrow[rd,"S^2", swap]
&
&
\DotB
\\
&
\DotB
\arrow[ru,"D", swap]
&
\end{tikzcd}
\right]
\]
and for \(n>2\),
\[
\DD^{\rKh_n}=
\left[
\begin{tikzcd}[row sep=-5pt,column sep=1cm]
&
\DotB
\arrow[start anchor=north, end anchor=north, no head, yshift=0.5em, decorate, decoration={brace}]{rr}{(n-2) \textnormal{ arrows}}
\arrow{r}{S^2}
&
\cdots
\arrow{r}{}
&
\DotB
\arrow{rd}{S^2 \textnormal{ or } D}
&
\\
\DotB
\arrow{ru}{D}
\arrow{rd}[swap]{S^2}
&
&
&
&
\DotB
\\
&
\DotB
\arrow{r}[swap]{D}
&
\cdots
\arrow{r}{}
&
\DotB
\arrow{ru}[swap]{D \textnormal{ or } S^2}
&
\end{tikzcd}
\right]
\]
\end{lemma}

\begin{proof}[Proof of Lemma~\ref{lem:cone_curves}] 
	Assume first that the local system on the curve is trivial. 
	Then we may represent the type D structure $\DD^\rKh$ as a graph with two-valent vertices $\DotB$ and edges whose labels alternate between $D$ and $S^2$. 
	In the following, we will use super- and subscripts to distinguish vertices of this graph. 
	Every edge 
	\(\!\!
	\begin{tikzcd}[column sep=12pt]
	\DotB_1 
	\arrow[r]
	& 
	\DotB_2
	\end{tikzcd}
	\!\!\)
	implies the quantum grading relation $q(\DotB_2)=q(\DotB_1)+2$. 
	Consequently, there has to be a pair of vertices that look like 
	\[
	\begin{tikzcd}[column sep=12pt]
	\phantom{}
	& 
	\DotB
	\arrow[l]
	\arrow[r]
	&
	\phantom{}
	\end{tikzcd}
	\quad
	\text{and}
	\quad
	\begin{tikzcd}[column sep=12pt]
	\phantom{}
	& 
	\DotB
	\arrow[leftarrow,l]
	\arrow[leftarrow,r]
	&
	\phantom{}
	\end{tikzcd}
	\] 
	Choose a shortest sequence of identically oriented consecutive arrows that connect two such vertices:
	\[
	\begin{tikzcd}[column sep=12pt, row sep =1.3cm]
	\DotB_0 
	\arrow[r]
	& 
	\DotB_1   
	& 
	\DotB_2  
	\arrow[l]
	& 
	\cdots  
	\arrow[l] 
	&  
	\DotB_{n} 
	\arrow[l]  
	&  
	\DotB_{n+1} 
	\arrow[l] 
	\arrow[r] 
	&  
	\DotB_{n+2}
	\end{tikzcd}
	\]
	Let us assume that the arrow 
	\begin{tikzcd}[column sep=12pt]
	\DotB_1 
	& 
	\DotB_2
	\arrow[l]
	\end{tikzcd}
	is labelled by \(S^2\). If it is labelled by \(D\), the argument is completely analogous.   
	By the minimality hypothesis, the generator $\DotB_{n+1}$ is followed by (at least) $n$ arrows pointing to the right: 
	\[
	\begin{tikzcd}[column sep=12pt, row sep =1.3cm]
	\DotB_0 
	\arrow[r]
	& 
	\DotB_1   
	& 
	\DotB_2  
	\arrow{l}[swap]{S^2}
	& 
	\cdots  
	\arrow[l] 
	&  
	\DotB_{n} 
	\arrow[l]  
	&  
	\DotB_{n+1} 
	\arrow[l] 
	\arrow[r] 
	&  
	\DotB_{n+2} 
	\arrow[r] 
	&  
	\cdots 
	\arrow[r] 
	& 
	\DotB_{2n}
	\arrow[r] 
	& 
	\DotB_{2n+1}
	\end{tikzcd}
	\]
	We claim that $\DotB_{2n}=\DotB_{0}$ and $\DotB_{2n+1}=\DotB_{1}$. 
	To show this, let $Y$ be the type D structure defined by the full subgraph consisting of the vertices \(\DotB_{i}\) for \(i=1,\dots,n\), which we relabel \(\DotB^{i}\):
	\[
	Y=
	\left[
	\begin{tikzcd}[column sep=12pt, row sep =1.3cm]
	\DotB^1   
	& 
	\DotB^2  
	\arrow{l}[swap]{S^2}
	& 
	\DotB^3 
	\arrow[l]   
	& 
	\cdots  
	\arrow[l] 
	&  
	\DotB^{n} 
	\arrow[l]
	\end{tikzcd}
	\right]
	\]
	There is an obvious inclusion map $f\co Y\hookrightarrow \DD^\rKh$ given by \(f(\DotB^i)=\DotB_i \otimes 1\). 
	$H\cdot \id_{\DD^\rKh}$ is null-homotopic, and hence so too is the morphism 
	$$
	f_H\coloneqq (H\cdot \id_{\DD^\rKh} \circ f)\in \Mor(Y,\DD^\rKh),
	\quad
	f_H (\DotB^i)=\DotB_i \otimes H
	$$
	Any null-homotopy for \(f_H\) necessarily contains components 
	\(\!\!
	\begin{tikzcd}[column sep=12pt]
	\DotB^i 
	\arrow[dashed]{r}{\id}
	& 
	\DotB_{i+1}
	\end{tikzcd}
	\!\!\)
	for \(i=1,\dots, n-1\):
	\[
	\begin{tikzcd}[column sep=20pt, row sep=25pt]
	&
	\DotB^1 
	\arrow[d,"H"]
	\arrow[rd, dashed]
	&
	\DotB^2 
	\arrow[d,"H"]
	\arrow[rd, dashed]  
	\arrow{l}[swap]{S^2}
	&
	\cdots  
	\arrow[l]
	\arrow[d,"H"]
	\arrow[rd, dashed ]
	&
	\DotB^{n} 
	\arrow[l] 
	\arrow[d,"H"]   
	\\
	\DotB_0 
	\arrow[r]
	&
	\DotB_1   
	&
	\DotB_2  
	\arrow{l}[swap]{S^2} 
	&
	\cdots  
	\arrow[l] 
	&
	\DotB_{n} 
	\arrow[l]  
	&
	\DotB_{n+1} 
	\arrow[l] 
	\arrow[r] 
	&
	\DotB_{n+2} 
	\arrow[r] 
	&
	\cdots 
	\arrow[r] 
	&
	\DotB_{2n}
	\arrow[r] 
	&
	\DotB_{2n+1}
	\end{tikzcd}
	\]
	If \(h_1\) denotes the morphism given by all dashed arrows above, the morphism $f_H+d_{\DD^r}\circ h_1 + h_1 \circ d_Y$ consists of the two solid vertical arrows below
	\[
	\begin{tikzcd}[column sep=20pt, row sep=25pt]
	&
	\DotB^1 
	\arrow[d,"D" swap] 
	\arrow[drrrrrrr, dashed,looseness=0.2, in=135,out=-45]    
	&
	\DotB^2  
	\arrow[drrrrr,dashed,looseness=0.2, in=135,out=-45]    
	\arrow{l}[swap]{S^2}
	&
	\cdots  
	\arrow[l] 
	\arrow[drrr, dashed,looseness=0.2, in=135,out=-45] 
	&
	\DotB^{n} 
	\arrow[l] 
	\arrow[d,"D\text{ or }S^2" left,pos=0.7] 
	\arrow[dr, dashed,looseness=0.2, in=135,out=-45] 
	\\
	\DotB_0 
	\arrow[r]
	&
	\DotB_1   
	&
	\DotB_2  
	\arrow{l}[swap]{S^2}  
	&
	\cdots  
	\arrow[l] 
	&
	\DotB_{n} 
	\arrow[l]  
	&
	\DotB_{n+1} 
	\arrow[l] 
	\arrow[r] 
	&
	\DotB_{n+2} 
	\arrow[r] 
	&
	\cdots 
	\arrow[r] 
	&
	\DotB_{2n} 
	\arrow[r] 
	&
	\DotB_{2n+1}
	\end{tikzcd}
	\]
	In order to eliminate the component
	\(\!\! 
	\begin{tikzcd}[column sep=25pt]
	\DotB^n
	\arrow[r,"D\text{ or }S^2"]
	& 
	\DotB_{n}
	\end{tikzcd}
	\!\!\)
	(whose label depends on the parity of~\(n\)), 
	the null-homotopy also needs to include the components 
	\(\!\!
	\begin{tikzcd}[column sep=12pt]
	\DotB^i 
	\arrow[dashed]{r}{\id}
	& 
	\DotB_{2n+1-i}
	\end{tikzcd}
	\!\!\)
	for \(i=1,\dots, n\). Let $h_2$ denote the sum of all these components with \(h_1\). The morphism $f_H+d_{\DD^r}\circ h_2 + h_2 \circ d_Y$ consist of two arrows below:
	$$
	\begin{tikzcd}[column sep=20pt, row sep=25pt]
	&
	\DotB^1 
	\arrow[d,"D" left] 
	\arrow[drrrrrrrr,"D" above,looseness=0.2, in=135,out=-45]    
	&
	\DotB^2    
	\arrow{l}[swap]{S^2}
	&
	\cdots  
	\arrow[l]  
	&
	\DotB^{n} 
	\arrow[l]   
	\\
	\DotB_0 
	\arrow[r]
	&
	\DotB_1   
	&
	\DotB_2  
	\arrow{l}[swap]{S^2}
	&
	\cdots  
	\arrow[l] 
	&
	\DotB_{n} 
	\arrow[l]  
	&
	\DotB_{n+1} 
	\arrow[l] 
	\arrow[r] 
	&
	\DotB_{n+2} 
	\arrow[r] 
	&
	\cdots 
	\arrow[r] 
	&
	\DotB_{2n} 
	\arrow{r}{D} 
	&
	\DotB_{2n+1}
	\end{tikzcd}
	$$ 
	The arrow 
	\(\!\!
	\begin{tikzcd}[column sep=12pt]
	\DotB^1 
	\arrow{r}{D}
	& 
	\DotB_{2n+1}
	\end{tikzcd}
	\!\!\)
	cannot be homotoped away, since the label to the right of $\DotB_{2n+1}$ is labelled by $S^2$. 
	Thus the arrow 
	\(\!\!
	\begin{tikzcd}[column sep=12pt]
	\DotB^1 
	\arrow{r}{D}
	& 
	\DotB_{2n+1}
	\end{tikzcd}
	\!\!\)
	should be cancelled by the other arrow 
	\(\!\!
	\begin{tikzcd}[column sep=12pt]
	\DotB^1 
	\arrow{r}{D}
	& 
	\DotB_{1}
	\end{tikzcd}
	\!\!\).
	In other words, \(h_2\) is a null-homotopy for \(f_H\) and \(\DotB_0=\DotB_{2n}\) and \(\DotB_1=\DotB_{2n+1}\). 
	Hence $\DD^\rKh=\DD^{\rKh_{n}}$. 
	
	The above argument can be easily adapted to curves with non-trivial local systems. 
	Given a local system $X \in \GL_m(\F)$, replace each vertex \(\DotB_i\) of the graph representing \(\DD^\rKh\) by \(\F^m\otimes\DotB_i\) and each arrow labelled \(a\in\{D,S^2\}\) by $\id_{\F^m}\otimes\,a$, except for the arrow 
	\(\!\!
	\begin{tikzcd}[column sep=15pt]
	\DotB_0
	\arrow{r}{}
	& 
	\DotB_1
	\end{tikzcd}
	\!\!\)
	which we replace by
	\[
	\begin{tikzcd}[column sep=30pt]
	\F^m \otimes \DotB_0
	\arrow{r}{X\otimes\,a}
	& 
	\F^m \otimes \DotB_1
	\end{tikzcd}
	\]
	With these changes, the above proof goes through: 
	In the end, two arrows 
	\[
	\begin{tikzcd}[column sep=30pt]
	\F^m \otimes \DotB^1
	\arrow{r}{\id_{\F^m}\otimes D}
	& 
	\F^m \otimes \DotB_{1}
	\end{tikzcd}
	\quad
	\text{and}
	\quad
	\begin{tikzcd}[column sep=30pt]
	\F^m \otimes \DotB^1
	\arrow{r}{X\otimes D}
	& 
	\F^m \otimes \DotB_{2n+1}
	\end{tikzcd}
	\]
	remain that need agree, so $X=\id_{\F^m}$.
\end{proof}

Every rational curve $\rKh_d(\tfrac p q)$ twines around non-special punctures in the covering space \(\PuncturedPlane\) of~$\FourPuncturedSphereKh$ (see Figure~\ref{fig:geography}).
These punctures project to two punctures in $\FourPuncturedSphereKh$; 
we say that the curve $\rKh_d(\tfrac p q)$ \emph{is based} on these two punctures. 
For example, the component $\rKh_1(\tfrac{1}{2})$ from Figure~\ref{fig:Kh:example:Curve:Downstairs} is based on $\rightPunctures$. 
In fact, it is easy to see that for any \(\tfrac{p}{q}\in\QPI\), 
\begin{align*}
\rKh_d(\tfrac p q)
\text{ is based on }
\bottomPunctures
&
\iff
\text{$p$ is even, $q$ is odd.}
\\
\rKh_d(\tfrac p q)
\text{ is based on }
\diagonalPunctures
&
\iff
\text{\(p\) and \(q\) are both odd;}
\\
\rKh_d(\tfrac p q)
\text{ is based on }
\rightPunctures
&
\iff
\text{$p$ is odd, $q$ is even;}
\end{align*}
We now prove that odd-length rational curves detect the connectivity of a tangle. 
\begin{theorem}\label{thm:connectivity_detection}
Suppose \(\Khr(T)\) contains a rational component \(\rKh_{d}(\tfrac p q)\) for some odd integer $d$. 
Then \(\rKh_{d}(\tfrac p q)\) is based on ends that are connected by the tangle \(T\). 
More explicitly,
\begin{align}
\tag{Case~1}\label{thm:connectivity_detection:case:i}
\rKh_d(\tfrac p q)
\text{ is based on }
\bottomPunctures
&
\iff
\text{$T$ has connectivity $\Lo$.}
\\
\tag{Case~2}\label{thm:connectivity_detection:case:ii}
\rKh_d(\tfrac p q)
\text{ is based on }
\diagonalPunctures
&
\iff
\text{$T$ has connectivity \(\ConnectivityX\);}
\\
\tag{Case~3}\label{thm:connectivity_detection:case:iii}
\rKh_d(\tfrac p q)
\text{ is based on }
\rightPunctures
&
\iff
\text{$T$ has connectivity $\Li$;}
\end{align}
\end{theorem}

An analogous result holds for the invariant \(\HFT\).
Interestingly, the proof of this fact for \(\HFT\) follows from a simple observation about the Alexander grading on \(\HFT\) \cite[Observation~6.1]{pqMod}. 
In contrast, the proof of Theorem~\ref{thm:connectivity_detection} will rely on the basepoint action on Khovanov homology, see Lemma~\ref{lem:Basepoint_Moving_Lemma} below. 
The only difference between the connectivity detection results for \(\HFT\) and \(\Khr\) is that all rational components of \(\HFT\) have the same length and that they may carry non-trivial local systems. 
Note, however, that at the time of writing, no tangle is known whose invariant \(\Khr\) contains a rational component of length \(\geq 3\) or whose invariant \(\HFT\) contains a rational component with non-trivial local system. 

The statement and proof of the following corollary of connectivity detection is analogous to \cite[Proposition~3.13]{LMZ}. 

\begin{corollary}\label{cor:odd_number_of_odd_length_rational_curves}
For any tangle \(T\) without closed components, the total number of odd-length rational components of \(\Khr(T)\) of the form \(\rKh_{2n+1}(\tfrac p q)\), where \(n\) and \(\tfrac{p}{q}\) may vary, is odd.
\end{corollary}
\begin{proof}[Proof of Corollary~\ref{cor:odd_number_of_odd_length_rational_curves}]
By naturality of \(\Khr\) under the action of the mapping class group, we may assume without loss of generality that the connectivity of $T$ is $\Li$. 
Suppose for contradiction that $\Khr(T)$ has an even number of curves of the form $\rKh_{2n+1}(\tfrac p q)$. 
Theorem~\ref{thm:connectivity_detection} implies that these odd-length rational components are based on the same two ends.
Furthermore, by Theorem~\ref{thm:further_geography_of_Khr}, the remaining components of $\Khr(T)$ are of the form $\sKh_{2n}(\tfrac p q)$ or $\rKh_{2n}(\tfrac p q)$.
This implies that
$$
\HF(\BNr(\Lo),\Khr(T))\cong \Khr( \Lo\cup T)
$$
is even-dimensional, contradicting the fact that reduced Khovanov homology of a knot is odd-dimensional. The latter follows from the identity $V_K(1)=1$.
\end{proof}

\pagebreak[1]%low priority pagebreak; push next line onto next page if necessary
The remainder of this section is devoted to the proof of Theorem~\ref{thm:connectivity_detection}, which we divide into two steps:
\begin{proposition}\label{prop:curve_based_detection}
Suppose \(\DD^{\rKh}\) is the type D structure associated with the curve \(\rKh_{d}(\tfrac p q)\) for some odd integer \(d\). Then: 
\begin{align}
\tag{Case~1}\label{prop:curve_based_detection:case:i}
\rKh_d(\tfrac p q)
\text{ is based on }
\bottomPunctures
&
\iff
D_{\circ}\cdot \id_{\DD^{\rKh}} 
\simeq 
0
\text{ and }
D_{\bullet} \cdot \id_{\DD^{\rKh}} 
\simeq 
S^2 \cdot  \id_{\DD^{\rKh}}
\not\simeq
0
\\
\tag{Case~2}\label{prop:curve_based_detection:case:ii}
\rKh_d(\tfrac p q)
\text{ is based on }
\diagonalPunctures
&
\iff
S^2 \cdot  \id_{\DD^{\rKh}} \simeq 0
\text{ and }
D_{\bullet}\cdot \id_{\DD^{\rKh}} 
\simeq 
D_{\circ} \cdot \id_{\DD^{\rKh}}
\not\simeq
0;
\\
\tag{Case~3}\label{prop:curve_based_detection:case:iii}
\rKh_d(\tfrac p q)
\text{ is based on }
\rightPunctures
&
\iff
D_{\bullet}\cdot \id_{\DD^{\rKh}} 
\simeq 
0
\text{ and }
S^2 \cdot \id_{\DD^{\rKh}} 
\simeq 
D_{\circ} \cdot  \id_{\DD^{\rKh}}
\not\simeq
0; 
\end{align}
\end{proposition}
\begin{proposition}\label{prop:connectivity:to_action}
	For any type~D structure \(\DD^{\rKh}\) corresponding to some component of \(\Khr(T)\),
	\begin{align}
	\tag{Case~1}\label{prop:connectivity:to_action:case:i}
	\text{$T$ has connectivity $\Lo$}
	&
	\implies
	D_{\circ}\cdot \id_{\DD^{\rKh}} 
	\simeq 
	0;
	\\
	\tag{Case~2}\label{prop:connectivity:to_action:case:ii}
	\text{$T$ has connectivity $\ConnectivityX$}
	&
	\implies
	S^2 \cdot  \id_{\DD^{\rKh}} \simeq 0;
	\\
	\tag{Case~3}\label{prop:connectivity:to_action:case:iii}
	\text{$T$ has connectivity $\Li$}
	&
	\implies
	D_{\bullet}\cdot \id_{\DD^{\rKh}} 
	\simeq 
	0.
	\end{align}
\end{proposition}

\begin{proof}[Proof of Theorem~\ref{thm:connectivity_detection}]
	Immediate from Propositions~\ref{prop:curve_based_detection} and~\ref{prop:connectivity:to_action}. 
\end{proof}

Before we start with the proof of Proposition~\ref{prop:curve_based_detection}, let us make some preliminary observations and comments. 
The mapping class group $\operatorname{Mod}(\FourPuncturedSphereKh)$ is generated by two braid moves: 
\[
\tau_1 = \tauRight
\quad\text{and}\quad
\tau_2 = \tauBottom
\] 
The following diagram describes how these two braid moves act on the left hand sides of the three cases in Proposition~\ref{prop:curve_based_detection}:
\begin{equation}\label{eqn:mcgaction:cases}
\begin{tikzcd}[column sep=40pt]
\text{(Case~1)}
\arrow[in=-170,out=170,looseness=5,swap,"\tau_2"]
\arrow[r,bend right=8,"\tau_1"]
&
\text{(Case~2)}
\arrow[l,bend right=8]
\arrow[r,bend right=8,"\tau_2"]
&
\text{(Case~3)}
\arrow[l,bend right=8]
\arrow[in=10,out=-10,looseness=5,swap,"\tau_1"]
\end{tikzcd}
\end{equation}
For example, if $\rKh_d(\tfrac p q)$ is based on \(\rightPunctures\) (Case~3), $\tau_2(\rKh_d(\tfrac p q))$ is based on \(\diagonalPunctures\) (Case~2).

For the proof of the naturality of the invariants \(\Khr\) and \(\BNr\) under the action of the mapping class group $\operatorname{Mod}(\FourPuncturedSphereKh)$ in~\cite[Section~8]{KWZ}, we constructed certain type AD bimodules ${}_\BNAlgH(\tau_i)^{\BNAlgH}$ for \(i=1,2\) that translate the geometric operations \(\tau_i\) into algebraic ones.
In particular, if a type~D structure $\DD^\rKh$ corresponds to a curve $\rKh_{d}(\tfrac p q)$, then the type~D structure that corresponds to the curve $\tau_i (\rKh_{d}(\tfrac p q))$ is homotopy equivalent to $\DD^\rKh \boxtimes {}_\BNAlgH(\tau_i)^{\BNAlgH}$.
The type AD bimodule ${}_\BNAlgH(\tau_1)^{\BNAlgH}$ is depicted in Figure~\ref{fig:twisting:bimodule:full}; ${}_\BNAlgH(\tau_2)^{\BNAlgH}$ is obtained from ${}_\BNAlgH(\tau_1)^{\BNAlgH}$ by interchanging both left and right idempotents $\DotB \leftrightarrow \DotC$. 
Naturally, there are also type AD bimodules ${}_\BNAlgH(\tau^{\pm 1}_i)^{\BNAlgH}$ that are inverse to ${}_\BNAlgH(\tau_i)^{\BNAlgH}$ up to homotopy. 
\begin{lemma}\label{lem:mcgaction:morphisms}
Given a type D structure \(\DD^\rKh\) associated with a component of \(\Khr(T)\), the bimodules \(\tau_1^{\pm 1}\) and \(\tau_2^{\pm 1}\) act on the morphisms as follows (up to homotopy):
\begin{align*}
D_{\bullet}\cdot \id_{\DD^\rKh} 
&
\xmapsto{-\boxtimes{}_\BNAlgH(\tau_1^{\pm 1})^{\BNAlgH}} 
D_{\bullet}\cdot \id_{\DD^\rKh \boxtimes{}_\BNAlgH(\tau_1^{\pm 1})^{\BNAlgH}}
&
D_{\bullet}\cdot \id_{\DD^\rKh} 
&
\xmapsto{-\boxtimes{}_\BNAlgH(\tau_2^{\pm 1})^{\BNAlgH}} 
S^2\cdot \id_{\DD^\rKh \boxtimes{}_\BNAlgH(\tau_2^{\pm 1})^{\BNAlgH}}
\\
S^2\cdot \id_{\DD^\rKh}
&
\xmapsto{-\boxtimes{}_\BNAlgH(\tau_1^{\pm 1})^{\BNAlgH}}  
D_{\circ}\cdot \id_{\DD^\rKh \boxtimes{}_\BNAlgH(\tau_1^{\pm 1})^{\BNAlgH}}
&
S^2\cdot \id_{\DD^\rKh} 
&
\xmapsto{-\boxtimes{}_\BNAlgH(\tau_2^{\pm 1})^{\BNAlgH}}
D_{\bullet}\cdot \id_{\DD^\rKh \boxtimes{}_\BNAlgH(\tau_2^{\pm 1})^{\BNAlgH}}
\\
D_{\circ}\cdot \id_{\DD^\rKh} 
&
\xmapsto{-\boxtimes{}_\BNAlgH(\tau_1^{\pm 1})^{\BNAlgH}}  
S^2\cdot \id_{\DD^\rKh \boxtimes{}_\BNAlgH(\tau_1^{\pm 1})^{\BNAlgH}}
&
D_{\circ}\cdot \id_{\DD^\rKh} 
&
\xmapsto{-\boxtimes{}_\BNAlgH(\tau_2^{\pm 1})^{\BNAlgH}} 
D_{\circ}\cdot \id_{\DD^\rKh \boxtimes{}_\BNAlgH(\tau_2^{\pm 1})^{\BNAlgH}}
\end{align*}
\end{lemma}

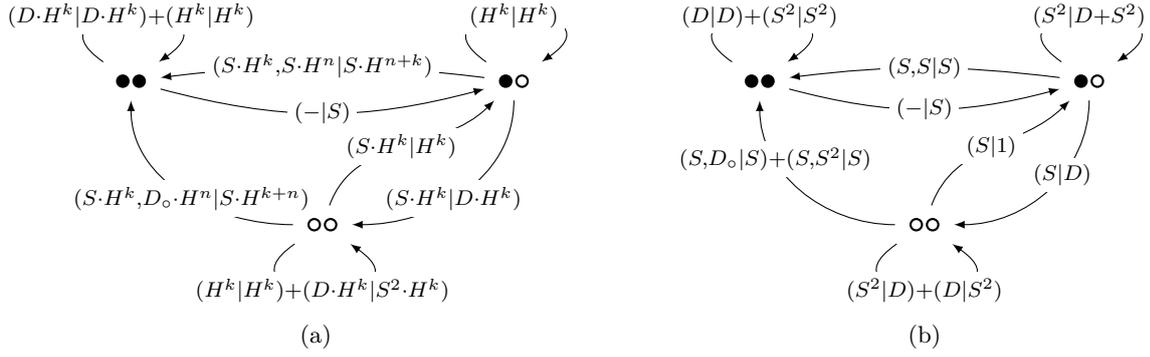
\begin{figure}[t]
	\centering
	\begin{subfigure}{0.52\textwidth}
		\centering
		\(
		\begin{tikzcd}[column sep=50pt,row sep=40pt]
		\DotB \DotB
		\arrow[bend right=15]{rr}[description]{(-\vert S)}
		\arrow[in=35,out=145,looseness=5]{rl}[description]{(D \cdot H^k \vert D \cdot H^k)+ (H^k \vert H^k)}
		&&
		\DotB \DotC
		\arrow[bend right=7]{ll}[description]{(S\cdot H^k,S \cdot H^n\vert S \cdot H^{n+k})}
		\arrow[in=0,out=-90,pos=0.55]{dl}[description]{(S \cdot H^k \vert D \cdot H^k )}
		\arrow[in=35,out=145,looseness=5]{lr}[description]{(H^k \vert H^k)}
		\\
		&
		\DotC \DotC
		\arrow[out=180,in=-90,pos=0.45]{lu}[description]{(S \cdot H^k ,D_\circ\cdot H^n \vert S\cdot H^{k+n})}
		\arrow[out=70,in=-135]{ru}[description]{(S\cdot H^k \vert H^k)}
		\arrow[in=-35,out=-145,looseness=5]{rl}[description]{(H^k \vert H^k)+(D \cdot H^k \vert S^2 \cdot H^k )}
		\end{tikzcd}
		\)
		\caption{}\label{fig:twisting:bimodule:full}
	\end{subfigure}
	\begin{subfigure}{0.47\textwidth}
		\centering
		\(
		\begin{tikzcd}[column sep=40pt,row sep=40pt]
		\DotB \DotB
		\arrow[bend right=15]{rr}[description]{(-\vert S)}
		\arrow[in=35,out=145,looseness=5]{rl}[description]{(D\vert D) + (S^2 \vert S^2)}
		&&
		\DotB \DotC
		\arrow[bend right=7]{ll}[description]{(S,S\vert S)}
		\arrow[in=0,out=-90,pos=0.4]{dl}[description]{(S \vert D )}
		\arrow[in=35,out=145,looseness=5]{lr}[description]{(S^2 \vert D+S^2)}
		\\
		&
		\DotC \DotC
		\arrow[out=180,in=-90,pos=0.7]{lu}[description]{(S ,D_\circ \vert S)+(S ,S^2 \vert S)}
		\arrow[out=70,in=-135]{ru}[description]{(S \vert 1)}
		\arrow[in=-35,out=-145,looseness=5]{rl}[description]{(S^2 \vert D)+(D \vert S^2)}
		\end{tikzcd}
		\)
		\caption{}\label{fig:twisting:bimodule:simplified}
		\end{subfigure}
	\caption{%
		(a) The bimodule \({}_{\BNAlgH}(\tau_1)^{\BNAlgH}\) and (b) the actions that are relevant for the computation of morphisms.  Note that the differential on the bimodule on the left is written with respect to the basis \(\{H^k,D\cdot H^k,S\cdot H^k\}_{k\geq0}\) of \(\BNAlgH\); the action on the right is written with respect to the generators \(\{S,D,S^2\}\). 
	}\label{fig:twisting:bimodule}
\end{figure}

\begin{proof}
We prove the statement for $\tau_1$. 
For $\tau_2$ the proof is the same. 
For $\tau_i^{-1}$ the statement follows from the fact that $-\boxtimes {}_\BNAlgH(\tau_i^{-1})^{\BNAlgH}$ is an inverse for $-\boxtimes {}_\BNAlgH(\tau_i)^{\BNAlgH}$ up to homotopy.

Figure~\ref{fig:twisting:bimodule:full} shows the bimodule \({}_\BNAlgH(\tau_1)^{\BNAlgH}\) from~\cite[Figure~44]{KWZ}. 
Such bimodules act in the obvious way on morphism spaces; see \cite[Lemma~2.3.3]{LOTBimodules}. 
For the computation of the action of \({}_\BNAlgH(\tau_1)^{\BNAlgH}\) on our morphisms, only those actions that involve the subspace of \(\BNAlgH\) spanned by \(\{S_\circ,D_\circ,S^2_\circ,S_\bullet,D_\bullet,S^2_\bullet\}\) is relevant, since these are the only input labels. 
These actions are shown in Figure~\ref{fig:twisting:bimodule:simplified}. 
So obviously
$$ (D_{\bullet}\cdot \id_{\DD^{\rKh}}) \boxtimes  {}_{\BNAlgH}(\tau_1)^{\BNAlgH} = D_{\bullet}\cdot \id_{\DD^{\rKh} \boxtimes  {}_{\BNAlgH}(\tau_1)^{\BNAlgH}} $$
The proof of the equivalence
$$ (S^2\cdot \id_{\DD^{\rKh}}) \boxtimes  {}_{\BNAlgH}(\tau_1)^{\BNAlgH} \simeq D_{\circ}\cdot \id_{\DD^{\rKh} \boxtimes  {}_{\BNAlgH}(\tau_1)^{\BNAlgH}} $$
is slightly more involved. 
First, we eliminate those components of the morphism $S^2\cdot \id_{\DD^{\rKh}}$ near each differential $\DotC\xrightarrow{S}\DotB$ in $\DD^{\rKh}$, using the dotted homotopy below:
$$\begin{tikzcd}
\cdots
\arrow[r,"D" above, leftrightarrow, dashed]
&
\DotC 
\arrow[r,"S" above]
\arrow[d,"S^2" left]
&
\DotB
\arrow[d,"S^2" left]
\arrow[dl,"S" description, dotted]
&
\cdots
\arrow[l,"D" above, leftrightarrow, dashed]
\\
\cdots
\arrow[r,"D" above, leftrightarrow, dashed]
&
\DotC 
\arrow[r,"S" above]
&
\DotB
&
\cdots
\arrow[l,"D" above, leftrightarrow, dashed]
\end{tikzcd}
\mapsto
\begin{tikzcd}
\cdots
\arrow[r,"D" above, leftrightarrow, dashed]
&
\DotC 
\arrow[r,"S" above]
&
\DotB
&
\cdots
\arrow[l,"D" above, leftrightarrow, dashed]
\\
\cdots
\arrow[r,"D" above, leftrightarrow, dashed]
&
\DotC 
\arrow[r,"S" above]
&
\DotB
&
\cdots
\arrow[l,"D" above, leftrightarrow, dashed]
\end{tikzcd}
$$
Denote the resulting morphism by $(S^2\cdot \id_{\DD^{\rKh}})'$. 
To study how ${}_{\BNAlgH}(\tau_1)^{\BNAlgH}$ acts on this morphism, let us first ignore all differentials in $\DD^{\rKh}$ except the ones labelled \(S\). 
Then $\DD^{\rKh}$ decomposes into three types of pieces, namely
\([\DotC\xrightarrow{S}\DotB]\) and the singletons \([\DotC]\) and \([\DotB]\). 
On these, ${}_{\BNAlgH}(\tau_1)^{\BNAlgH}$ acts as follows: 
$$
[\DotC\xrightarrow{S}\DotB]
\mapsto
[\DotC \xrightarrow{1} \DotC \xleftarrow{S} \DotB]
\qquad
[\DotC]
\mapsto
[\DotC]
\qquad
[\DotB]
\mapsto
[\DotB\xrightarrow{S}\DotC]
$$
The restriction of the morphism $(S^2\cdot \id_{\DD^{\rKh}})'$ to the first piece vanishes. 
On the restriction of $(S^2\cdot \id_{\DD^{\rKh}})'$ to the second and third piece, ${}_{\BNAlgH}(\tau_1)^{\BNAlgH}$ acts as follows:
$$
\left[
	\begin{tikzcd}[ampersand replacement = \&]
	\DotC
	\arrow[d,"S^2"]
	\\
	\DotC
	\end{tikzcd}
\right]
\mapsto
\left[
	\begin{tikzcd}[ampersand replacement = \&]
	\DotC
	\arrow[d,"D"]
	\\
	\DotC
	\end{tikzcd}
\right]
\qquad
\left[
	\begin{tikzcd}[ampersand replacement = \&]
	\DotB
	\arrow[d,"S^2"]
	\\
	\DotB
	\end{tikzcd}
\right]
\mapsto
\left[
	\begin{tikzcd}[ampersand replacement = \&]
	\DotB
	\arrow[r,"S"]
	\arrow[d,"S^2"]
	\&
	\DotC
	\arrow[d,"D+S^2"]
	\\
	\DotB
	\arrow[r,"S"]
	\&
	\DotC
	\end{tikzcd}
\right]
$$
Observe that all higher actions on the type~A side of the bimodule ${}_{\BNAlgH}(\tau_1)^{\BNAlgH}$ are length 2 and they involve the morphism \(S\). 
By construction, there are no components of $(S^2\cdot \id_{\DD^{\rKh}})'$ that start or end at a generator adjacent to an arrow labelled \(S\). 
Hence, the above are all non-zero components of the morphism $(S^2\cdot \id_{\DD^{\rKh}})'\boxtimes  {}_{\BNAlgH}(\tau_1)^{\BNAlgH}$. 
This morphism is homotopic to the identity multiplied by \(D_{\circ}\), which can be seen as follows: 
The component on the image of the second piece already has the desired form. 
On the image of the third piece, we would like to perform the homotopy indicated by the following dotted arrow:
$$
\left[
\begin{tikzcd}[ampersand replacement = \&]
\DotB
\arrow[r,"S"]
\arrow[d,"S^2"]
\&
\DotC
\arrow[d,"D+S^2"]
\arrow[dl,"S" description, dotted]
\\
\DotB
\arrow[r,"S"]
\&
\DotC
\end{tikzcd}
\right]
$$
This cancels the two vertical arrows labelled \(S^2\), but may contribute other differentials. 
The additional contributions come from arrows that are labelled by \(1\) or a power of \(S\) and that either end at the top right generator \(\DotC\) or start at the bottom left generator \(\DotB\). 
By inspection of Figure~\ref{fig:twisting:bimodule:simplified}, we see that the only such arrows are images of a component of the differential \([\DotB\xrightarrow{S^2}\DotB]\) in $\DD^{\rKh}$. 
In fact, these additional contributions from the homotopies cancel in pairs: 
$$
\begin{tikzcd}[ampersand replacement = \&]
\DotB
\arrow[r,"S^2" above]
\arrow[d,"S^2" right]
\&
\DotB 
\arrow[d,"S^2" right]
\\
\DotB
\arrow[r,"S^2" above]
\&
\DotB
\end{tikzcd}
\mapsto
\begin{tikzcd}[ampersand replacement = \&]
\DotB
\arrow[r,"S" below]
\arrow[rr,"S^2" above, bend left=14]
\arrow[d,"S^2" right]
\&
\DotC
\arrow[d,"D+S^2" right]
\arrow[dl,"S" description, dotted]
\arrow[rr,"S^2+D" above, bend left=14]
\&
\DotB 
\arrow[r,"S" below]
\arrow[d,"S^2" right]
\&
\DotC
\arrow[dl,"S" description, dotted]
\arrow[d,"S^2+D" right]
\\
\DotB
\arrow[rr,"S^2" below, bend right=14]
\arrow[r,"S" above]
\&
\DotC
\arrow[rr,"S^2+D" below, bend right=14]
\&
\DotB 
\arrow[r,"S" above]
\&
\DotC
\end{tikzcd}
$$
Finally, we may cancel the identity component on the images of the first piece; 
this does not contribute any additional components to the morphism. 
The result is indeed equal to the identity multiplied by \(D_{\circ}\). 

Finally, the equivalence
$$
( D_{\circ}\cdot \id_{\DD^{\rKh}} )\boxtimes  {}_{\BNAlgH}(\tau_1)^{\BNAlgH} 
\simeq
S^2\cdot \id_{\DD^{\rKh} \boxtimes  {}_{\BNAlgH}(\tau_1)^{\BNAlgH}} 
$$
follows from $(D_{\circ} + D_{\bullet} + S^2)\cdot \id_{\DD^{\rKh}} = H\cdot \id_{\DD^{\rKh}} \simeq 0$ (Lemma~\ref{lem:H_is_nullhomotopic_on_Khr}).
\end{proof}

\begin{proof}[Proof of Proposition~\ref{prop:curve_based_detection}]
Lemma~\ref{lem:mcgaction:morphisms} implies that Diagram~\eqref{eqn:mcgaction:cases} on page~\pageref{eqn:mcgaction:cases} also describes how the braid moves act on the right hand sides of the three cases in Proposition~\ref{prop:curve_based_detection}. 
So it suffices to prove the proposition for $\tfrac p q=0$; 
more explicitly, it suffices to prove that for every odd integer~$d$, the type~D structure $\DD^{\rKh_d}$ from Lemma~\ref{lem:cone_curves} satisfies 
$D_{\circ}\cdot \id_{\DD^{\rKh_d}} \simeq 0$, 
$ D_{\bullet} \cdot \id_{\DD^{\rKh_d}} \simeq S^2 \cdot  \id_{\DD^{\rKh_d}}\not\simeq 0$.

Label the generators in $\DD^{\rKh_d}$ as follows:
\[
\begin{tikzcd}[column sep=20pt]
\DotB_{d+1}
\arrow[r, "D" above]
& 
\DotB_{d+2} 
\arrow[r, "S^2" above]   
& 
\DotB_{d+3}  
\arrow[r, "D" above]  
& 
\cdots  
\arrow[r, "S^2" above] 
&  
\DotB_{2d} 
\arrow[r, "D" above]  
&  
\DotB_{1}  
&  
\DotB_{2} 
\arrow[l, "S^2" above] 
&  
\cdots 
\arrow[l, "D" above] 
& 
\DotB_{d} 
\arrow[l, "D" above] 
& 
\DotB_{d+1}
\arrow[l, "S^2" above]
\end{tikzcd}
\]
The morphism $D_{\circ}\cdot \id_{\DD^{\rKh_d}}$ is trivially zero, because there are no generators with idempotent $\DotC$. The fact that $D_{\bullet}\cdot \id_{\DD^{\rKh_d}} \simeq S^2_{\bullet}\cdot \id_{\DD^{\rKh_d}}$ follows from $(D+S^2)\cdot \id_{\DD^{\rKh_d}} = H \cdot \id_{\DD^{\rKh_d}} \simeq 0$ (Lemma~\ref{lem:H_is_nullhomotopic_on_Khr}). It remains to prove that $D_{\bullet}\cdot \id_{\DD^{\rKh_d}} \not\simeq  0$.

As in the proof of Lemma~\ref{lem:cone_curves}, consider the type D structure $Y$ and the morphism $f$: 
$$
Y=
\left[
\begin{tikzcd}[column sep=20pt]
\DotB^{1}  
&  
\DotB^{2} 
\arrow[l, "S^2" above] 
&  
\cdots 
\arrow[l, "D" above] 
& 
\DotB^{d} 
\arrow[l, "D" above] 
\end{tikzcd}
\right], 
\qquad 
f\in \Mor(Y,\DD^{\rKh_d}),~ 
f(\DotB^i)=\DotB_i \otimes 1
$$ 
Suppose for contradiction that $D_{\bullet} \cdot \id_{\DD^{\rKh_d}}$ is null-homotopic. Then so too is the morphism 
$$
f_D
\coloneqq 
(D_{\bullet} \cdot \id_{\DD^{\rKh_{d}}} \circ f)
\in
\Mor(Y,\DD^{\rKh_{d}}),
\quad
f_D (\DotB^i)=\DotB_i \otimes D
$$
which is indicated by the solid vertical arrows in the following diagram:
\[
\begin{tikzcd}[column sep=20pt, row sep=25pt]
&&&&
\DotB^{1}
\arrow[d, "D" right, near end]
\arrow[dl, dashed,looseness=0.2, in=45,out=-135]
&
\DotB^{2}
\arrow[d, "D" right, near end]
\arrow[dlll,dashed,looseness=0.2, in=45,out=-135]
\arrow[l, "S^2" above]
&
\cdots
\arrow[l, "D" above]
&
\DotB^{d-1}
\arrow[d, "D"]
\arrow[l,"S^2",swap]
\arrow[dllllll,dashed,looseness=0.2, in=45,out=-135]
&
\DotB^{d}
\arrow[d, "D"]
\arrow[l, "D",swap]
\arrow[dllllllll,dashed,looseness=0.2, in=45,out=-135]
\\
\DotB_{d+1} 
\arrow[r, "D" above]
& 
\DotB_{d+2}
\arrow[r, "S^2" above]   
& 
\cdots  
\arrow[r, "S^2" above] 
&  
\DotB_{2d} 
\arrow[r, "D" above]  
& 
\DotB_{1}
&
\DotB_{2}
\arrow[l, "S^2",swap]
&
\cdots
\arrow[l, "D",swap]
&
\DotB_{d-1}
\arrow[l,"S^2",swap]
&
\DotB_{d}
\arrow[l, "D",swap]
&
\DotB_{d+1}
\arrow[l, "S^2",swap]
\end{tikzcd}
\]
By considering the component 
\(\!\! 
\begin{tikzcd}[column sep=12pt]
\DotB^1
\arrow[r,"D"]
& 
\DotB_1
\end{tikzcd}
\!\!\)
we see that any null-homotopy for \(f_D\) contains components 
\(\!\! 
\begin{tikzcd}[column sep=12pt]
\DotB^{i}
\arrow[dashed,r,"\id"]
& 
\DotB_{2d-i+1}
\end{tikzcd}
\!\!\)
for \(i=1,\dots,d\),
ie the dashed arrows in the above diagram. Let \(h_1\) be the sum of all these components. Then $f_D + d_{\DD^{\rKh_{d}}}\circ h_1 + h_1\circ d_Y$ contains a component 
\(\!\! 
\begin{tikzcd}[column sep=12pt]
\DotB^{d}
\arrow[r,"S^2"]
& 
\DotB_{d}
\end{tikzcd}
\!\!\), which cannot be further homotoped away. 
So $f_D \not\simeq 0$, contradicting our assumption.
\end{proof}

In preparation for the proof of Proposition~\ref{prop:connectivity:to_action}, recall from Page~\pageref{eq:cobb_relations} that the chain complex $\KhTb{\Diag_T}$ over the category $\Cobb$ is obtained from a given pointed tangle diagram $\Diag_T$ via a cube of resolutions.
The identity morphism on $\KhTb{\Diag_T}$ can be written as
$$
\id_{\KhTb{\Diag_T}}
\coloneqq
\bigoplus_{\text{resolutions }\Diag\text{ of }\Diag_T} 
\big(
\Diag \times [0,1]
\big)
$$
Let $p$ be a point on a strand of the diagram $\Diag_T$ away from a crossing. 
Then \(p\) distinguishes one component of each resolution \(\Diag\) of \(\Diag_T\). 
We define an associated endomorphism
\[
D_p \cdot \id_{\KhTb{\Diag_T}} 
\in
\Mor_{\Cobb}\big(\KhTb{\Diag_T},\KhTb{\Diag_T}\big)
\]
that is obtained from \(\id_{\KhTb{\Diag_T}}\) by placing a dot on the component of \(\Diag \times [0,1]\) containing $p\times [0,1]$ for every resolution \(\Diag\) of \(\Diag_T\). 
Furthermore, there is an endomorphism 
\(H \cdot \id_{\KhTb{\Diag_T}}\)
which we define componentwise. 
The following lemma is due to Bar-Natan~\cite{BN_mutation}:
\begin{lemma}[Basepoint Moving Lemma]\label{lem:Basepoint_Moving_Lemma}
If two basepoints \(p\) and \(p'\) are separated by a single crossing like \(\CrossingLBasepoint\) or \(\CrossingRBasepoint\), then
$$D_p \cdot \id_{\KhTb{\Diag_T}} \simeq (H-D_{p'})\cdot \id_{\KhTb{\Diag_T}}$$
\end{lemma}

\begin{proof}
	Since all the maps are equal to the identity away from the crossing that separates \(p\) and \(p'\) and $\KhTb{\Diag_T}$ is natural with respect to gluing \cite[Section~5]{BarNatanKhT}, it suffices to show the Lemma for \(\Diag_T=\CrossingL\). 
	% Claudius: I do not think we need to spell this out explicitly: The argument why it suffices to consider a single crossing is the same as for the proof of invariance. 
	This goes as follows:
	Suppose \(p\) lies on the top right and \(p'\) on the bottom left. 
	In the following diagram, the morphisms \(D_p \cdot \id_{\KhTb{\Diag_T}}\) and \(D_{p'} \cdot \id_{\KhTb{\Diag_T}}\) are indicated by the vertical dashed arrows:
	%\begin{subfigure}{0.4\textwidth}
	$$
	\begin{tikzcd}[column sep=4cm,row sep=1cm]
	\Ni
	\arrow[r,"\Nil" description]
	\arrow[dashed]{d}[left,solid]{\NiDotL=D_{p'}}
	\arrow[phantom]{d}[right,solid]{\NiDotR=D_p}
	&
	\No
	\arrow[dotted]{dl}[description,solid]{\Nol}
	\arrow[dashed]{d}[left,solid]{\NoDotB=D_{p'}}
	\arrow[phantom]{d}[right,solid]{\NoDotT=D_p}
	\\
	\Ni
	\arrow[r,"\Nil" description]
	&
	\No
	\end{tikzcd}
	$$
%	\caption{}
%	\label{fig:moving_thru_crossing}
	%\end{subfigure}
	%\begin{subfigure}{0.59\textwidth}
	%$$
	%\begin{tikzcd}[column sep=2cm,row sep=2cm]
	%\KhTb{\Diag^\circ_T} \otimes_g \Ni
	%\arrow[r,"\id \otimes_g \Nil" above]
	%\arrow[d,"\id \otimes_g \NiDotL= D_{p'}" left,"\id \otimes_g \NiDotR=D_p" right]
	%&
	%\KhTb{\Diag^\circ_T} \otimes_g \No 
	%\arrow[dl,"h=\id \otimes_g \Nol" right, near end]
	%\arrow[d,"\id \otimes_g \NoDotB=D_{p'}" left,"\id \otimes_g  \NoDotT=D_p" right]
	%\\
	%\KhTb{\Diag^\circ_T} \otimes_g \Ni
	%\arrow[r,"\id \otimes_g \Nil" below]
	%&
	%\KhTb{\Diag^\circ_T} \otimes_g \No
	%\end{tikzcd}
	%$$ 
	%\caption{}
	%\label{fig:general_moving_thru_crossing}
	%\end{subfigure}
	%\caption{}
	Here, \(\Nil\) and \(\Nol\) denote saddle cobordisms and \(\NoDotT\), \(\NiDotR\), \(\NoDotB\), and \(\NiDotL\) dot cobordisms. 
	The dotted arrow is the desired homotopy, which can be checked using the relation 
	$$
	\tube=\DiscLdot\DiscR+\DiscL\DiscRdot -H\cdot \DiscL \DiscR
	$$
	If \(p\) lies on the top left and \(p'\) on the bottom right, the desired homotopy is the same. 
%	For a more general $\Diag_T$, if we cut out the crossing so that $\Diag_T= \Diag^\circ_T\cup_{(S^1,4\pt)} \CrossingL$, the behavior of $\KhTb{-}$ under gluing~\cite[Section~5]{BarNatanKhT} implies
%	$$
%	\KhTb{\Diag_T}=\KhTb{\Diag^\circ_T} \otimes_g \KhTb{\CrossingL}=\KhTb{\Diag^\circ_T} \otimes_g [\Ni \xrightarrow{\Nil}\No] = \left[\KhTb{\Diag^\circ_T} \otimes_g \Ni  \xrightarrow{\id \otimes_g\Nil} \KhTb{\Diag^\circ_T} \otimes_g \No \right]
%	$$
%	where $\otimes_g$ glues the generators along $(S^1,4\pt)$, and transforms the differential according to the Leibnitz rule. The desired homotopy $h$ now is given in Figure~\ref{fig:general_moving_thru_crossing}
\end{proof}

\begin{proof}[Proof of Proposition~\ref{prop:connectivity:to_action}]
Suppose $T$ has connectivity $\Lo$ (\ref{prop:connectivity:to_action:case:i}). 
Given a diagram $\Diag_T$, place the basepoints $p$ and $p'$ on the top-right and top-left ends respectively. 
Then there is an even number of crossings separating them.
Thus, the Basepoint Moving Lemma implies 
$$D_p \cdot \id_{\KhTb{\Diag_T}} \simeq D_{p'} \cdot \id_{\KhTb{\Diag_T}}=0$$
where the second equality is due to the relation $\planedotstar =0$. 
After delooping $\KhTb{\Diag_T}$ and recasting it as a type D structure $\DD(T)^{\BNAlgH}$, the morphism $D_p \cdot \id_{\KhTb{\Diag_T}}$ becomes the morphism $D_{\circ} \cdot \id_{\DD(T)}$, which is therefore null-homotopic as well. 
%$$ D_{\circ} \cdot \id_{\DD(T)}\simeq 0 \quad  \implies \quad D_{\circ} \cdot \id_{\DD(T)}=h\circ d_{\DD(T)} + d_{\DD(T)}\circ h$$
Since $\DD_1(T)$ is simply the mapping cone of $H\cdot \id_{\DD(T)}$, also \(D_{\circ} \cdot \id_{\DD_1(T)} \simeq 0\). 
%$$D_{\circ} \cdot \id_{\DD_1(T)} \simeq 0 \quad \text{since} \quad D_{\circ} \cdot \id_{\DD_1(T)}= (h\oplus h)\circ d_{\DD_1(T)} + d_{\DD_1(T)}\circ (h\oplus h)$$
Finally, by the same argument as in the proof of Lemma~\ref{lem:H_is_nullhomotopic_on_Khr}, we deduce 
$D_{\circ} \cdot \id_{\DD^{\rKh}}\simeq0$. 

\ref{prop:connectivity:to_action:case:iii} is analogous, except the basepoint $p$ is placed on the bottom-left end of $\Diag_T$.

In~\ref{prop:connectivity:to_action:case:ii}, the basepoint $p$ is placed on the bottom-right end of $\Diag_T$, and there is an odd number of intersections between $p$ and $p'$. This means that 
$$(D_p-H) \cdot \id_{\KhTb{\Diag_T}} \simeq -D_{p'} \cdot \id_{\KhTb{\Diag_T}}=0$$
The morphism $(D_p-H) \cdot \id_{\KhTb{\Diag_T}}$ corresponds to $(D-H) \cdot \id_{\DD(T)}=S^2\cdot \id_{\DD(T)}$. 
\end{proof}

% conventions:
% - Canadian Spelling: eg neighbourhood, center, colour, parametrize
% - write `two-dimensional' instead of `2-dimensional', `four-ended' instead of `4-ended', etc.
% - use parentheses for mirroring: m(HFT(T)), not mHFT(T). 
% - use T^* and L^* for mirroring tangles and links
% - write `two-fold branched cover', not branched double cover
% - no full stops after common abbreviations like ie and eg

\begin{small}
	\pdfbookmark[section]{Acknowledgements}{Acknowledgements}
	\noindent\textbf{Acknowledgements.}
	The authors thank William Ballinger, Adeel Khan, Yank\i\ Lekili, and Lukas Lewark for helpful conversations. 
\end{small}

\nocite{khtpp}
\newcommand*{\arxiv}[1]{\href{http://arxiv.org/abs/#1}{ArXiv:\ #1}}
\newcommand*{\arxivPreprint}[1]{\href{http://arxiv.org/abs/#1}{ArXiv preprint #1}}
\bibliographystyle{alpha}
\bibliography{main}

\newcommand{\etalchar}[1]{$^{#1}$}
\begin{thebibliography}{KLM{\etalchar{+}}21}

\bibitem[AAE{\etalchar{+}}13]{AAEKO}
Mohammed Abouzaid, Denis Auroux, Alexander~I. Efimov, Ludmil Katzarkov, and
  Dmitri~O. Orlov.
\newblock Homological mirror symmetry for punctured spheres.
\newblock {\em J. Amer. Math. Soc.}, 26(4):1051--1083, 2013.
\newblock \arxiv{1103.4322v2}.

\bibitem[Bal20]{Ballinger}
William Ballinger.
\newblock {C}oncordance invariants from the {E}(-1) spectral sequence on
  {K}hovanov homology, 2020.
\newblock \arxivPreprint{2004.10807v1}.

\bibitem[BN]{BN_mutation}
Dror Bar-Natan.
\newblock
  \url{http://drorbn.net/?title=Mutation_Invariance_of_Khovanov_Homology}.
\newblock Accessed: 2021-04-29.

\bibitem[BN05]{BarNatanKhT}
Dror Bar-Natan.
\newblock Khovanov's homology for tangles and cobordisms.
\newblock {\em Geom. Topol.}, 9:1443--1499, 2005.
\newblock \arxiv{math/0410495v2}.

\bibitem[Boc16]{Bocklandt}
Ralf Bocklandt.
\newblock Noncommutative mirror symmetry for punctured surfaces.
\newblock {\em Trans. Amer. Math. Soc.}, 368(1):429--469, 2016.
\newblock With an appendix by Mohammed Abouzaid, \arxiv{1111.3392v2}.

\bibitem[Dri04]{Drinfeld}
Vladimir Drinfeld.
\newblock D{G} quotients of {DG} categories.
\newblock {\em J. Algebra}, 272(2):643--691, 2004.
\newblock \arxiv{math/0210114v7}.

\bibitem[HKK17]{HKK}
Fabian Haiden, Ludmil Katzarkov, and Maxim Kontsevich.
\newblock Flat surfaces and stability structures.
\newblock {\em Publ. Math. Inst. Hautes \'{E}tudes Sci.}, 126:247--318, 2017.

\bibitem[HN10]{KhDetectsTwoComponentUnlink}
Matthew Hedden and Yi~Ni.
\newblock Manifolds with small {H}eegaard {F}loer ranks.
\newblock {\em Geom. Topol.}, 14(3):1479--1501, 2010.
\newblock \arxiv{0906.4771v1}.

\bibitem[HRW16]{HRW}
Jonathan Hanselman, Jacob~A. Rasmussen, and Liam Watson.
\newblock Bordered {F}loer homology for manifolds with torus boundary via
  immersed curves, 2016.
\newblock \arxivPreprint{1604.03466v2}.

\bibitem[Kel06]{Keller}
Bernhard Keller.
\newblock On differential graded categories.
\newblock In {\em International {C}ongress of {M}athematicians. {V}ol. {II}},
  pages 151--190. Eur. Math. Soc., Z\"{u}rich, 2006.
\newblock \arxiv{math/0601185v5}.

\bibitem[Kho06]{Kh_frob}
Mikhail Khovanov.
\newblock Link homology and {F}robenius extensions.
\newblock {\em Fund. Math.}, 190:179--190, 2006.
\newblock \arxiv{math/0411447v2}.

\bibitem[KLM{\etalchar{+}}21]{KLMWZ}
Artem Kotelskiy, Tye Lidman, Allison Moore, Liam Watson, and Claudius
  Zibrowius.
\newblock {Conway spheres in Heegaard Floer and Khovanov homology}, 2021.
\newblock In preparation.

\bibitem[KM10]{KhDetectsUnknot}
Peter Kronheimer and Tomasz Mrowka.
\newblock Khovanov homology is an unknot-detector.
\newblock {\em Publications mathématiques de l'IHÉS}, 113, 05 2010.
\newblock \arxiv{1005.4346v1}.

\bibitem[KR07]{KR_convolutions}
Mikhail Khovanov and Lev Rozansky.
\newblock Virtual crossings, convolutions and a categorification of the {${\rm
  SO}(2N)$} {K}auffman polynomial.
\newblock {\em J. G\"{o}kova Geom. Topol. GGT}, 1:116--214, 2007.
\newblock \arxiv{math/0701333v1}.

\bibitem[KR08a]{KR_mf_I}
Mikhail Khovanov and Lev Rozansky.
\newblock Matrix factorizations and link homology.
\newblock {\em Fund. Math.}, 199(1):1--91, 2008.
\newblock \arxiv{math/0401268v2}.

\bibitem[KR08b]{KR_mf_II}
Mikhail Khovanov and Lev Rozansky.
\newblock Matrix factorizations and link homology. {II}.
\newblock {\em Geom. Topol.}, 12(3):1387--1425, 2008.
\newblock \arxiv{math/0505056v2}.

\bibitem[KWZ19]{KWZ}
Artem Kotelskiy, Liam Watson, and Claudius Zibrowius.
\newblock {I}mmersed curves in {K}hovanov homology, 2019.
\newblock \arxivPreprint{1910.14584v2}.

\bibitem[KWZ21a]{tangle-atlas}
Artem Kotelskiy, Liam Watson, and Claudius Zibrowius.
\newblock Database of {K}hovanov tangle invariants that are used in this paper
  and that were computed with {\texttt{kht++}} \cite{khtpp}.
\newblock URL:
  \url{https://cbz20.raspberryip.com/code/khtpp/examples/ThinLinksAndConwaySpheres.html},
  2021.

\bibitem[KWZ21b]{KWZ-strong}
Artem Kotelskiy, Liam Watson, and Claudius Zibrowius.
\newblock {K}hovanov homology and strong inversions, 2021.
\newblock \arxivPreprint{2104.13592v1}.

\bibitem[KWZ21c]{KWZ-thin}
Artem Kotelskiy, Liam Watson, and Claudius Zibrowius.
\newblock {T}hin links and {C}onway spheres, 2021.
\newblock \arxivPreprint{2105.06308}.

\bibitem[LMZ20]{LMZ}
Tye Lidman, Allison~H. Moore, and Claudius Zibrowius.
\newblock {L}-space knots have no essential {C}onway spheres, 2020.
\newblock \arxivPreprint{2006.03521v2}, accepted for publication in Geom.
  Topol.

\bibitem[LO10]{Lunts_Orlov}
Valery~A. Lunts and Dmitri~O. Orlov.
\newblock Uniqueness of enhancement for triangulated categories.
\newblock {\em J. Amer. Math. Soc.}, 23(3):853--908, 2010.
\newblock \arxiv{0908.4187v5}.

\bibitem[LOT15]{LOTBimodules}
Robert Lipshitz, Peter~S. Ozsv\'{a}th, and Dylan~P. Thurston.
\newblock Bimodules in bordered {H}eegaard {F}loer homology.
\newblock {\em Geom. Topol.}, 19(2):525--724, 2015.
\newblock \arxiv{1003.0598v4}.

\bibitem[LP20]{LekPol}
Yank\i\ Lekili and Alexander Polishchuk.
\newblock Homological mirror symmetry for higher-dimensional pairs of pants.
\newblock {\em Compos. Math.}, 156(7):1310--1347, 2020.
\newblock \arxiv{1811.04264v6}.

\bibitem[Nad16]{Nadler}
David Nadler.
\newblock Wrapped microlocal sheaves on pairs of pants, 2016.
\newblock \arxiv{1604.00114v1}.

\bibitem[Orl04]{Orlov}
Dmitri~O. Orlov.
\newblock Triangulated categories of singularities and {D}-branes in
  {L}andau-{G}inzburg models.
\newblock {\em Tr. Mat. Inst. Steklova}, 246(Algebr. Geom. Metody, Svyazi i
  Prilozh.):240--262, 2004.
\newblock \arxiv{math/0302304v2}.

\bibitem[OS04]{OSHFK}
Peter~S. Ozsv\'{a}th and Zolt\'{a}n Szab\'{o}.
\newblock Holomorphic disks and knot invariants.
\newblock {\em Adv. Math.}, 186(1):58--116, 2004.
\newblock \arxiv{math/0209056v4}.

\bibitem[Ras03]{Jake}
Jacob A.~Andrew Rasmussen.
\newblock {\em {F}loer homology and knot complements}.
\newblock ProQuest LLC, Ann Arbor, MI, 2003.
\newblock Thesis (Ph.D.)--Harvard University. \arxiv{math/0306378v1}.

\bibitem[Ras10]{JakeSInvariant}
Jacob~A. Rasmussen.
\newblock Khovanov homology and the slice genus.
\newblock {\em Invent. Math.}, 182(2):419--447, 2010.
\newblock \arxiv{math/0402131v1}.

\bibitem[Ras15]{Some_diff}
Jacob~A. Rasmussen.
\newblock Some differentials on {K}hovanov-{R}ozansky homology.
\newblock {\em Geom. Topol.}, 19(6):3031--3104, 2015.
\newblock \arxiv{math/0607544v2}.

\bibitem[Shu14]{Shum_torsion}
Alexander~N. Shumakovitch.
\newblock Torsion of {K}hovanov homology.
\newblock {\em Fund. Math.}, 225(1):343--364, 2014.
\newblock \arxiv{math/0405474v2}.

\bibitem[Wan20]{Wang}
Joshua Wang.
\newblock The cosmetic crossing conjecture for split links, 2020.
\newblock To appear in \textit{Geom. Topol.} \arxivPreprint{2006.01070v1}.

\bibitem[Wat17]{Watson2017}
Liam Watson.
\newblock Khovanov homology and the symmetry group of a knot.
\newblock {\em Adv. Math.}, 313:915--946, 2017.
\newblock \arxiv{1311.1085v4}.

\bibitem[Zib19]{pqSym}
Claudius Zibrowius.
\newblock {O}n symmetries of peculiar modules; or, {\(\delta\)}-graded link
  {F}loer homology is mutation invariant, 2019.
\newblock \arxivPreprint{1909.04267v2}, accepted for publication in J. Eur.
  Math. Soc.

\bibitem[Zib20]{pqMod}
Claudius Zibrowius.
\newblock Peculiar modules for 4-ended tangles.
\newblock {\em J. Topol.}, 13(1):77--158, 2020.
\newblock \arxiv{1712.05050v3}.

\bibitem[Zib21]{khtpp}
Claudius Zibrowius.
\newblock \texttt{kht++}, a program for computing {K}hovanov invariants for
  links and tangles.
\newblock URL: \url{https://cbz20.raspberryip.com/code/khtpp/docs/}, 2021.

\end{thebibliography}
\end{document}